\documentclass{amsart}

\usepackage{a4wide, enumerate, amsmath, amsfonts, amssymb, amsthm, wasysym, graphics, graphicx, xcolor, url, hyperref, hypcap, ifthen, xargs, slashbox, stackrel, footnote, multirow}
\hypersetup{colorlinks=true, citecolor=darkblue, linkcolor=darkblue, urlcolor=darkblue}
\usepackage[all]{xy}
\usepackage{tikz}\usetikzlibrary{trees,shapes,arrows,matrix,calc}
\usepackage{pgfplots}
\graphicspath{{figures/}}
\makesavenoteenv{tabular}

\makeatletter
\def\l@section{\@tocline{1}{4pt}{0pc}{}{}}
\makeatother
\let\oldtocsection=\tocsection

\renewcommand{\tocsection}[2]{\hspace{0em}\bf\oldtocsection{#1}{#2}}

\title[Analytic combinatorics of chord and hyperchord diagrams with $k$ crossings]{Analytic combinatorics of chord and hyperchord diagrams with $k$ crossings}

\author{Vincent Pilaud}
\address{(VP) CNRS \& LIX, \'Ecole Polytechnique, Palaiseau}
\email{vincent.pilaud@lix.polytechnique.fr}
\urladdr{http://www.lix.polytechnique.fr/~pilaud/}

\author{Juanjo Ru\'e}
\address{(JR) Instituto de Ciencias Matem\'aticas (CSIC-UAM-UC3M-UCM), Madrid}
\email{juanjo.rue@icmat.es}
\urladdr{http://www-ma2.upc.edu/jrue/}

\thanks{
V.\,P.~was supported by the spanish MICINN grant MTM2011-22792, by the French ANR grant EGOS 12 JS02 002 01, and by the European Research Project ExploreMaps~(ERC~StG~208471). \\
\indent
J.\,R.~ was supported by a JAE-DOC grant from the Junta para la Ampliaci\'on de Estudios (CSIC), by the MTM2011-22851 grant (Spain) and the ICMAT Severo Ochoa Project SEV-2011-0087 (Spain).
}

\newtheorem{theorem}{Theorem}[section]

\newtheorem{proposition}[theorem]{Proposition}
\newtheorem{lemma}[theorem]{Lemma}
\newtheorem{definition}[theorem]{Definition}

\theoremstyle{definition}
\newtheorem{example}[theorem]{Example}
\newtheorem{remark}[theorem]{Remark}

\newcommand{\R}{\mathbb{R}} 
\newcommand{\N}{\mathbb{N}} 
\newcommand{\C}{\mathbb{C}} 

\newcommand{\set}[2]{\left\{ #1 \;\middle|\; #2 \right\}} 
\newcommand{\bigset}[2]{\big\{ #1 \;|\; #2 \big\}} 
\newcommand{\ssm}{\smallsetminus} 
\newcommand{\eqdef}{\mbox{\,\raisebox{0.2ex}{\scriptsize\ensuremath{\mathrm:}}\ensuremath{=}\,}} 
\newcommand{\nX}{\textbf{X}}
\newcommand{\PP}{\mathbb{P}}

\newcommand{\fref}[1]{Figure~\ref{#1}} 
\newcommand{\ie}{\textit{i.e.}~} 
\newcommand{\eg}{\textit{e.g.}~} 
\newcommand{\ordinal}{\textsuperscript{th}} 
\definecolor{darkblue}{rgb}{0,0,0.7} 
\newcommand{\darkblue}{\color{darkblue}} 
\newcommand{\defn}[1]{\emph{\darkblue #1}} 
\renewcommand{\b}[1]{{\bf #1}} 
\newcommand{\Exp}{\mathbb{E}} 
\newcommand{\Var}{\mathbb{V}\mathrm{ar}} 

\newcommand{\core}[1]{#1^\star} 
\newcommandx{\F}[1][1=D]{\mathcal{#1}} 
\newcommandx{\Fcoeff}[2][1=D, 2={n,m,k}]{\mathcal{#1}(#2)} 
\newcommandx{\gf}[3][1=D, 2=_k, 3={x,y}]{{\bf #1}{#2}\!\left( #3 \right)} 
\newcommandx{\gfo}[2][1=D, 2={x,y}]{\gf[#1][_0][#2]} 
\newcommandx{\diff}[2][1={}, 2=x]{\frac{d^{#1}}{d{#2}^{#1}}} 
\newcommand{\coeff}[2]{[#1] \, {#2}} 
\newcommand{\potentialMatchings}{\phi} 
\newcommand{\maxPotentialMatchings}[1]{\Phi(#1)} 
\newcommand{\potentialDiagrams}{\psi} 
\newcommand{\maxPotentialDiagrams}[1]{\Psi(#1)} 
\newcommand{\bS}{\mathbb{S}} 

\begin{document}

\begin{abstract}
Using methods from Analytic Combinatorics, we study the families of perfect matchings, partitions, chord diagrams, and hyperchord diagrams on a disk with a prescribed number of crossings.
For each family, we express the generating function of the configurations with exactly $k$ crossings as a rational function of the generating function of crossing-free configurations.
Using these expressions, we study the singular behavior of these generating functions and derive asymptotic results on the counting sequences of the configurations with precisely~$k$ crossings.
Limiting distributions and random generators are also studied.
\end{abstract}

\vspace*{-.7cm}

\maketitle

\vspace*{-.8cm}

\tableofcontents
\label{tableofcontents}
\vspace*{-1.1cm}


\section{Introduction}
\label{sec:introduction}


\subsection{Nearly-planar chord diagrams}
\label{subsec:nearlyPlanarChordDiagrams}

Let~$V$ be a set of~$n$ labeled points on the unit circle.
A \defn{chord diagram} on~$V$ is a set of chords between points of~$V$.
We say that two chords \defn{cross} when their relative interior intersect.
The \defn{crossing graph} of a chord diagram is the graph with a vertex for each chord and an edge between any two crossing chords.

The enumeration properties of crossing-free (or planar) chord diagrams have been largely studied in the literature, see in particular~\cite{FlajoletNoy-nonCrossing}.
A more recent trend studies chord diagrams with some but restricted crossings.
The several ways to restrict their crossings lead to various interesting notions of \defn{nearly-planar chord diagrams}.
Among others, it is interesting to study chord diagrams
\begin{enumerate}
\item with at most $k$ crossings, or \label{enum:atMost}
\item with no $(k+2)$-crossing (meaning $k+2$ pairwise crossing edges), or \label{enum:multi}
\item where each chord crosses at most~$k$ other chords, or \label{enum:degree}
\item which become crossing-free when removing at most $k$ well-chosen chords. \label{enum:cover}
\end{enumerate}
Note that these conditions are natural restrictions on the crossing graphs of the chord diagrams.
Namely, the corresponding crossing graphs have respectively \eqref{enum:atMost} at most $k$ edges, \eqref{enum:multi} no \mbox{$(k+2)$-clique}, \eqref{enum:degree} vertex degree at most~$k$, and \eqref{enum:cover} a vertex cover of size~$k$.
For~${k=0}$, all these conditions coincide and lead to crossing-free chord diagrams.
Other natural restrictions on their crossing graphs can lead to other interesting notions of nearly-planar chord diagrams.

Families of $(k+2)$-crossing-free chord diagrams have been studied in recent literature.
On the one hand, $(k+2)$-crossing-free matchings (as well as their $(k+2)$-nesting-free counterparts) were enumerated in~\cite{ChenDengDuStanleyYan}.
On the other hand, maximal $(k+2)$-crossing-free chord diagrams, also called $(k+1)$-triangulations, were introduced in~\cite{CapoyleasPach}, studied in~\cite{Nakamigawa, PilaudSantos}, and enumerated in~\cite{Jonsson, SerranoStump}, among others.
As far as we know, Conditions~\eqref{enum:atMost}, \eqref{enum:degree} and \eqref{enum:cover}, as well as other natural notions of nearly-planar chord diagrams, still remain to be studied in details.
We focus in this paper on the Analytic Combinatorics of chord configurations under Condition~\eqref{enum:atMost}.


\subsection{Rationality of generating functions}
\label{subsec:rationality}

In this paper, we study enumeration and asymptotic properties for different families of configurations: chord diagrams, hyperchord diagrams, hyperchord diagrams with restricted hyperchord sizes, perfect matchings, partitions, and partitions with restricted block sizes.
Let~$\F[C]$ denote one of these families of configurations.
For enumeration purposes, we consider the configurations of~$\F[C]$ combinatorially: in each configuration we insert a root between two consecutive vertices, and we consider two rooted configurations~$C$ and~$C'$ of~$\F[C]$ as equivalent if there is a continuous bijective automorphism of the circle which sends the root, the vertices, and the (hyper)chords of~$C$ to that of~$C'$.
We focus on three parameters of the configurations of~$\F[C]$: their number~$n$ of vertices, their number~$m$ of (hyper)chords, and their number~$k$ of crossings.
Note that for hyperchord diagrams and partitions, we count all crossings involving two chords contained in two distinct hyperchords.
Moreover, we can assume that no three chords cross at the same point, so that there is no ambiguity on whether or not we count crossings with multiplicity.
We denote by~$\Fcoeff[C]$ the set of configurations in~$\F[C]$ with $n$ vertices, $m$ (hyper)chords and $k$ crossings, and we let
\[
\gf[C][][x,y,z] \eqdef \sum_{n,m,k \in \N} |\Fcoeff[C]| \, x^n y^m z^k
\]
denote the generating function of~$\F[C]$, and
\[
\gf[C] \eqdef \sum_{n,m \in \N} |\Fcoeff[C]| \, x^n y^m = \coeff{z^k}{\gf[C][][x,y,z]}
\]
denote the generating function of the configurations in~$\F[C]$ with precisely $k$ crossings.
Our first result concerns the rationality of the latter generating function.

\begin{theorem}
\label{theo:rationality}
The generating function~$\gf[C]$ of configurations in~$\F[C]$ with exactly $k$ crossings is a rational function of the generating function~$\gfo[C]$ of planar configurations in~$\F[C]$ and of the variables~$x$ and~$y$.
\end{theorem}

The idea behind this result is to confine crossings of the configurations of~$\F[C]$ to finite subconfigurations.
Namely, we define the \defn{core configuration}~$\core{C}$ of a configuration~$C \in \F[C]$ to be the subconfiguration formed  by all (hyper)chords of~$C$ containing at least one crossing.
The key observation is that
\begin{enumerate}[(i)]
\item there are only finitely many core configurations with $k$ crossings, and
\item all configurations of~$\F[C]$ with $k$ crossings can be constructed from their core configuration inserting crossing-free subconfigurations in the remaining regions.
\end{enumerate}

This translates in the language of generating functions to a rational expression of~$\gf[C]$ in terms of~$\gfo[C]$ and its successive derivatives with respect to~$x$, which in turn are rational in~$\gfo[C]$ and the variables~$x$ and~$y$.
For certain families mentioned above, the dependence in~$y$ can even be eliminated, obtaining rational functions in~$\gfo[C]$ and~$x$ .
Similar decomposition ideas were used for example by E.~Wright in his study of graphs with fixed excess~\cite{WrightI, WrightII, WrightIII, WrightIV}, or more recently by G.~Chapuy, M.~Marcus, G.~Schaeffer in their enumeration of unicellular maps on surfaces~\cite{ChapuyMarcusSchaeffer}. See also \cite{BernardiRue}.

Note that Theorem~\ref{theo:rationality} extends a specific result of M.~B\'ona~\cite{Bona} who proved that the generating function of the partitions with $k$ crossings is a rational function of the generating function of the Catalan numbers.
We note that his method was slightly different.
The advantage of our decomposition scheme is to be sufficiently elementary and general to apply to different families of configurations such as matchings, partitions (even with restricted block sizes), chord diagrams, and hyperchord diagrams (even with restricted hyperchord sizes).
However, to illustrate the limits of our method and to point out directions for further research, we mention in our last section other possible extensions such as the case of trees (acyclic connected chord diagrams) and of diagrams on surfaces of higher genus.


\subsection{Asymptotic analysis and random generation}
\label{subsec:asymptotic}

From the rational expression of the generating function~$\gf[C]$ in terms of~$\gfo[C]$, we can extract the asymptotic behavior of configurations in~$\F[C]$ with $k$ crossings. Our asymptotic results are summarized in the following statement.

\begin{theorem}
\label{theo:asymptotic}
For~$k \ge 1$, the number of configurations in~$\F[C]$ with $k$ crossings and~$n$ vertices is
\[
\coeff{x^n}{\gf[C][_k][x,1]} \stackbin[n \to \infty]{}{=} \Lambda \, n^\alpha \, \rho^{-n} \, (1+o(1)),
\]
for certain constants~$\Lambda, \alpha, \rho \in \R$ depending on the family~$\F[C]$ and on the parameter~$k$.
\end{theorem}

The values of~$\Lambda$, $\alpha$ and~$\rho$ for different families of configurations are given in Table~\ref{table:values}.

\begin{savenotes}
\begin{table}[h]
\renewcommand{\arraystretch}{2.5}
\centerline{
\begin{tabular}{c|c|c|c|c}
family & constant~$\Lambda$ & exponent~$\alpha$ & singularity~$\rho^{-1}$ & Proposition \\
\hline
matchings\footnote{The asymptotic estimate for the number of matchings with $n$ vertices is obviously only valid when~$n$ is even.} & $\dfrac{\sqrt{2} \, (2k-3)!!}{4^{k-1} \, k! \; \Gamma\big(k-\frac{1}{2}\big)}$ & $k-\dfrac{3}{2}$ & $2$ & \ref{prop:asymptMatchings} \\[.2cm]
partitions & $\dfrac{(2k-3)!!}{2^{3k-1} \, k! \; \Gamma\big(k-\frac{1}{2}\big)}$ & $k-\dfrac{3}{2}$ & $4$ & \ref{prop:asymptPartitions} \\[.2cm]
\hline
chord diagrams & $\dfrac{\left( -2+3\sqrt{2} \right)^{3k} \! \sqrt{-140+99\sqrt{2}} \, (2k-3)!!}{2^{3k+1} \, (3-4\sqrt{2})^{k-1} \, k! \; \Gamma(k-\frac{1}{2})}$ & $k-\dfrac{3}{2}$ & $6+4\sqrt{2}$ & \ref{prop:asymptDiagrams} \\[.2cm]
hyperchord diagrams\footnote{The expression of $\rho^{-1}$ and $\Lambda$ for hyperchord diagrams is obtained from approximations of roots of polynomials, and approximate evaluations of analytic functions. Details can be found in Propositions~\ref{prop:asymptCrossingFreeHyperchordDiagrams} and~\ref{prop:asymptHyperchordDiagrams}.\vspace*{-.5cm}} & $\simeq \dfrac{1.034^{3k} \; 0.003655 \, (2k-3)!!}{0.03078^{k-1} \, k! \; \Gamma(k-\frac{1}{2})}$ & $k-\dfrac{3}{2}$ & $\simeq 64.97$ & \ref{prop:asymptHyperchordDiagrams} \\[.2cm]
\end{tabular}
}
\renewcommand{\arraystretch}{1}
\vspace{.2cm}
\caption{The values of~$\Lambda$, $\alpha$ and~$\rho$ in the asymptoptic estimate of Theorem~\ref{theo:asymptotic} for different families of chord diagrams.}
\label{table:values}
\vspace*{-.7cm}
\end{table}
\end{savenotes}
\bigskip

For partitions with restricted block sizes and for hyperchord diagrams with restricted hyperchord sizes, the values of~$\Lambda$, $\alpha$ and~$\rho$ are more involved. We refer to Propositions~\ref{prop:asymptPartitionsFixedSizes} and~\ref{prop:asymptHyperchordDiagramsFixedSizes} for precise statements.

Theorem~\ref{theo:asymptotic} and Table~\ref{table:values} already raise the following remarks:
\begin{enumerate}[(i)]
\item The position of the singularity of the generating function~$\gf[C]$ always arises from that of the corresponding planar family~$\gfo[C]$. The values of these singularities are very easy to compute for matchings and partitions, but more involved for chord and hyperchord diagrams and for partitions or diagrams with restricted block sizes.
\item Although the exponent~$\alpha$ seems to always equal~$k-\frac{3}{2}$ as in Table~\ref{table:values}, this is not true in general. This exponent is dictated by the number of core configurations in~$\F[C]$ maximizing a certain functional (see Sections~\ref{subsec:asymptoticMatchings}, \ref{subsec:partitionsFixedSizes}, \ref{subsec:asymptoticDiagrams}, and~\ref{subsec:hyperchordDiagramsFixedSizes}). Families of configurations with restricted block sizes can have different exponents, see Sections \ref{subsec:partitionsFixedSizes} and~\ref{subsec:hyperchordDiagramsFixedSizes}.
\item Although Theorem~\ref{theo:asymptotic} seems generic, the different families of configurations studied in this paper require different techniques for their asymptotic analysis. Certain methods used for the analysis are elementary, but some other are more complicated machinery borrowed from Analytic Combinatorics~\cite{FlajoletSedgewick}.
\end{enumerate}

As another relevant application of the rational expression of the generating function~$\gf[C]$ from Theorem~\ref{theo:rationality}, we obtain random generation schemes for the configurations in~$\F[C]$ with precisely $k$ crossings, using the methods developed in~\cite{DuchonFlajoletLouchardSchaeffer}.


\subsection{Overview}
\label{subsec:overview}

The paper is organized as follows (see also the table of contents on page~\pageref{tableofcontents}).
In Section~\ref{sec:matchings}, we study in full details the case of perfect matchings with $k$ crossings, since we believe that their analysis already illustrates the method and its ramifications, while remaining technically elementary. In particular, we define and study core matchings in Sections~\ref{subsec:coreMatchings} to \ref{subsec:computingCoreMatchingPolynomials}, obtain an expression of the generating function of matchings with $k$ crossings in Section~\ref{subsec:generatingFunctionMatchings}, study its asymptotic behavior in Sections~\ref{subsec:maximalCoreMatchings} and~\ref{subsec:asymptoticMatchings}, and discuss random generation of matchings with $k$ crossings in Section~\ref{subsec:randomGenerationMatchings}. We extend these results to partitions and to partitions with restricted block sizes in Sections~\ref{subsec:partitions} and~\ref{subsec:partitionsFixedSizes} respectively.

In Section~\ref{sec:diagrams}, we apply the same method to deal with chord diagrams, hyperchord diagrams and hyperchord diagrams with restricted hyperchord sizes. Although we apply a similar decomposition, the results and analysis are slightly more technical, in particular since the generating functions of crossing-free chord and hyperchord diagrams are not as simple as for matchings and partitions.

Finally, we discuss in Section~\ref{sec:limitsMethod} the situations of trees with $k$ crossings and of chord configurations on orientable surfaces with boundaries, for which the method presented in this paper reaches its limits.

Throughout this paper, we use language and basic results of \defn{Analytic Combinatorics}.
We refer to the book of P.~Flajolet and R.~Sedgewick~\cite{FlajoletSedgewick} for a detailed presentation of this area.
For the convenience of the reader, we recall in Appendix~\ref{app:methodology} the main tools used in this paper.


\section{Perfect matchings and partitions}
\label{sec:matchings}

\enlargethispage{.3cm}
In this section, we consider the family $\F[M]$ of perfect matchings with endpoints on the unit circle.
Each perfect matching~$M$ of~$\F[M]$ is \defn{rooted}: we mark (with the symbol $\triangle$) an arc of the circle between two endpoints of~$M$, or equivalently, we label the vertices of~$M$ counterclockwise starting just after the mark~$\triangle$.
Although it is equivalent to considering matchings of~$[n]$, the representation on the disk suits better for the presentation of our results.

Let~$\Fcoeff[M][n,k]$ denote the set of matchings in~$\F[M]$ with $n$ vertices and~$k$ crossings.
We denote by
\[
\gf[M][][x,z] \eqdef \sum_{n,k \in \N} |\Fcoeff[M][n,k]| \, x^n z^k
\]
the generating function of~$\F[M]$ where $x$ encodes the number of vertices and $z$ the number of crossings.
Observe that we do not encode here the number of chords since it is just half of the number of vertices.
We want to study the generating function
\[
\gf[M][_k][x] \eqdef \coeff{z^k}{\gf[M][][x,z]}
\]
of perfect matchings with exactly $k$ crossings.

\begin{example}
The generating function of crossing-free perfect matchings satisfies the functional equation
\[
\gfo[M][x] = 1 + x^2 \, \gfo[M][x]^2,
\]
leading to the expression
\[
\gfo[M][x] = \frac{1 - \sqrt{1-4x^2}}{2x^2} = \sum_{m \in \N} \frac{1}{m+1} \, \binom{2m}{m} \, x^{2m} = \sum_{m \in \N} C_m \, x^{2m},
\]
where~$C_m \eqdef \frac{1}{m+1} \, \binom{2m}{m}$ denotes the $m$\ordinal{} Catalan number.
The asymptotic behavior of the number of crossing-free perfect matchings is thus given by
\[
\coeff{x^{2m}}{\gfo[M][x]} = C_m \stackbin[m \to \infty]{}{=} \frac{1}{\Gamma\big(\frac{1}{2}\big)} \, n^{-\frac{3}{2}} \, 4^n \, (1+o(1)) = \frac{1}{\sqrt{\pi}} \, n^{-\frac{3}{2}} \, 4^n \, (1+o(1)).
\]
\end{example}

The goal of this paper is to go beyond crossing-free objects.
We thus assume from now on that~$k \ge 1$.


\subsection{Core matchings}
\label{subsec:coreMatchings}

Let~$M$ be a perfect matching with some crossings.
Our goal is to separate the contribution of the chords of~$M$ involved in crossings from that of the chords of~$M$ with no crossings.

\begin{definition}
A \defn{core matching} is a perfect matching where each chord is involved in a crossing.
It is a \defn{$k$-core matching} if it has exactly $k$ crossings.
The \defn{core}~$\core{M}$ of a perfect matching~$M$ is the submatching of~$M$ formed by all its chords involved in at least one crossing.
See~\fref{fig:matching}.
\end{definition}

\begin{figure}[h]
  \centerline{\includegraphics[scale=.75]{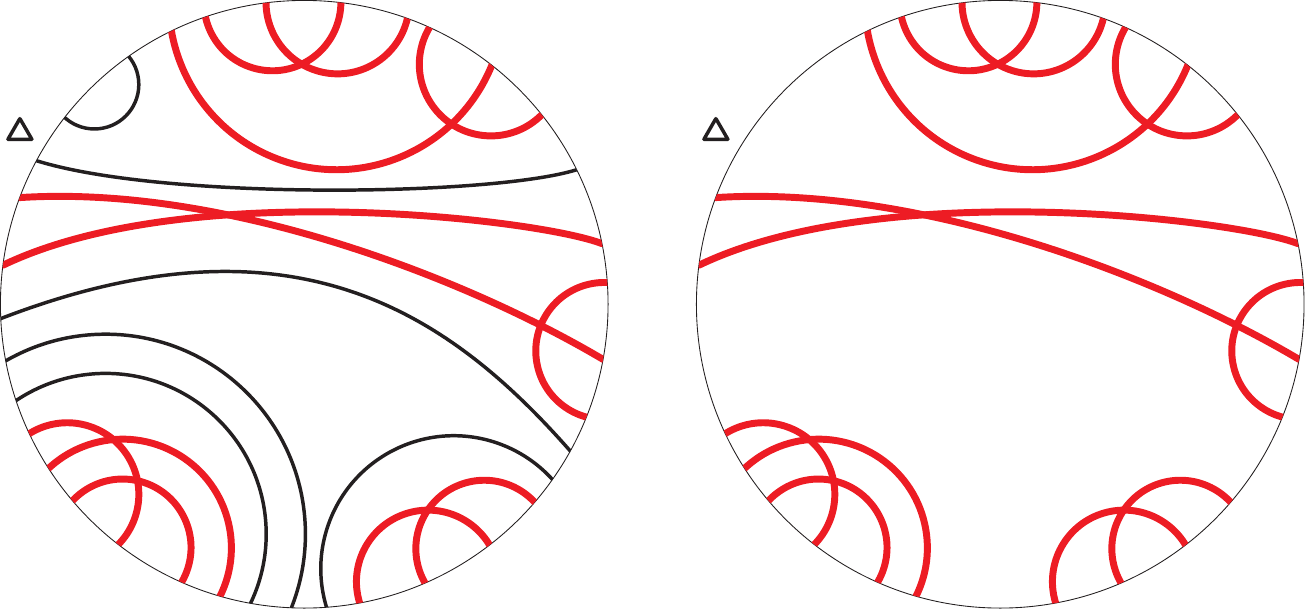}}
  \caption{A perfect matching~$M$ with $7$ crossings (left) and its $7$-core~$\core{M}$ (right). The core matching~$\core{M}$ has $n(\core{M}) = 24$ vertices and $k(\core{M}) = 7$ crossings. Moreover, $\b{n}(\core{M}) = (17,2,1)$ since it has $17$ regions with one boundary arc, $2$ with two boundary arcs, and $1$ with three boundary arcs.}
  \label{fig:matching}
\end{figure}

Let~$K$ be a core matching.
We let~$n(K)$ denote its number of vertices and $k(K)$ denote its number of crossings.
We call \defn{regions} of~$K$ the connected components of the complement of~$K$ in the unit disk.
A region has~$i$ \defn{boundary arcs} if its intersection with the unit circle has~$i$ connected arcs.
We let~$n_i(K)$ denote the number of regions of~$K$ with $i$ boundary arcs, and we set~$\b{n}(K) \eqdef (n_i(K))_{i \in [k]}$.
Note that~$n(K) = \sum_i i n_i(K)$.
See again \fref{fig:matching} for an illustration.

Since a crossing only involves $2$ chords, a $k$-core matching can have at most $2k$ chords.
This immediately implies the following crucial observation.

\begin{lemma}
There are only finitely many $k$-core matchings.
\end{lemma}

The $k$-core matchings will play a central role in the analysis of the generating function~$\gf[M][_k][x]$.
Hence, we encapsulate the enumerative information of these objects into a formal polynomial in several variables.

\begin{definition}
We encode the finite list of all possible $k$-core matchings~$K$ and their parameters~$n(K)$ and~$\b{n}(K) \eqdef (n_i(K))_{i \in [k]}$ in the \defn{$k$-core matching polynomial}
\[
\gf[KM][_k][\b{x}] \eqdef \gf[KM][_k][x_1,\dots,x_k] \eqdef \sum_{\substack{K\;k\text{-core} \\ \text{matching}}} \frac{\b{x}^{\b{n}(K)}}{n(K)} \eqdef \sum_{\substack{K\;k\text{-core} \\ \text{matching}}} \frac{1}{n(K)}\prod_{i \in [k]} {x_i}^{n_i(K)}.
\]
\end{definition}

\begin{example}
The $7$-core of \fref{fig:matching}\,(right) contributes to~$\gf[KM][_7][\b{x}]$ as the monomial $\frac{1}{24} \, {x_1}^{17} {x_2}^2 x_3$.
\end{example}

\begin{example}
\label{exm:smallCoreMatchings}
\fref{fig:smallCoreMatchings} represents all $1$-, $2$-, and $3$-core matchings, forgetting the position of their roots.
From this exhaustive enumeration, we can compute the $1$-, $2$-, and $3$-core matching polynomials:
\begin{figure}
  \centerline{\includegraphics[width=.9\textwidth]{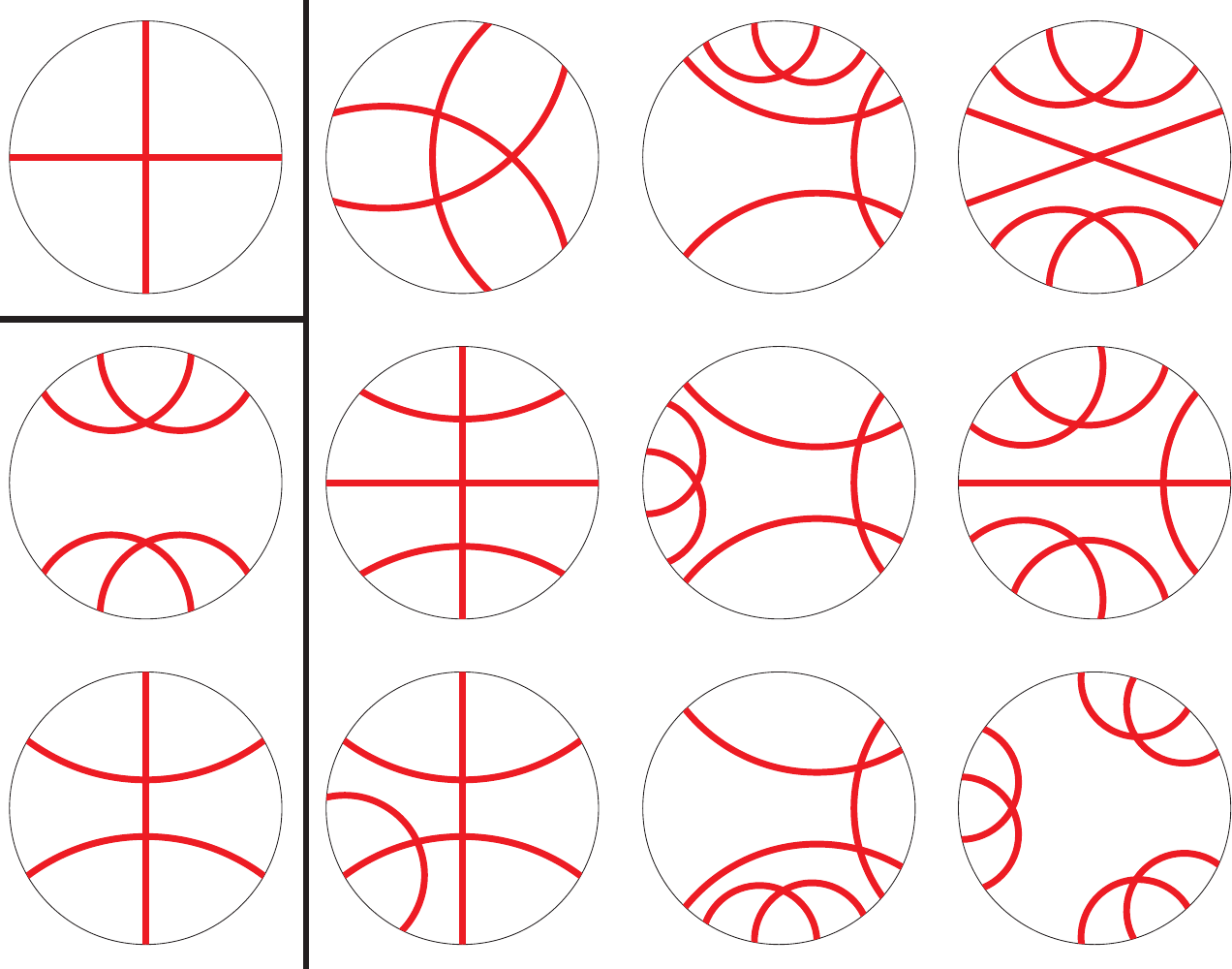}}
  \caption{The $1$-, $2$-, and $3$-core matchings (unrooted).}
  \label{fig:smallCoreMatchings}
\end{figure}
\begin{gather*}
\gf[KM][_1][\b{x}] = \frac{1}{4} \, {x_1}^4, \\
\gf[KM][_2][\b{x}] = \frac{1}{2} \, {x_1}^6 + \frac{1}{2} \, {x_1}^6 \, x_2, \\
\gf[KM][_3][\b{x}] = \frac{1}{6} \, {x_1}^6 + \frac{3}{2} \, {x_1}^8 + 3 \, {x_1}^8 \, x_2 + \frac{3}{2} \, {x_1}^8 \, {x_2}^2 + \frac{1}{3} \, {x_1}^9 \, x_3. \\
\end{gather*}
\end{example}

For later use, we also denote by
\[
\gf[KM][][\b{x}, z] \eqdef \sum_{\substack{K \text{ core} \\ \text{matching}}} \frac{\b{x}^{\b{n}(K)}z^{k(K)}}{n(K)} = \sum_{k \in \N} \gf[KM][_k][\b{x}] z^k
\]
the generating function of all core matchings.
Note that each core is weighted by the inverse of its number of vertices, both in~$\gf[KM][_k][\b{x}]$ and~$\gf[KM][][\b{x}, z]$.


\subsection{Connected matchings}
\label{subsec:connectedMatchings}

In this section, we study the class of connected matchings (see the definition below) and provide an algorithmic technique to enumerate them.

\begin{definition}
\label{def:connectedMatchings}
Call \defn{connected components} of a perfect matching the connected components of its crossing graph.
A matching is \defn{connected} if it has a unique connected component.
Let~$\gf[CM][][x,z]$ denote the generating function of connected matchings.
\end{definition}

\begin{lemma}
\label{lem:connectedMatchings}
If a perfect matching~$M$ with $k$ crossings is connected, then it coincides with its core $M = \core{M}$, and ${n_1(M) = n(M) \le 2(k+1)}$ while $n_i(M) = 0$ for all~$i > 1$.
\end{lemma}

\begin{proof}
Since the matching has only one connected component,~$M = \core{M}$ and $n_i(M) = 0$ for all~$i > 1$ (a region with at least two boundary arcs would disconnect~$M$).
Therefore~${n(M) = n_1(M)}$.
Finally, from any chord of~$M$, we can reconstruct~$M$ chord by chord, keeping a connected submatching.
Each step adds two vertices and at least one crossing, thus leading to the inequality~$n(M) \le 2(k+1)$.
\end{proof}

All perfect matchings are obtained by means of compositions of connected perfect matchings, thus leading to the following equation of generating functions:
\begin{equation}
\label{eq:inversionMatchings}
\gf[M][][x,z] = 1 + \gf[CM][][x\,{\gf[M][][x,z]},z].
\end{equation}

If we temporarily forget the parameter $z$ codifying crossings, we know that
\[
\gf[M][][x,1] = \sum_{m \ge 0} \frac{(2m)!}{2^m m!}\,x^{2m},
\]
from which we can derive
\[
\gf[CM][][x,1] = x^2 + x^4 + 4 \, x^6 + 27 \, x^8 + 248 \, x^{10} + 2830 \, x^{12} + 38232 \, x^{14} + 593859 \, x^{16} + 10401712 \, x^{18} \dots
\]
by inversion of Equation~\eqref{eq:inversionMatchings} when~$z=1$.
This sequence is indexed as A000699 in the \defn{Sloane's On-Line Encyclopedia of Integer Sequences}~\cite{OEIS}.
Note that M.~Klazar already studied the generating function~$\gf[CM][][x,1]$ in~\cite{Klazar-connectedMatchings} and proved that it is not $D$-finite (\ie the solution of a differential equation with polynomial coefficients).

Since we have no expression for $\gf[M][][x,z]$ in general, we cannot compute~$\gf[CM][][x,z]$ by the previous inversion technique.
However, the first terms of this generating function can be computed by an exhaustive enumeration algorithm explained below.

Consider a rooted connected matching~$C$.
We cut the circle at the position of the root of~$C$, see \fref{fig:orderArcs}.
We consider~$C$ as a matching on a line, and therefore we speak of the \defn{arcs} of~$C$.
The \defn{level} of an arc~$\alpha$ of~$C$ is the graph distance, in the crossing graph of~$C$, between $\alpha$ and the leftmost arc of~$C$.
We order the arcs of~$C$ first according to their level and then according to their leftmost endpoint (lexicographic order).
See \fref{fig:orderArcs}.

\begin{figure}[b]
  \centerline{\includegraphics[scale=1.35]{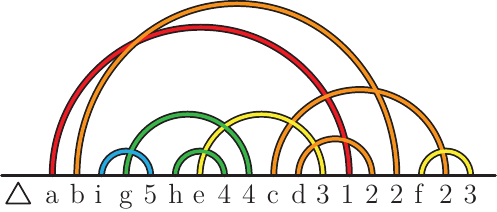}}
  \caption{A connected matching. The level of each arc appears below its right endpoint, and the order of the arcs is given by the letters below their left endpoint.}
  \label{fig:orderArcs}
\end{figure}

\enlargethispage{.2cm}
The algorithm generates all possible connected matchings, starting from a single arc and adding arcs one by one as follows.
At each step, if the last constructed arc was at level~$\ell$,~then
\begin{enumerate}
\item either we add a new arc in the current level~$\ell$. This arc should start after the leftmost endpoints of all the arcs at level~$\ell$, should cross at least one arc at level~$\ell-1$, and should not cross any arc at level~$< \ell-1$.
\item or we start a new level~$\ell+1$. The new arc should start after the leftmost arc, should cross at least one arc at level~$\ell$, and should not cross any arc at level~$< \ell$.
\end{enumerate}
The different possibilities for placing the new arc while respecting these conditions lead to different branches in the algorithm.
The computation continues until we reach matchings with $k$ crossings.

Using this algorithm, we have computed the number of connected matchings with $k$ crossings and $m$ chords for the first values of $k$~and~$m$.
See Table~\ref{table:connectedMatchings}.
The sum of all entries in each line is the number of connected matchings with $m$ chords, given by the coefficient~$\coeff{x^{2m}}{\gf[CM][][x,1]}$ above.

\begin{table}[h]
\begin{center}
\begin{tabular}{c|cccccccccccc}
\backslashbox{$m$}{$k$} 	
	  & 1 & 2 & 3  & 4  & 5   & 6    & 7     & 8     & 9      & 10      & \dots & Total      \\
\hline
2 	  & 1 &   &    &    &     &      &       &       &        &         & \dots & 1          \\
3 	  &   & 3 & 1  &    &     &      &       &       &        &         & \dots & 4          \\
4 	  &   &   & 12 & 10 & 4   & 1    &       &       &        &         & \dots & 27         \\
5 	  &   &   &    & 55 & 77  & 60   & 35    & 15    & 5      & 1       & \dots & 248        \\
6 	  &   &   &    &    & 273 & 546  & 624   & 546   & 391    & 240     & \dots & 2830       \\
7 	  &   &   &    &    &     & 1428 & 3740  & 5600  & 6405   & 6125    & \dots & 38232      \\
8 	  &   &   &    &    &     &      & 7752  & 25194 & 46512  & 65076   & \dots & 593859     \\
9 	  &   &   &    &    &     &      &       & 43263 & 168245 & 368676  & \dots & 10401712   \\
10 	  &   &   &    &    &     &      &       &       & 246675 & 1118260 & \dots & 202601898  \\
11 	  &   &   &    &    &     &      &       &       &        & 1430715 & \dots & 4342263000 \\
\hline
Total & 1 & 3 & 13 & 65 & 354 & 2035 & 12151 & 74618 & 468233 & 2989093
\end{tabular}
\end{center}
\medskip
\caption{The numbers of connected matchings with $k$ crossings and $m$ chords.}
\label{table:connectedMatchings}
\vspace{-.5cm}
\end{table}


\subsection{Computing core matching polynomials}
\label{subsec:computingCoreMatchingPolynomials}

From the algorithmic enumeration of connected matchings presented in Section~\ref{subsec:connectedMatchings}, we can now derive the first terms of the generating function~$\gf[KM][][\b{x},z]$ for core matchings, and thus the $k$-core matching polynomials~$\gf[KM][_k][\b{x}]$ for small values of~$k$. To obtain all core matchings, we decompose them into connected matchings using the following family of trees.

Consider the family~$\F[T]$ of rooted unlabeled embedded trees, where the set of leaves incident to each internal vertex~$u$ has cardinality even and at least 4, and contains the leftmost and rightmost children of~$u$.
\fref{fig:forest} shows two trees of~$\F[T]$.

\begin{figure}[h]
  \centerline{\includegraphics[scale=1.2]{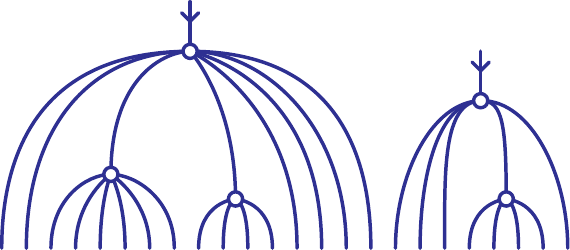}}
  \caption{Two trees~$T_1$ and~$T_2$ of the family~$\F[T]$. We have~$\b{n}(T_1) = (12,0,1,0,\dots)$ and $\b{p}(T_1) = (1,2,0,\dots)$, while $\b{n}(T_2) = (5,1,0,\dots)$ and~$\b{p}(T_2) = (2,0,\dots)$.}
  \label{fig:forest}
\end{figure}

Consider a tree~$T$ of~$\F[T]$.
For~${i \ge 1}$, we denote by~$n_i(T)$ the number of consecutive leaves of an internal vertex of~$T$ separated by $i-1$ children of this vertex.
For~${j \ge 2}$, we denote by~$p_j(T)$ the number of internal vertices of~$T$ incident to $2j$ leaves.
We set~$\b{n}(T) \eqdef (n_i(T))_{i \ge 1}$ and $\b{p}(T) \eqdef (p_j(T))_{j \ge 2}$.
Observe that
\[
1 + \sum_{i \ge 1} (i-1) \, n_i(T) = \sum_{j \ge 2} p_j(T)
\]
is the number of internal vertices of~$T$, while
\[
1 + \sum_{i \ge 1} i \, n_i(T) = 2 \sum_{j \ge 2} j \, p_j(T)
\]
is the number of leaves of~$T$.
These relations ensure that if $\b{p}(T)$ eventually vanishes, then so does~$\b{n}(T)$.
See again \fref{fig:forest}.

We consider the generating function
\[
\gf[T][][\b{x},\b{t}] \eqdef \sum_{T \in \F[T]} \b{x}^{\b{n}(T)} \b{t}^{\b{p}(T)} = \sum_{T \in \F[T]} \prod_{i \ge 1} {x_i}^{n_i} \prod_{j \ge 2} {t_j}^{p_j}.
\]
Note that this generating function depends on two infinite sets of variables~$\b{x} \eqdef (x_i)_{i \ge 1}$ and $\b{t} \eqdef (t_j)_{j \ge 2}$.
By means of the Symbolic Method~\cite{FlajoletSedgewick} applied on~$\F[T]$, this function satisfies the following implicit equation:
\begin{equation}
\label{eq:inversionTrees}
\gf[T][][\b{x},\b{t}] = \sum_{j \ge 2} t_j \bigg( \sum_{i \ge 1} x_i \, \gf[T][][\b{x},\b{t}]^{i-1} \bigg)^{2j-1}.
\end{equation}

Using Equation~\eqref{eq:inversionTrees}, we can recursively compute the first terms of the generating function~$\gf[T][][\b{x},\b{t}]$.
More precisely, let $\gf[T][_{\preccurlyeq p}][\b{x}, \b{t}] \eqdef \coeff{\b{t}^{\preccurlyeq p}}{\gf[T][][\b{x}, \b{t}]}$ denote the polynomial formed by all monomials $\alpha_{\b{n},\b{p}} \, \b{x}^{\b{n}} \b{t}^{\b{p}}$ of~$\gf[T][][\b{x}, \b{t}]$ such that $\sum_{j} (j-1) p_j \le p$.
Note that $\gf[T][_{\preccurlyeq p}][\b{x}, \b{t}]$ is indeed a polynomial: the inequality forces $\b{p}$, and therefore also~$\b{n}$, to eventually vanish.
We can compute inductively $\gf[T][_{\preccurlyeq p}][\b{x}, \b{t}]$ using Equation~\eqref{eq:inversionTrees} as follows:
\[
\gf[T][_{\preccurlyeq p+1}][\b{x}, \b{t}] =  \coeff{\b{t}^{\preccurlyeq p+1}}{\sum_{j = 2}^{p+1} \bigg( \sum_{i = 1}^{p+2-j} x_i \, \gf[T][_{\preccurlyeq p}][\b{x}, \b{t}]^{i-1} \bigg)^{2j-1} \, t_j},
\]
This enables us to compute the expressions of the first polynomials:
\begin{align*}
\gf[T][_{\preccurlyeq 1}][\b{x}, \b{t}] & = {x_1}^3 \, t_2, \\
\gf[T][_{\preccurlyeq 2}][\b{x}, \b{t}] & = \coeff{\b{t}^{\preccurlyeq 3}}{\left( \big( x_1 + {x_1}^3 \, x_2 \, t_2 \big)^3 \, t_2 + x_1^5 \, t_3 \right)} =  {x_1}^3 \, t_2 + 3 \, {x_1}^5 \, x_2 \, {t_2}^2 + {x_1}^5 \, t_3,\\
\gf[T][_{\preccurlyeq 3}][\b{x}, \b{t}] & = \coeff{\b{t}^{\preccurlyeq 4}}{\left( \big( x_1 + x_2 \, \gf[T][_{\preccurlyeq 2}][\b{x}, \b{t}] + x_3 \, \gf[T][_{\preccurlyeq 2}][\b{x}, \b{t}]^2 \big)^3 \, t_2 + \big( x_1 + x_2 \, \gf[T][_{\preccurlyeq 2}][\b{x}, \b{t}] \big)^5 \, t_3 + {x_1}^7 \, t_4 \right)} \\
& = {x_1}^3 \, t_2 + 3 \, {x_1}^5 \, x_2 \, {t_2}^2 + {x_1}^5 \, t_3 + 12 \, {x_1}^7 \, {x_2}^2 \, {t_2}^3 + 3 \, {x_1}^8 \, x_3 \, {t_2}^3 + 8 \, {x_1}^7 \, x_2 \, t_2  \, t_3 + {x_1}^7 \, t_4.
\end{align*}

\begin{figure}[b]
  \centerline{\includegraphics[scale=1.2]{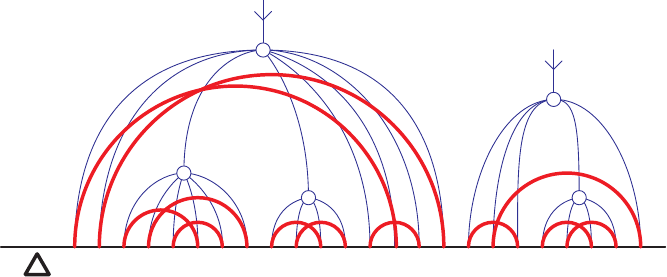}}
  \caption{The core matching of \fref{fig:matching}\,(right) corresponds to the trees of \fref{fig:forest}.}
  \label{fig:forestMatching}
\end{figure}

\medskip
We now relate these trees with both core matchings and connected matchings.
We can associate to a core matching~$K$ an ordered sequence of trees of~$\F[T]$ as illustrated in \fref{fig:forestMatching}.
Namely, we cut the circle according to the root of~$K$ and obtain a sequence of sets of nested connected matchings.
We then replace each such nested component~$N$ by a tree~$T$ of~$\F[T]$ whose structure corresponds to the nested structure of~$N$.
More precisely, the vertices of~$T$ correspond to the connected components of~$N$, the leaves of~$T$ correspond to the vertices of~$N$, and the internal arcs in~$T$ correspond to the cover relations in the nested relation of the connected components of~$N$.
Therefore, the generating function~$\gf[KM][][\b{x},z]$ can be obtained from the generating functions~$\gf[T][][\b{x},\b{t}]$ and $\gf[CM][^{2j}][z] \eqdef \coeff{x^{2j}}{\gf[CM][][x,z]}$ by
\[
\gf[KM][][\b{x},z] = \int_{s=0}^{s=1} \frac{1}{s}\sum_{i \ge 1} x_i \, s^i \, \gf[T][][x_j \leftarrow x_j s^j, t_j \leftarrow {\gf[CM][^{2j}][z]}]^i ds.
\]
Note that we integrate a bound variable~$s$ in order to quotient each monomial~$\b{x}^{\b{n}(K)} z^{k(K)}$ by its weight~$n(K) = \sum_i i n_i(K)$ which we need in the definition of~$\gf[KM][][\b{x},z]$.
From this equality, we finally derive an expression of the $k$-core matching polynomial~$\gf[KM][_k][\b{x}]$ in terms of the polynomials~$\gf[T][_{\preccurlyeq k}][\b{x}, \b{t}]$ and~$\gf[CM][^{2j}][z]$
\[
\gf[KM][_k][\b{x}] = \coeff{z^k}{\int_{s=0}^{s=1} \frac{1}{s}\sum_{i \ge 1} x_i \, s^i \, \gf[T][_{\preccurlyeq k}][x_j \leftarrow x_j s^j, t_j \leftarrow {\gf[CM][^{2j}][z]}]^i ds}.
\]
One can check the expressions of the $1$-, $2$-, and $3$-core matching polynomials in Example~\ref{exm:smallCoreMatchings}.


\subsection{Generating function of matchings with $k$ crossings}
\label{subsec:generatingFunctionMatchings}

In this section, we express the generating function~$\gf[M][_k][x]$ of matchings with $k$ crossings as a rational function of the generating function~$\gfo[M][x]$ of crossing-free matchings, using the $k$-core matching polynomial~$\gf[KM][_k][\b{x}]$ whose computation has been discussed in the previous sections.

We study perfect matchings with $k$ crossings focussing on their $k$-cores.
For this, we consider the following weaker notion of rooting of perfect matchings.
We say that a perfect matching with $k$ crossings is \defn{weakly rooted} if we have marked an arc between two consecutive vertices of its \mbox{$k$-core}.
Note that a rooted perfect matching is automatically weakly rooted (the weak root marks the arc of the $k$-core containing the root of the matching), while a weakly rooted perfect matching corresponds to several rooted perfect matchings.
To overtake this technical problem, we use the following rerooting argument.

\begin{lemma}
\label{lem:rerooting}
Let $K$ be a $k$-core with $n(K)$ vertices.
The number~$M_K(n)$ of rooted perfect matchings on $n$ vertices with core~$K$ and the number~$\bar M_K(n)$ of weakly rooted matchings on $n$ vertices with core~$K$ are related by
\[
n(K) M_K(n) = n \bar M_K(n).
\]
\end{lemma}

\begin{proof}
By double counting, erasing the roots in each family.
\end{proof}

Observe now that we can construct any perfect matching with $k$ crossings by inserting crossing-free submatchings in the regions left by its $k$-core.
From the $k$-core matching polynomial~$\gf[KM][_k][\b{x}]$, we can therefore derive the following expression of the generating function~$\gf[M][_k][x]$ of the perfect matchings with $k$ crossings.

\begin{proposition}
\label{prop:generatingFunctionMatchings}
For any~$k \ge 1$, the generating function~$\gf[M][_k][x]$ of the perfect matchings with $k$ crossings is given by
\[
\gf[M][_k][x] = x\diff\gf[KM][_k][x_i \leftarrow {\frac{x^i}{(i-1)!} \diff[i-1] \big( x^{i-1} \gfo[M][x] \big)}].
\]
In particular, $\gf[M][_k][x]$ is a rational function of~$\gfo[M][x]$ and~$x$.
\end{proposition}

\begin{proof}
Consider a rooted crossing-free perfect matching~$M$.
We say that $M$ is \defn{$i$-marked} if we have placed $i-1$ additional marks between consecutive vertices of~$M$.
Note that we can place more than one mark between two consecutive vertices.
Since we have $\binom{n+i-1}{i-1}$ possible ways to place these $(i-1)$ additional marks, the generating function of the $i$-marked crossing-free perfect matchings is given by
\[
\frac{1}{(i-1)!} \diff[i-1] \big( x^{i-1} \gfo[M][x] \big).
\]

Consider now a weakly rooted perfect matching~$M$ with $k \ge 1$ crossings.
We decompose this matching into several submatchings as follows.
On the one hand, the core~$\core{M}$ contains all crossings of~$M$.
This core is rooted by the root of~$M$.
On the other hand, each region~$R$ of~$\core{M}$ contains a (possibly empty) crossing-free submatching~$M_R$.
We root this submatching~$M_R$ as follows:
\begin{enumerate}[(i)]
\item if the root of~$M$ points out of the region~$R$, then $M_R$ is just rooted by the root of~$M$;
\item otherwise, $M_R$ is rooted on the first boundary arc of~$\core{M}$ before the root of~$M$ in clockwise direction.
\end{enumerate}
Moreover, we place additional marks on the remaining boundary arcs of the complement of~$R$ in the unit disk.
We thus obtain a rooted $i$-marked crossing-free submatching~$M_R$ in each region~$R$ of~$\core{M}$ with $i$ boundary arcs.
See \fref{fig:decompositionMatching}.
Reciprocally, we can reconstruct the weakly rooted perfect matching~$M$ from its rooted core~$\core{M}$ and its rooted $i$-marked crossing-free submatchings~$M_R$.

\begin{figure}
  \centerline{\includegraphics[scale=.5]{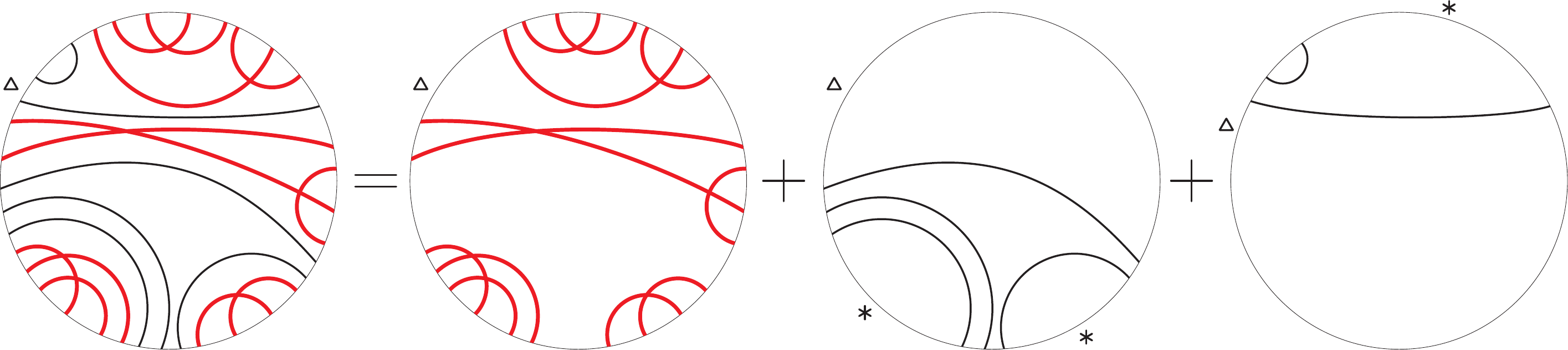}}
  \caption{The decomposition of the matching of \fref{fig:matching}\,(left) by its core into rooted marked submatchings. The root is represented by~$\vartriangle$ and the other marks are represented by~$*$. Only non-empty submatchings are  represented.}
  \label{fig:decompositionMatching}
\end{figure}

By this bijection, we thus obtain the generating function of weakly rooted perfect matchings with $k$ crossings.
From this generating function, and by application of Lemma~\ref{lem:rerooting}, we derive the generating function~$\gf[M][_k][x]$ of rooted perfect matchings with $k$ crossings:
\begin{align}
\label{eq:product-matching}
\gf[M][_k][x] & = \hspace{-5pt} \sum_{\substack{K \; k\text{-core} \\ \text{matching}}} \frac{x}{n(K)} \diff x^{n(K)} \prod_{i \ge 1} \bigg( \frac{1}{(i-1)!} \diff[i-1] \big( x^{i-1} \gfo[M][x] \big) \bigg)^{n_i(K)} \nonumber \\
& = x \diff \sum_{\substack{K \; k\text{-core} \\ \text{matching}}} \frac{1}{n(K)} \prod_{i \ge 1} \bigg( \frac{x^i}{(i-1)!} \diff[i-1] \big( x^{i-1} \gfo[M][x] \big) \bigg)^{n_i(K)} \\
& = x\diff\gf[KM][_k][x_i \leftarrow {\frac{x^i}{(i-1)!} \diff[i-1] \big( x^{i-1} \gfo[M][x] \big)}]. \nonumber
\end{align}

Since~$\gfo[M][x]$ is given by
\[
\gfo[M][x] = \frac{1-\sqrt{1-4x^2}}{2x^2}
\]
and satisfies the functional equation
\[
\gfo[M][x] = 1 + x^2 \, \gfo[M][x]^2,
\]
its derivative is rational in~$\gfo[M][x]$ and~$x$.
By induction, all its successive derivatives, and therefore $\gf[M][_k][x]$, are also rational in~$\gfo[M][x]$ and~$x$.
\end{proof}

\begin{example}
\label{ex:explicit-formula}
Using the expressions for the $k$-core polynomials given in Example~\ref{exm:smallCoreMatchings}, we can derive explicit generating functions for perfect matchings with $k$ crossings for any~$k \le 3$.
\begin{align*}
\gfo[M][x] & = \frac{1 - \sqrt{1 - 4x^2}}{2x^2} = 1 + x^2 + 2 \, x^4 + 5 \, x^6 + 14 \, x^8 + 42 \, x^{10} + 132 \, x^{12} + 429 \, x^{14} + 1430 \, x^{16} \dots ,
\\
\gf[M][_1][x] & = \frac{x^4 \, \gfo[M][x]^4}{1 - 2x^2 \, \gfo[M][x]} = \frac{\big( 1 - \sqrt{1 - 4x^2} \big)^4}{16x^4\sqrt{1 - 4x^2}} = x^4 + 6 \, x^6 + 28 \, x^8 + 120 \, x^{10} + 495 \, x^{12} + 2002 \, x^{14} \dots ,
\\[.2cm]
\gf[M][_2][x] & = \frac{x^6 \, \gfo[M][x]^6 (3 - 8x^2 \, \gfo[M][x] + 5x^4 \, \gfo[M][x]^2)}{(1 - 2x^2 \, \gfo[M][x])^3} = \frac{\big( 1 - \sqrt{1-4x^2} \big)^5 \big( 1 + 5 \sqrt{1-4x^2} \big)}{64 \, x^4 \, \sqrt{1 - 4x^2}^3} \\
& = 3 \, x^6 + 28 \, x^8 + 180 \, x^{10} + 990 \, x^{12} + 5005 \, x^{14} + 24024 \, x^{16} + 111384 \, x^{18} + 503880 \, x^{20} \dots ,
\\[.2cm]
\gf[M][_3][x] & =
\frac{- \big( 1 - \sqrt{1-4x^2} \big)^6 \big(  (1-x^2) \sqrt{1-4x^2} + 7x^2 - 26x^4 \big)}{64 x^6 \sqrt{1-4x^2}^5} \\
& = x^6 + 20 \, x^8 + 195 \, x^{10} + 1430 \, x^{12} + 9009 \, x^{14} + 51688 \, x^{16} + 278460 \, x^{18} + 1434120 \, x^{20} \dots
\end{align*}
We skipped the expression of~$\gf[M][_3][x]$ as a rational function of~$\gfo[M][x]$ since it is too long and quite meaningless.
The coefficient sequences of~$\gfo[M][x]$, $\gf[M][_1][x]$ and $\gf[M][_2][x]$ are indexed as Sequences A000108, A002694 and A074922 in the \defn{Sloane's On-Line Encyclopedia of Integer Sequences}~\cite{OEIS}.
\end{example}


\subsection{Maximal core matchings}
\label{subsec:maximalCoreMatchings}

Before establishing asymptotic formulas of the number of perfect matchings with $k$ crossings in Section~\ref{subsec:asymptoticMatchings}, we need to introduce and characterize here certain $k$-core matchings that we call \defn{maximal}.

\begin{example}
\label{exm:maximalCoreMatchings}
\fref{fig:maximalCoreMatchings} illustrates the first few examples of a family of $k$-core matchings with $n_k(K) = 1$.
Note that, except the first one, these $k$-core matchings can be rooted in four different (meaning non-equivalent) positions.

\begin{figure}[h]
  \centerline{\includegraphics[width=\textwidth]{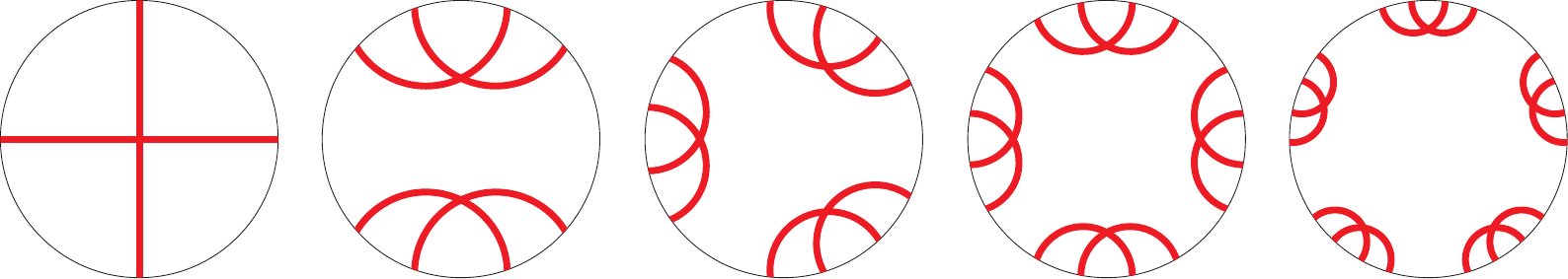}}
  \caption{Maximal core matchings (unrooted).}
  \label{fig:maximalCoreMatchings}
\end{figure}
\end{example}

\begin{lemma}
\label{lem:maximalCoreMatchings}
The following assertions are equivalent for an (unrooted) $k$-core matching~$K$:
\begin{enumerate}[(i)]
\item
\label{item:maximalCoreMatchings:fig}
$K$ is one of the $k$-core matchings presented in \fref{fig:maximalCoreMatchings}.

\item
\label{item:maximalCoreMatchings:n_k}

$n_1(K) = 3k$, $n_k(K) = 1$ and $n_i(K) = 0$ for all other values of~$i$ (here, $k \ge 2$).
\item
\label{item:maximalCoreMatchings:n_1}
$K$ maximizes $n_1(K)$ among all possible $k$-core matchings (here, $k \ge 3$).

\item
\label{item:maximalCoreMatchings:potential}
$K$ maximizes the potential function
\[
\potentialMatchings(K) \eqdef \sum_{i > 1} (2i-3) \, n_i(K)
\]
among all possible $k$-core matchings.
\end{enumerate}
We call \defn{maximal} a $k$-core matching satisfying these conditions.
\end{lemma}

\begin{proof}
Assume that~$k \ge 2$.
The implication~\eqref{item:maximalCoreMatchings:fig}$\implies$\eqref{item:maximalCoreMatchings:n_k} is immediate.
For the reverse implication, observe that if a region~$R$ of~$K$ has $k$ boundary arcs, then $K$ has at least, and thus precisely, one crossing between any two consecutive boundary arcs of~$R$.
This implies that~$K$ is one of the $k$-core matchings presented in \fref{fig:maximalCoreMatchings}.

We now prove that~\eqref{item:maximalCoreMatchings:n_k}$\iff$\eqref{item:maximalCoreMatchings:n_1} when~$k \ge 3$.
Observe on \fref{fig:smallCoreMatchings} that~$n_1(K) = 4 = 3 \, k(K)+1$ for the unique $1$-core matching~$K$, and that~$n_1(K) = 6 = 3 \, k(K)$ for any $2$-core matching~$K$.
Given any core matching~$K$ with~$k \ge 3$ crossings, we now prove by induction on the number of connected components of~$K$ that~$n_1(K) \le 3k$, with equality if and only if $K$ satisfies the conditions of~\eqref{item:maximalCoreMatchings:n_k}.
If~$K$ is connected, we have~$n_1(K) \le 2(k+1) < 3k$ according to Lemma~\ref{lem:connectedMatchings}.
Otherwise, we split the unit disk along a region of~$K$ with~$r > 2$ boundary arcs, and we obtain~$r$ core matchings~$K_1, \dots, K_r$.
See~\fref{fig:splitCoreMatching} for an example with~$r=2$.
\begin{figure}
  \centerline{\includegraphics[width=.7\textwidth]{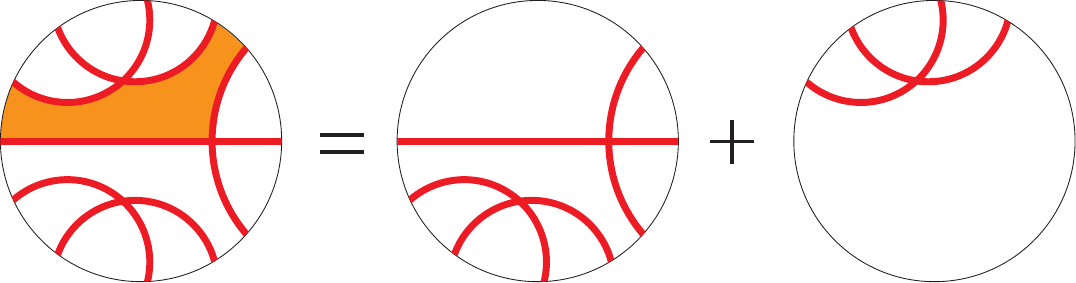}}
  \caption{Splitting a core matching~$K$ into two core matchings~$K_1, K_2$. Observe that~$k(K) = k(K_1) + k(K_2)$ and $n(K) = n(K_1) + n(K_2)$.}
  \label{fig:splitCoreMatching}
\end{figure}
Observe that
\[
k(K) = \sum_{j \in [r]} k(K_j) \qquad\text{and}\qquad n(K) = \sum_{j \in [r]} n(K_j),
\]
where the second equality can be refined to
\[
n_1(K) = \sum_{j \in [r]} \big( n_1(K_j) - 1 \big), \quad n_r(K) = 1 + \sum_{j \in [r]} n_r(K), \quad\text{and}\quad n_i(K) = \sum_{j \in [r]} n_i(K) \quad \text{for } i \notin \{1,r\}.
\]
Let~$s$ denote the number of core matchings~$K_j$ with~$k(K_j) > 1$.
For these cores~$K_j$, we have $n_1(K_j) \le 3k(K_j)$ by induction hypothesis (and by our previous observation on the special case of~$2$-core matchings).
For the other cores~$K_j$, with~$k(K_j) = 1$, we have~$n_1(K_j) = 4 = 3 \, k(K_j) + 1$ as observed earlier.
Therefore, we obtain
\[
n_1(K) = \sum_{j \in [r]} \big( n_1(K_j) - 1 \big) \le \bigg(\sum_{j \in [r]} 3k(K_j) \bigg) - s = 3k(K) - s \le 3k(K),
\]
with equality if and only if~$s=0$.
The latter condition is clearly equivalent to~\eqref{item:maximalCoreMatchings:n_k}.

Using a similar method, we finally prove that~\eqref{item:maximalCoreMatchings:n_k}$\iff$\eqref{item:maximalCoreMatchings:potential} when~$k \ge 2$.
Namely, given a core matching~$K$ with~$k \ge 2$ crossings, we prove by induction on the number of connected components of~$K$ that~$\potentialMatchings(K) \le 2k-3$, with equality if and only if $K$ satisfies the conditions of~\eqref{item:maximalCoreMatchings:n_k}.
If~$K$ is connected, then $n_i(K)=0$ for all $i>1$, and $\potentialMatchings(K) = 0 < 2k-3$.
Otherwise, we split the unit disk along a region of~$K$ with~$r > 2$ boundary arcs, and we obtain~$r$ core matchings~$K_1, \dots, K_r$.
Let~$s$ denote the number of core matchings~$K_j$ with~$k(K_j) > 1$.
Up to relabeling, we can assume that $K_1, K_2, \dots, K_s$ are the cores~$K_j$ with more than $1$~crossing.
By induction hypothesis, we have for all~$j \in [s]$,
\[
\sum_{i > 1} (2i-3) \, n_i(K_j) \leq 2 k(K_j) - 3,
\]
and therefore
\[
\sum_{j \in [s]} \sum_{i > 1} (2i-3) \, n_i(K_j) \leq 2 \sum_{j \in [s]} k(K_j) - 3s.
\]
For the core matching~$K$, we therefore obtain
\begin{align*}
\potentialMatchings(K) & = \sum_{i > 1} (2i-3) \, n_i(K) = (2r-3) + \sum_{j \in [s]} \sum_{i > 1} (2i-3) \, n_i(K_j) \\ & \leq (2r-3) + 2 \sum_{j \in [s]} k(K_j) - 3s = 2 \bigg(r-s+ \sum_{j \in [s]}k(K_j) \bigg) - 3 - s \\ & = 2 \, k(K) - 3 - s \le 2 \, k(K) - 3,
\end{align*}
with equality if and only if~$s=0$, \ie if and only if~$K$ satisfies the conditions of~\eqref{item:maximalCoreMatchings:n_k}
\end{proof}


\subsection{Asymptotic analysis}
\label{subsec:asymptoticMatchings}

We now describe the asymptotic behavior of the number of perfect matchings with~$k \ge 1$ crossings.
We start with the asymptotics of perfect matchings with $k \le 3$ crossings, which can be worked out from the explicit expressions obtained in Example~\ref{ex:explicit-formula}.

\begin{example}
\label{exm:asymptoticMatchings}
Setting~$X_+ \eqdef \sqrt{1-2x}$ and~$X_- \eqdef \sqrt{1+2x}$, we rewrite the expression of the generating function~$\gf[M][_1][x]$ obtained in Example~\ref{ex:explicit-formula} as
\[
\gf[M][_1][x] = \frac{(1-X_+ \, X_-)^4}{16x^4 \, X_+ \, X_-}.
\]
Direct expansions around the singularities~$x = \pm \frac{1}{2}$ of~$X_+$ and~$X_-$ give
\[
\gf[M][_1][x] \stackbin[x \sim \frac{1}{2}]{}{=} \frac{1}{\sqrt{2}} \, {X_+}^{-1}+O\left(1\right)
\qquad \text{and} \qquad
\gf[M][_1][x] \stackbin[x \sim -\frac{1}{2}]{}{=} \frac{1}{\sqrt{2}} \, {X_-}^{-1}+O\left(1\right).
\]
Applying the Transfer Theorem for singularity analysis (see Theorem~\ref{theo:transfer} in Appendix~\ref{app:methodology}), we obtain:
\[
\coeff{x^n}{\gf[M][_1][x]} \stackbin[n \to \infty]{}{=} \frac{1}{\sqrt{2} \, \Gamma\big(\frac{1}{2}\big)} \, n^{-\frac{1}{2}} \, ( 2^n + (-2)^n ) (1+o(1)),
\]
Writing this expression for $n = 2m$, we get the final estimate
\[
\coeff{x^{2m}}{\gf[M][_1][x]} \stackbin[m \to \infty]{}{=} \frac{1}{\Gamma\big(\frac{1}{2}\big)} \, m^{-\frac{1}{2}} \, 4^m \, (1+o(1)).
\]

A similar analysis leads to the expressions for the asymptotics of the number of matchings with~$2$ and~$3$ crossings:
\[
\coeff{x^{2m}}{\gf[M][_2][x]} \stackbin[m \to \infty]{}{=} \frac{1}{4 \, \Gamma\big(\frac{3}{2}\big)} \, m^{\frac{1}{2}} \, 4^m \, (1+o(1))
\qquad \text{and} \qquad
\coeff{x^{2m}}{\gf[M][_3][x]} \stackbin[m \to \infty]{}{=} \frac{1}{8 \, \Gamma\big(\frac{5}{2}\big)} \, m^{\frac{3}{2}} \, 4^m \, (1+o(1)).
\]
\end{example}

The analysis is more involved for general values of~$k$.
The method consists in studying the asymptotic behavior of~$\gfo[M][x]$ and of all its derivatives around their minimal singularities, and to exploit the rational expression of~$\gf[M][_k][x]$ in terms of~$\gfo[M][x]$ and~$x$ given in Proposition~\ref{prop:generatingFunctionMatchings}.
Along the way, we naturally study which $k$-cores have the main asymptotic contributions.
In fact, the potential function studied in Section~\ref{subsec:maximalCoreMatchings} will naturally show up in the analysis, and the main contribution to the number of perfect matchings with~$k$ crossings and~$n$ vertices will asymptotically arise from the maximal $k$-core matchings (observe that in the special case~$k=1$, the unique $1$-core is maximal).
We obtain the following asymptotic estimates.

\begin{proposition}
\label{prop:asymptMatchings}
For any~$k \ge 1$, the number of perfect matchings with~$k$ crossings and~$n = 2m$ vertices is
\[
\coeff{x^{2m}}{\gf[M][_k][x]} \stackbin[m \to \infty]{}{=} \frac{(2k-3)!!}{2^{k-1} \, k! \; \Gamma\big(k-\frac{1}{2}\big)} \, m^{k-\frac{3}{2}} \, 4^m \, (1+o(1)).
\]
where $(2k-3)!! \eqdef (2k-3) \cdot (2k-5) \cdots 3 \cdot 1$.
\end{proposition}

\begin{proof}
The result follows from Example~\ref{exm:asymptoticMatchings} when~$k = 1$.
Note that the result also matches that of Example~\ref{exm:asymptoticMatchings} when~$k = 2$ or~$3$ since
\[
\frac{(1)!!}{2^1 \, 2! \, \Gamma\big(\frac{3}{2}\big)} =\frac{1}{4 \, \Gamma\big(\frac{3}{2}\big)}
\qquad\text{and}\qquad
\frac{(3)!!}{2^2 \, 3! \, \Gamma\big(\frac{5}{2}\big)} =\frac{1}{8 \, \Gamma\big(\frac{5}{2}\big)}.
\]
In the remaining of the proof, we assume that~$k \ge 2$.

We first study the asymptotic behavior of~$\gfo[M][x]$ and of all its derivatives around their minimal singularities.
The generating function~$\gfo[M][x]$ defines an analytic function around the origin.
Its dominant singularities are located at~$x = \pm \frac{1}{2}$.
Denoting by~$X_+ \eqdef \sqrt{1-2x}$ and~$X_- \eqdef \sqrt{1+2x}$, the Puiseux's expansions of~$\gfo[M][x]$ around~$x = \frac{1}{2}$ and~$x = -\frac{1}{2}$ are
\[
\gfo[M][x] \stackbin[x \sim \frac{1}{2}]{}{=} 2 - 2 \sqrt{2} \, X_+  + O\left({X_+}^2\right)
\qquad \text{and} \qquad
\gfo[M][x] \stackbin[x \sim -\frac{1}{2}]{}{=} 2 - 2 \sqrt{2} \, X_- \,  + O\left({X_-}^2\right),
\]
valid in a domain dented at $x=1/2$ and $x=-1/2$, respectively. Consequently,
\[
\diff\gfo[M][x] \stackbin[x \sim \frac{1}{2}]{}{=}  2 \sqrt{2} \, {X_+}^{-1} \, + O(1)
\qquad \text{and} \qquad
\diff\gfo[M][x] \stackbin[x \sim -\frac{1}{2}]{}{=}  - 2 \sqrt{2} \, {X_-}^{-1} \, + O(1),
\]
and for~$i>1$, the $i$\ordinal{} derivative of~$\gfo[M][x]$ has singular expansion around~$x = \pm \frac{1}{2}$
\begin{align*}
\diff[i]\gfo[M][x] & \stackbin[x \sim \frac{1}{2}]{}{=} 2\sqrt{2} \, (2i-3)!! \, {X_+}^{1-2i} \, +O\left({X_+}^{2-2i}\right),\\
\diff[i]\gfo[M][x] & \stackbin[x \sim -\frac{1}{2}]{}{=} (-1)^i \, 2\sqrt{2} \, (2i-3)!! \, {X_-}^{1-2i} \, +O\left({X_-}^{2-2i}\right),
\end{align*}
where $(2i-3)!! \eqdef (2i-3) \cdot (2i-5) \cdots 3 \cdot 1$.
These expansions are also valid in a dented domain at $x = \frac{1}{2}$ and $x = -\frac{1}{2}$, respectively.

We now exploit the expression of the generating function~$\gf[M][_k][x]$ given by Equation~\eqref{eq:product-matching} in the proof of Proposition~\ref{prop:generatingFunctionMatchings}.
The dominant singularities of~$\gf[M][_k][x]$ are located at~$x = \pm \frac{1}{2}$.
We provide the full analysis around~$x = \frac{1}{2}$, the computation for~$x=-\frac{1}{2}$ being similar.
For conciseness in the following expressions, we set by convention~$(-1)!!=1$.
We therefore obtain:
\begin{align*}
\gf[M][_k][x] & \; = x \diff \sum_{\substack{K \; k\text{-core} \\ \text{matching}}} \frac{1}{n(K)} \prod_{i \ge 1} \bigg( \frac{x^i}{(i-1)!} \diff[i-1] \big( x^{i-1} \gfo[M][x] \big) \bigg)^{n_i(K)} \\
& \stackbin[x \sim \frac{1}{2}]{}{=} \frac{1}{2} \diff \sum_{\substack{K \; k\text{-core} \\ \text{matching}}} \frac{1}{n(K)} \prod_{i \ge 1} \bigg( \frac{1}{2^{2i-1} \, (i-1)!} \diff[i-1] \gfo[M][x] \bigg)^{n_i(K)} \\
& \stackbin[x \sim \frac{1}{2}]{}{=} \frac{1}{2} \diff \sum_{\substack{K \; k\text{-core} \\ \text{matching}}} \frac{1}{n(K)} \prod_{i > 1} \bigg( \!\! -\frac{\sqrt{2} \, (2i-5)!!}{4^{i-1} \, (i-1)!} \, {X_+}^{3-2i} + O\left({X_+}^{4-2i}\right) \bigg)^{n_i(K)} \\
& \stackbin[x \sim \frac{1}{2}]{}{=} \frac{1}{2} \diff \sum_{\substack{K \; k\text{-core} \\ \text{matching}}} \frac{1}{n(K)} \prod_{i > 1} \bigg( \!\! -\frac{\sqrt{2} \, (2i-5)!!}{4^{i-1} \, (i-1)!} \bigg)^{n_i(K)} {X_+}^{-\potentialMatchings(K)} + O\left({X_+}^{-\potentialMatchings(K)+1}\right) \\
& \stackbin[x \sim \frac{1}{2}]{}{=} \frac{1}{2} \sum_{\substack{K \; k\text{-core} \\ \text{matching}}} \frac{-\potentialMatchings(K)}{n(K)} \prod_{i > 1} \bigg( \!\! -\frac{\sqrt{2} \, (2i-5)!!}{4^{i-1} \, (i-1)!} \bigg)^{n_i(K)} {X_+}^{-\potentialMatchings(K)-2} +O\left({X_+}^{-\potentialMatchings(K)-1}\right),
\end{align*}
where $\potentialMatchings(K) \eqdef \sum_{i>1} (2i-3) \, n_i(K)$ denotes the potential function studied in the previous section.
Observe that in order to obtain the third equality, we used the fact that~$k>1$, and thus, that there exists $k$-cores~$K$ such that~$n_i(K) \neq 0$ when~$i>1$.
Combining Lemma~\ref{lem:maximalCoreMatchings} and the Transfer Theorem for singularity analysis (see Theorem~\ref{theo:transfer} in Appendix~\ref{app:methodology}), we conclude that the main contribution in the asymptotic of the previous sum arises from maximal $k$-cores, as they maximize the value $2+\potentialMatchings(K)$.
There are exactly 4 maximal $k$-cores with~$n_1(K)=3k$,~$n_k(K)=1$, $n(K) = 4k$, and $\potentialMatchings(K)=2k-3$.
Hence,
\begin{align*}
\coeff{x^n}{\gf[M][_k][x]} & \stackbin[x \sim \frac{1}{2}]{}{=} \coeff{x^n}{\frac{1}{2}\sum_{\substack{K \; k\text{-core} \\ \text{matching}}} \frac{-\potentialMatchings(K)}{n(K)} \prod_{i > 1}  \bigg( \!\! -\frac{\sqrt{2} \, (2i-5)!!}{4^{i-1} \, (i-1)!} \bigg)^{n_i(K)} {X_+}^{-\potentialMatchings(K)-2} +O\left({X_+}^{-\potentialMatchings(K)-1}\right) }  \\
& \stackbin[x \sim \frac{1}{2}]{}{=} \frac{2\sqrt{2} \, (2k-3)!!}{4^k \, k!} \, \coeff{x^n} \, \sqrt{1-2x}^{1-2k} +O\left((1-2x)^{1-k}\right)   \\
& \stackbin[n \to \infty]{}{=}  \frac{2\sqrt{2} \, (2k-3)!!}{4^{k} \, k! \; \Gamma\big(k-\frac{1}{2}\big)} \, n^{k-\frac{3}{2}} \, 2^{n} \, (1 + o(1)),
\end{align*}
where the last equality is obtained by an application of the Transfer Theorem for singularity analysis (see Theorem~\ref{theo:transfer} in Appendix~\ref{app:methodology}).

Finally, we obtain the stated result by adding together the expression obtained when studying $\gf[M][_k][x]$ around $x=\frac{1}{2}$ and $x=-\frac{1}{2}$.
In fact, one can check that the asymptotic estimate of $\coeff{x^n}{\gf[M][k][x]}$ around~$x = -\frac{1}{2}$ is the same but with an additional multiplicative constant $(-1)^n$.
Consequently, the contribution is equal to $0$ when $n$ is odd and the estimate in the statement when~$n = 2m$ is even.
This is a particular example of the situation of Theorem~\ref{theo:multiple-singularities}.
\end{proof}


\subsection{Random generation}
\label{subsec:randomGenerationMatchings}

The composition scheme presented in Proposition~\ref{prop:generatingFunctionMatchings} can also be exploited in order to provide Boltzmann samplers for random generation of perfect matchings with~$k$ crossings.
Throughout this section we consider a positive real number~$\theta < \frac{1}{2}$, which acts as a ``control-parameter'' for the random sampler (see~\cite{DuchonFlajoletLouchardSchaeffer} for further details).

The Boltzmann sampler works in three steps:
\begin{enumerate}[(i)]
\item We first decide which is the core of our random object.
\item Once this core is chosen, we complete the matching by means of non-crossing (and possibly marked) matchings.
\item Finally, we place the root of the resulting perfect matching with~$k$ crossings.
\end{enumerate}

We start with the choice of the $k$-core.
For each $k$-core~$K$, let~$\gf[M][_K][x]$ denote the generating function of matchings with~$k$ crossings and whose $k$-core is $K$, where~$x$ marks as usual the number of vertices.
Note that this generating function is computed as in Proposition~\ref{prop:generatingFunctionMatchings}, using only the contribution of the $k$-core~$K$.
Therefore, we have
\[
\gf[M][_k][x]=\sum_{\substack{K \; k\text{-core} \\ \text{matching}}} \gf[M][_K][x].
\]
This sum defines a probability distribution in the following way: once fixed the parameter~$\theta$, let
\[
p_{K}=\frac{\gf[M][_K][\theta]}{\gf[M][_k][\theta]}.
\]
This set of values defines a Bernoulli distribution~$\{p_K\}_{\substack{K \; k\text{-core} \\ \text{matching}}}$, which can be easily simulated.

\begin{remark}
As it has been pointed out in Section~\ref{subsec:asymptoticMatchings}, the main contribution to the enumeration of perfect matchings with $k$ crossings, when the number of vertices is large enough, arises from the ones whose $k$-core is maximal.
Consequently, when $\theta$ is close enough to $\frac{1}{2}$, the first step in the random sampling would provide a maximal core with high probability.
To illustrate this fact, we have represented in \fref{fig:probabilitiesCoreMatchings} the probability of each possible $3$-core for a random perfect matching with~$3$ crossings.
\begin{figure}
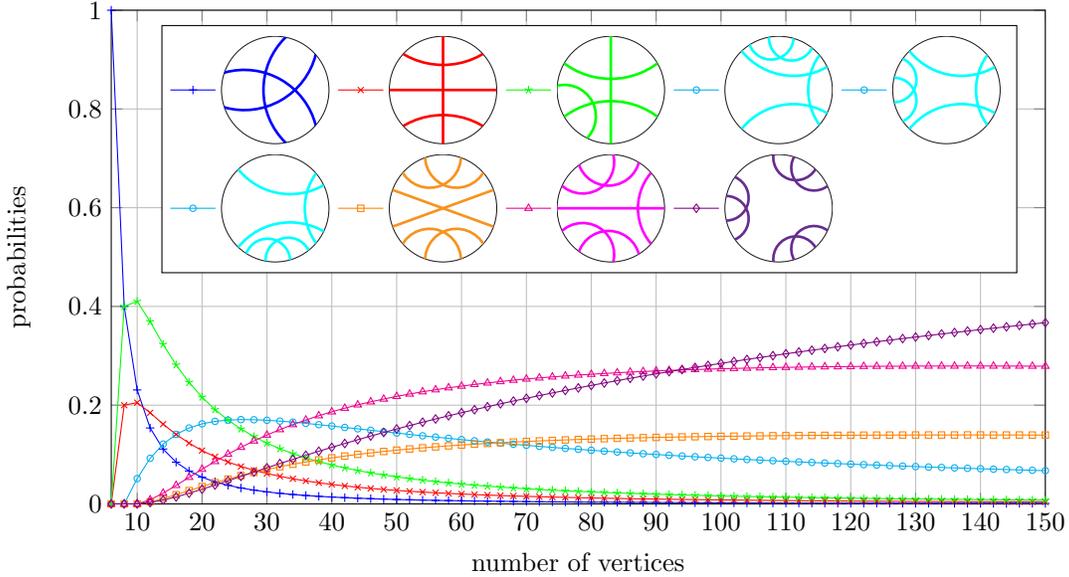

\begin{center}
\hspace*{-1cm}
\include{tableProbabilitiesCoreMatchings}
\end{center}
\vspace{-.7cm}
  \caption{Probabilities of appearance of the different $3$-core matchings.}
  \label{fig:probabilitiesCoreMatchings}
\end{figure}
\end{remark}

Once we have fixed the core of the random matching, we continue in the second step filling in its regions with crossing-free perfect matchings.
For this purpose it is necessary to start having a procedure to generate crossing-free perfect matchings, namely $\Gamma \gf[M][_0][\theta]$.
As $\gf[M][_0][\theta]$ satisfies the recurrence relation $\gf[M][_0][\theta]=1+\theta^2 \gf[M][_0][\theta]^2$, a Boltzmann sampler $\Gamma \gf[M][_0][\theta]$ can be defined in the following way.
Let $p=\frac{1}{\Gamma \gf[M][_0][\theta]}$.
Then, using the language of~\cite{DuchonFlajoletLouchardSchaeffer},
\[
\Gamma \gf[M][_0][\theta] \eqdef \mathrm{Bern}(p) \longrightarrow \varnothing \mid (\Gamma \gf[M][_0][\theta] \, , \, \bullet - \bullet \, , \, \Gamma \gf[M][_0][\theta]),
\]
where $\bullet - \bullet$ means that the Boltzmann sampler is generating a single chord (or equivalently, two vertices in the border of the circle).
This Boltzmann sampler is defined when $\theta < \frac{1}{2}$, in which case the defined branching process is subcritical.
In such situation the algorithm stops in finite expected time, see~\cite{DuchonFlajoletLouchardSchaeffer}.

Once this random sampler is performed, we can deal with a term of the form $\diff[i-1] x^{i-1}\gf[M][_0][\theta]$.
Indeed, once a random crossing-free perfect matching~$\Gamma \gf[M][_0][\theta]$ of size~$n(\Gamma \gf[M][_0][\theta])$ is generated, there exist
\[
\binom{n(\Gamma \gf[M][_0][\theta])+i-1}{i-1}
\]
$i$-marked crossing-free perfect matching arising from $\Gamma \gf[M][_0][\theta]$.
Hence, with uniform probability we can choose one of these $i$-marked crossing-free perfect matchings.
As this argument follows for each choice of~$i$, and $\gf[KM][_K][\b{x}]$ is a polynomial, we can combine the generator of $i$-marked crossing-free diagrams with the Boltzmann sampler for the cartesian product of combinatorial classes (recall that we need to provide the substitution $x_i \leftarrow \frac{x^i}{(i-1)!} \diff[i-1] \left( x^{i-1} \gfo[M][x] \right)$).

Finally, we need to apply the root operator, which can be done by means of similar arguments as in the case of $i$-marked crossing-free diagrams.

Concerning the statistics of the random variable~$N$ corresponding to the size of the element generated by means of the previous random sampler, as it is shown in \cite{DuchonFlajoletLouchardSchaeffer}, the expected value~$\Exp[N]$ and the variance~$\Var[N]$ of the random variable $N$ satisfy
\[
\Exp[N]= \theta \, \frac{\gf[M]['_k][\theta]}{\gf[M][_k][\theta]}
\qquad\text{and}\qquad
\Var[N] =\frac{\theta^2 (\gf[M][''_k][\theta]\gf[M][_k][\theta] - \theta \, \gf[M]['_k][\theta]^2) + \theta \, \gf[M]['_k][\theta] }{\gf[M][_k][\theta]^2}.
\]
Hence, when $\theta$ tends to $\frac{1}{2}$, the expected value of the generated element tends to infinity, and the variance for the expected size also diverges.
Consequently, the random variable~$N$ is not concentrated around its expected value.

\begin{example}
\label{ex: sampler-3-cores}
For perfect matchings with~$3$ crossings, the expectation~$\Exp[N]$ and the variance~$\Var[N]$ are given by
\begin{center}
\begin{tabular}{c|cccccc}
$\theta$ & $0.40$ & $0.45$ & $0.465$ & $0.475$ & $0.48$ & $0.4999$ \\
$\Exp[N]$ & $17.31$ & $30.66$ & $41.78$ & $56.42$ & $69.14$ & $12508.22$ \\
$\sqrt{\Var[N]}$ & $7.69$ & $44.44$ & $109.44$ & $249.83$ & $427.32$ & $0.7406 \cdot 10^8$
\end{tabular}
\end{center}
\end{example}


\subsection{Connection to other results}
\label{subsec:connectionOtherResults}

From the works of J.~Touchard~\cite{Touchard} and J.~Riordan~\cite{Riordan}, we know a remarkable explicit formula for the distribution of crossings among all perfect matchings on $m$ chords (and thus with $2m$ vertices):
\[
\gf[M][^{2m}][z] \eqdef \coeff{x^{2m}}{\gf[M][][x,z]} = \frac{1}{(1-z)^{m}} \sum_{k = -m}^{m} (-1)^k \binom{2m}{m+k} z^{\binom{k}{2}}.
\]
Extracting the coefficient of $z^k$ in this formula and summing over all integers~$m$, we obtain again the generating function~$\gf[M][_k][x]$ for the perfect matchings with $k$ crossings:
\begin{align*}
\gf[M][_k][x] & = \coeff{z^k}{\gf[M][][x,z]} = \sum_{m \in \N} \big( \coeff{x^{2m}}{\coeff{z^k}{\gf[M][][x,z]}} \big)\,x^{2m} = \sum_{m \in \N} \big( \coeff{z^k}{\coeff{x^{2m}}{\gf[M][][x,z]}} \big)\,x^{2m} \\
& = \sum_{m \in \N} \;[z^k] \! \left( \frac{1}{(1-z)^{m}} \sum_{k = -m}^{m} (-1)^k \binom{2m}{m+k} z^{\binom{k}{2}} \right) \,x^{2m}.
\end{align*}
Applying Proposition~\ref{prop:generatingFunctionMatchings}, we therefore obtain for any fixed $k \in \N$ the identity
\begin{gather*}
x\diff\gf[KM][_k][x_i \leftarrow {\frac{x^i}{(i-1)!} \diff[i-1] \big( x^{i-1} \gfo[M][x] \big)}] = \sum_{m \in \N} \;[z^k] \! \left( \frac{1}{(1-z)^{m}} \sum_{k = -m}^{m} (-1)^k \binom{2m}{m+k} z^{\binom{k}{2}} \right) \,x^{2m}.
\end{gather*}
Here, one could expect to be able to extract the coefficients of the polynomial~$\gf[KM][_k][\b{x}]$ by identification of the coefficients in the first few terms of the series of this identity.
However, it turns out that the resulting system of equations is underdetermined.
The algorithm presented in Section~\ref{subsec:connectedMatchings} and~\ref{subsec:computingCoreMatchingPolynomials} is therefore needed to compute the $k$-core matching polynomial~$\gf[KM][_k][\b{x}]$.

Our results also complements the ones obtained in~\cite{FlajoletNoy-chordDiagrams} by P.~Flajolet and M.~Noy.
In this work the authors studied, among other parameters, the limit distribution of the number of crossings when the number of chords is large enough, obtaining normal limiting distributions.
The main tool used by the authors is exploiting by analytic means Touchard-Riordan formulas.


\subsection{Extension to partitions}
\label{subsec:partitions}

To finish this section, we extend our results from perfect matchings to partitions.
(See also Section~\ref{subsec:partitionsFixedSizes} for further extension to partitions with restricted block sizes.)
We now consider the family $\F[P]$ of partitions of point sets on the unit circle.
As before, the partitions are rooted by a mark on an arc between two vertices.
A \defn{crossing} between two blocks~$U,V$ of a partition~$P$ is a pair of crossing chords~$u_1u_2$ and~$v_1v_2$ where~$u_1,u_2 \in U$ and~$v_1, v_2 \in V$.
We count crossings with multiplicity: two blocks~$U,V$ cross as many times as the number of such pairs of crossing chords among~$U$ and~$V$.
Note that perfect matchings are particular partitions where all blocks have size~$2$.

Let~$\Fcoeff[P][n,m,k]$ denote the set of partitions in~$\F[P]$ with $n$ vertices, $m$ blocks, and~$k$ crossings (counted with multiplicity).
We define the generating functions
\[
\gf[P][][x,y,z] \eqdef \sum_{n,m,k \in \N} |\Fcoeff[P][n,m,k]| \, x^n y^m z^k \qquad\text{and}\qquad \gf[P][_k][x,y] \eqdef \coeff{z^k}{\gf[P][][x,y,z]},
\]
of partitions, and partitions with $k$ crossings respectively.
We study partitions with crossings focussing on their cores.

\begin{definition}
A \defn{core partition} is a partition where each block is involved in a crossing.
It is a \defn{$k$-core partition} if it has exactly $k$ crossings.
The \defn{core}~$\core{P}$ of a partition~$P$ is the subpartition of~$P$ formed by all its blocks involved in at least one crossing.
See \fref{fig:partition}.
\end{definition}

\begin{figure}[h]
  \centerline{\includegraphics[scale=.75]{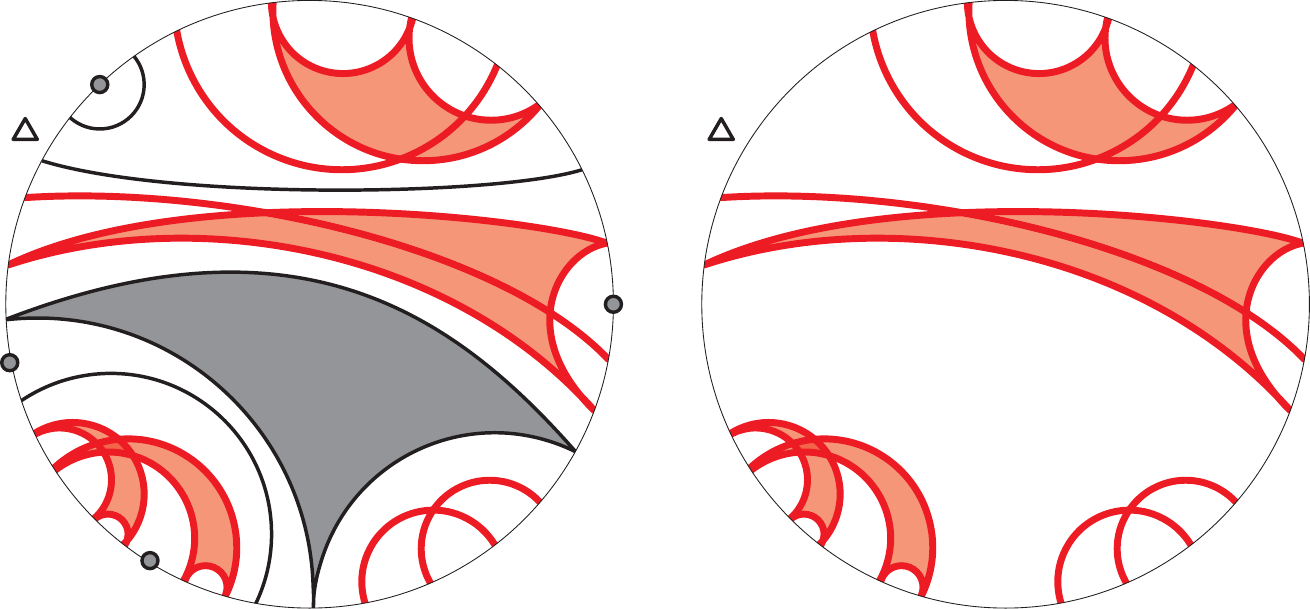}}
  \caption{A partition~$P$ with $9$ crossings (left) and its $9$-core~$\core{P}$ (right). Blocks are represented by shaded regions, and blocks of size~$1$ are represented by shaded vertices. The core partition~$\core{P}$ has $n(\core{P}) = 20$ vertices, $m(\core{P}) = 8$ blocks, and $k(\core{P}) = 9$ crossings. Moreover, $\b{n}(\core{P}) = (15,1,1)$ since it has $15$ regions with one boundary arc, $1$ with two boundary arcs, and $1$ with three boundary arcs.}
  \label{fig:partition}
\end{figure}

Let~$K$ be a core partition.
We let~$n(K)$ denote its number of vertices, $m(K)$ denote its number of blocks, and $k(K)$ denote its number of crossings.
We call \defn{regions} of~$K$ the connected components of the complement of~$K$ in the unit disk.
A region has~$i$ \defn{boundary arcs} if its intersection with the unit circle has~$i$ connected arcs.
We let~$n_i(K)$ denote the number of regions of~$K$ with $i$ boundary arcs, and we set~$\b{n}(K) \eqdef (n_i(K))_{i \in [k]}$.
Note that~$n(K) = \sum_i i n_i(K)$.
See again \fref{fig:partition} for an illustration.

Since a crossing only involves $2$ blocks, a $k$-core partition can have at most $2k$ blocks.
Moreover, since we count crossings with multiplicities, the size of each block of a $k$-core partition is at most~$k+1$.
This immediately implies the following crucial lemma.

\begin{lemma}
There are only finitely many $k$-core partitions.
\end{lemma}

Note that this lemma would be wrong if we would not count crossings between blocks of the partition with multiplicities.

\begin{definition}
We encode the finite list of all possible $k$-core partitions~$K$ and their parameters~$n(K)$, $m(K)$, and~$\b{n}(K) \eqdef (n_i(K))_{i \in [k]}$ in the \defn{$k$-core partition polynomial}
\[
\gf[KP][_k][\b{x},y] \eqdef \sum_{\substack{K\;k\text{-core} \\ \text{partition}}} \frac{\b{x}^{\b{n}(K)} y^{m(K)}}{n(K)} \eqdef \sum_{\substack{K\;k\text{-core} \\ \text{partition}}} \frac{1}{n(K)}\prod_{i \in [k]} {x_i}^{n_i(K)} \, y^{m(K)}.
\]
\end{definition}

\begin{example}
\label{exm:smallCorePartitions}
Besides the $k$-core matchings from \fref{fig:smallCoreMatchings}, there are height (unrooted) $k$-core partitions, represented in \fref{fig:smallCorePartitions}.
From this exhaustive enumeration, we can compute the $1$-, $2$-, and $3$-core partition polynomials:
\begin{figure}
  \centerline{\includegraphics[width=.9\textwidth]{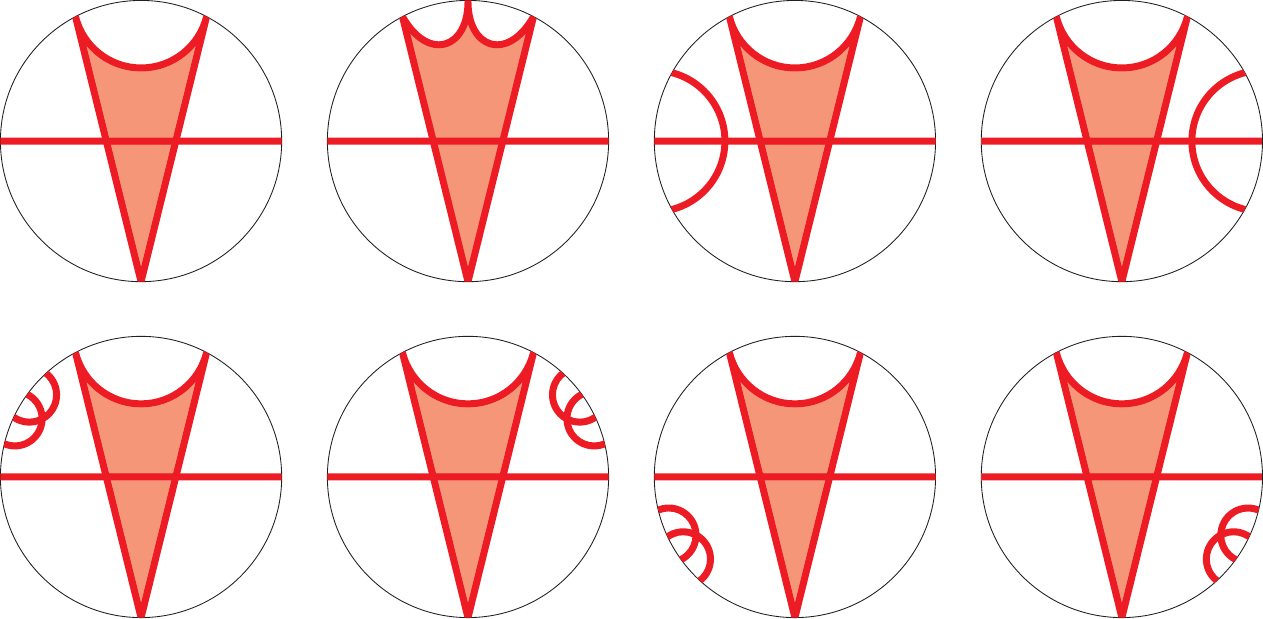}}
  \caption{The $2$- and $3$-core partitions (unrooted) which are not $2$- or $3$-core matchings.}
  \label{fig:smallCorePartitions}
\end{figure}
\begin{gather*}
\gf[KP][_1][\b{x}, y] = \frac{1}{4} \, {x_1}^4 \, y^2, \\
\gf[KP][_2][\b{x}, y] = \frac{1}{2} \, {x_1}^6 \, y^3+ \frac{1}{2} \, {x_1}^6 \, x_2 \, y^4 + {x_1}^5 \, y^2, \\
\gf[KP][_3][\b{x}, y] = \frac{1}{6} \, {x_1}^6 \, y^3 + \frac{3}{2} \, {x_1}^8 \, y^4 + 3 \, {x_1}^8 \, x_2 \, y^5 + \frac{3}{2} \, {x_1}^8 \, {x_2}^2 \, y^6 + \frac{1}{3} \, {x_1}^9 \, x_3 \, y^6 + 4 \, {x_1}^7 \, x_2 \, y^4 + {x_1}^6 \, y^2 + 2 \, {x_1}^7 \, y^3. \\
\end{gather*}
\end{example}

\begin{remark}
The algorithm presented in Section~\ref{subsec:connectedMatchings} to generate connected matchings can be extended and adapted to generate connected partitions (where we count crossings with multiplicities).
Similarly, we can still decompose a core partition into an arborescence of connected partitions, as we did for core matchings in Section~\ref{subsec:computingCoreMatchingPolynomials}.
The only difference here is that we have to consider the family~$\F[R]$ of rooted embedded unlabeled trees where each internal vertex has at least four leaves, including its first and last children.
(Compared to the case of matchings, we just drop the condition that the internal vertices have even degree, since connected partitions can have an odd number of vertices.)
Using similar notations as in Section~\ref{subsec:computingCoreMatchingPolynomials}, we obtain that
\begin{gather*}
\gf[R][][\b{x},\b{t}] = \sum_{j \ge 4} t_j \bigg( \sum_{i \ge 1} x_i \, \gf[R][][\b{x},\b{t}]^{i-1} \bigg)^{j-1} \\
\text{and}\qquad
\gf[KP][][\b{x},y,z] = \int_{s=0}^{s=1} \frac{1}{s} \sum_{i \ge 1} x_i \, s^i \, \gf[R][][x_j \leftarrow x_j s^j, t_j \leftarrow {\gf[CP][^j][y,z]}]^i ds.
\end{gather*}
As for matchings, this provides an effective method to compute $k$-core partition polynomials.
\end{remark}

Applying the same method as in Section~\ref{subsec:generatingFunctionMatchings}, we obtain an expression of the generating function~$\gf[P][_k][x,y]$ of partitions with~$k$ crossings in terms of the $k$-core partition polynomial~$\gf[KP][_k][\b{x},y]$.

\begin{proposition}
\label{prop:polynomialPartitions}
For any~$k \ge 1$, the generating function~$\gf[P][_k][x,y]$ of partitions with $k$ crossings is given by
\[
\gf[P][_k][x,y] = x\diff\gf[KP][_k][x_i \leftarrow {\frac{x^i}{(i-1)!} \diff[i-1] \big( x^{i-1} \gfo[P][x,y] \big)},y].
\]
In particular, $\gf[P][_k][x,y]$ is a rational function of~$\gfo[P][x,y]$ and~$x$.
\end{proposition}

\begin{proof}
The proof of the formula is identical to that of Proposition~\ref{prop:generatingFunctionMatchings}, replacing matchings by partitions.
The only slight difference concerns the proof of the rationality of~$\gf[P][_k][x,y]$ as a function of~$\gfo[P][x,y]$ and~$x$.
Splitting a crossing-free partition according to its block containing its first vertex, we obtain that
\[
\gfo[P][x,y] = 1 + \frac{x \, y \, \gfo[P][x,y]}{1 - x \, \gfo[P][x,y]},
\]
and therefore
\begin{equation}
\label{eq:crossingFreePartitions}
x \, \gfo[P][x,y]^2 + (xy-x-1) \, \gfo[P][x,y] + 1 = 0.
\end{equation}
Note that we recover the Catalan functional equation when we set~$y=1$.
Derivating (with respect to~$x$) the functional Equation~\eqref{eq:crossingFreePartitions}, we obtain
\[
\diff\gfo[P][x,y] = \frac{y \, \gfo[P][x,y]}{1 - x \, \gfo[P][x,y]^2 - xy}.
\]
Therefore, all the derivatives (with respect to~$x$) of~$\gfo[P][x,y]$, and thus also the generating function~$\gf[P][_k][x,y]$ of partitions with~$k$ crossings, are rational in~$\gfo[P][x,y]$ and the variables~$x$ and~$y$.
This concludes the proof since we can finally eliminate~$y$ from Equation~\eqref{eq:crossingFreePartitions}.
\end{proof}

\begin{remark}
The rationality of~$\gf[P][_k][x,y]$ as a function of~$\gfo[P][x,y]$ and~$x$ was already proved by M.~B\'ona~\cite{Bona} using a different method.
We believe that our method simplifies the proof and can be applied (as we will observe in the rest of this paper) to many other similar problems.
\end{remark}

\begin{example}
Since $\gf[KP][_1][\b{x}, y] = \frac{1}{4} \, {x_1}^4 \, y^2$, we have
\begin{align*}
\gf[P][_1][x,y] = & \; \frac{1}{2} x^2 \, \gfo[P][x]^2 (\gfo[P][x] - 1)^2 (1 - x \, \gfo[P][x]) (1 - 2x \, \gfo[P][x]) \\
= & \quad\;
x^4 \,\, y^2 \\
& + 5 \,\, x^5 \, y^3 \\
& + x^6 \,\, (6 \, y^3 + 15 \, y^4) \\
& + x^7 \,\, (7 \, y^3 + 42 \, y^4 + 35 \, y^5) \\
& + x^8 \,\, (8 \, y^3 + 84 \, y^4 + 168 \, y^5 + 70 \, y^6) \\
& + x^9 \,\, (9 \, y^3 + 144 \, y^4 + 504 \, y^5 + 504 \, y^6 + 126 \, y^7) \\
& + x^{10} \, (10 \, y^3 + 225 \, y^4 + 1200 \, y^5 + 2100 \, y^6 + 1260 \, y^7 + 210 \, y^8) \\
& + x^{11} \, (11 \, y^3 + 330 \, y^4 + 2475 \, y^5 + 6600 \, y^6 + 6930 \, y^7 + 2772 \, y^8 + 330 \, y^9) \, \dots
\end{align*}
If we forget the parameter~$y$ codifying the number of blocks, we obtain that
\[
\gf[P][_1][x,1] = \frac{x \big( 1 - \sqrt{1 - 4 x} \big)^3}{8 \sqrt{1 - 4 x}} = \sum_{n \ge 4} \binom{2n-5}{n-4} \, x^n.
\]
This nice expression of the number of partitions of~$[n]$ with $1$ crossing was already observed in~\cite{Bona}.

We omit the exact expressions of the generating functions of partitions with $2$ or $3$ crossings since they are too long for the linewidth of this paper.
We can however give the first few terms of their developments:
\begin{align*}
\gf[P][_2][x,y] = & \quad\;
5 \, x^5 \, y^2 \\
& + 33 \, x^6 \, y^3 \\
& + x^7 \,\, (35 \, y^3 + 126 \, y^4) \\
& + x^8 \,\, (40 \, y^3 + 308 \, y^4 + 364 \, y^5) \\
& + x^9 \,\, (45 \, y^3 + 567 \, y^4 + 1512 \, y^5 + 882 \, y^6) \\
& + x^{10} \, (50 \, y^3 + 930 \, y^4 + 4050 \, y^5 + 5460 \, y^6 + 1890 \, y^7) \\
& + x^{11} \, (55 \, y^3 + 1408 \, y^4 + 8965 \, y^5 + 19965 \, y^6 + 16170 \, y^7 + 3696 \, y^8) \, \dots
\end{align*}
\begin{align*}
\gf[P][_3][x,y] = & \quad\;
x^6 \, (6 \, y^2 + y^3) \\
& + x^7 \,\, (56 \, y^3 + 7 \, y^4) \\
& + x^8 \,\, (48 \, y^3 + 300 \, y^4 + 28 \, y^5) \\
& + x^9 \,\, (54 \, y^3 + 603 \, y^4 + 1188 \, y^5 + 84 \, y^6) \\
& + x^{10} \, (60 \, y^3 + 960 \, y^4 + 4065 \, y^5 + 3840 \, y^6 + 210 \, y^7) \\
& + x^{11} \, (66 \, y^3 + 1485 \, y^4 + 9042 \, y^5 + 19470 \, y^6 + 10692 \, y^7 + 462 \, y^8) \, \dots
\end{align*}
\end{example}

From the expression of the generating function~$\gf[P][_k][x,y]$ given in Proposition~\ref{prop:polynomialPartitions}, we can now extract asymptotic estimates for the number of partitions with $k$ crossings.
The proof of the following statement is similar to that of Proposition~\ref{prop:asymptMatchings}: the main contribution to the asymptotic of partitions with $k$ crossings still arises from maximal $k$-core matchings.
We leave the details to the reader.

\begin{proposition}
\label{prop:asymptPartitions}
For any~$k \ge 1$, the number of partitions with~$k$ crossings and~$n$ vertices is
\[
\coeff{x^{n}}{\gf[P][_k][x, 1]} \stackbin[n \to \infty]{}{=} \frac{(2k-3)!!}{2^{3k-1} \, k! \; \Gamma\big(k-\frac{1}{2}\big)} \, n^{k-\frac{3}{2}} \, 4^n \, (1+o(1)).
\]
\end{proposition}

We can also study the limiting distribution of the number of blocks for a partition with $k$ crossings and $n$ vertices, chosen uniformly at random.
With this purpose, we study the behavior of the singularity of $\gfo[P][x, y]$ when varying $y$ around $1$.
From Equation~\eqref{eq:crossingFreePartitions} we deduce that
\[
\gfo[P][x, y]=\frac{x+1-xy- \sqrt{(xy-x-1)^2-4x}}{2x}.
\]
Consequently, the singularity curve safisfies the implicit equation
\[
(y\rho(y)-\rho(y)-1)^2-4\rho(y)=0.
\]
In particular $\rho(1)=\frac{1}{4}$.
Direct computations give the following parameters:
\[
-\frac{\rho'(1)}{\rho(1)}= \frac{1}{2} \qquad\text{and}\qquad -\frac{\rho''(1)}{\rho(1)}-\frac{\rho'(1)}{\rho(1)}+\left(\frac{\rho'(1)}{\rho(1)}\right)^2=\frac{1}{8}.
\]
These parameters are useful in the following statement.

\begin{proposition}
\label{prop:limlawblocks}
The number of blocks in a partition with $k$ crossings and $n$ vertices, chosen uniformly at random, follows a normal distribution with expectation $\mu_n$ and variance $\sigma_n$, where
\[
\mu_n = \frac{1}{2} \, n \, (1+o(1))
\qquad\text{and}\qquad
\sigma_n = \frac{1}{8} \, n \, (1+o(1)).
\]
\end{proposition}

\begin{proof}
By Proposition~\ref{prop:polynomialPartitions} we have
\[
\gf[P][_k][x, y] \stackbin[x = \frac{1}{4}]{}{=} C(x,y) \left(1-\frac{x}{\rho(y)}\right)^{\frac{1}{2}-k}+O\left(\left(1-\frac{x}{\rho(y)}\right)^{1-k}\right).
\]
uniformly in a neighbourhood of $y=1$, where $C(x,y)$ is analytic around $(x,y) = (\rho(1),1)$.
Consequently, we can apply the Quasi-Powers Theorem~\ref{theo:quasi-powers} with the values obtained above.
\end{proof}


\subsection{Extension to partitions with restricted block sizes}
\label{subsec:partitionsFixedSizes}

For a non-empty subset~$S$ of $\N^* \eqdef \N \ssm \{0\}$, we denote by~$\F[P]^S$ the family of partitions of point sets on the unit circle, where the cardinality of each block belongs to the set~$S$.
For example, matchings are partitions where all blocks have size~$2$, \ie $\F[M] = \F[P]^{\{2\}}$.
Observe that depending on~$S$ and~$k$, it is possible that no partition of~$\F[P]^S$ has exactly $k$ crossings.
For example, since two triangles can have either $0$, $4$, or~$6$ crossings, there is no $3$-uniform partition (\ie with~$S = \{3\}$) with an odd number of crossings.

Applying once more the same method as in Section~\ref{subsec:generatingFunctionMatchings}, we obtain an expression of the generating function~$\gf[P][^S_k][x,y]$ of partitions of~$\F[P]^S$ with~$k$ crossings in terms of the corresponding $k$-core partition polynomial
\[
\gf[KP][^S_k][\b{x},y] \eqdef \sum_{\substack{K\;k\text{-core} \\ \text{partition of } \F[P]^S}} \frac{\b{x}^{\b{n}(K)} y^{m(K)}}{n(K)}.
\]
We say that~$S \subset \N^*$ is \defn{ultimately periodic} if it can be written as
\[
S = A_S \cup \bigcup_{b \in B_S} \set{b+up_S}{u \in \N}
\]
for two finite subsets~$A_S,B_S \subset \N^*$ and a period~$p_S \in \N^*$.

\begin{proposition}
\label{prop:polynomialPartitionsFixedSizes}
For any~$k \ge 1$, the generating function~$\gf[P][^S_k][x,y]$ of partitions with $k$ crossings and where the size of each block belongs to~$S$ is given by
\[
\gf[P][^S_k][x,y] = x\diff\gf[KP][^S_k][x_i \leftarrow {\frac{x^i}{(i-1)!} \diff[i-1] \big( x^{i-1} \gf[P][^S_0][x,y] \big)},y].
\]
If~$S$ is finite or ultimately periodic, then $\gf[P][^S_k][x,y]$ is a rational function of~$\gf[P][^S_0][x,y]$ and~$x$.
\end{proposition}

\begin{proof}
The proof is again similar to that of Proposition~\ref{prop:generatingFunctionMatchings}, replacing matchings by partitions of~$\F[P]^S$.
Again, the difference lies in proving that the successive derivatives of~$\gf[P][^S_0][x,y]$ and the variable~$y$ are all rational functions of~$\gf[P][^S_0][x,y]$ and~$x$.
Splitting a crossing-free partition of~$\F[P]^S$ with respect to its block containing its first vertex, we obtain the functional equation
\[
\gf[P][^S_0][x,y] = 1 + y \, \sum_{s \in S} x^s \, \gf[P][^S_0][x,y]^s.
\]
If~$S$ is finite or ultimately periodic, we write $S = A_S \cup \bigcup_{b \in B_S} \set{b+up_S}{u \in \N}$ for finite subsets~$A_S,B_S \subset \N^*$ and a period~$p_S \in \N^*$, and we can write
\[
\sum_{s \in S} t^s = A_S(t) + \frac{B_S(t)}{1-t^{p_S}},
\]
where~$A_S(t) \eqdef \sum_{a \in A_S} t^a$ and~$B_S(t) \eqdef \sum_{b \in B_S} t^b$.
We thus obtain that
\begin{gather*}
(\gf[P][^S_0][x,y] - 1 - y \, A_S(x \, \gf[P][^S_0][x,y])) (1 - x^{p_S} \, \gf[P][^S_0][x,y]^{p_S}) - y \, B_S(x \, \gf[P][^S_0][x,y]) = 0
\\
\text{and}\qquad
y = \frac{(\gf[P][^S_0][x,y] - 1) \big( 1 - x^{p_S} \, \gf[P][^S_0][x,y]^{p_S} \big)}{A_S(x \, \gf[P][^S_0][x,y]) \big( 1- x^{p_S} \, \gf[P][^S_0][x,y]^{p_S} \big) + B_S(x \, \gf[P][^S_0][x,y])}.
\end{gather*}
Derivating the former functional equation ensures that the successive derivatives of~$\gf[P][^S_0][x,y]$ are all rational functions of~$\gf[P][^S_0][x,y]$ and the variables~$x$ and~$y$.
The latter equation ensures that~$y$ itself is rational in~$\gf[P][^S_0][x,y]$ and~$x$, thus concluding the proof.
\end{proof}

From the expression of~$\gf[P][^S_k][x,y]$ given in Proposition~\ref{prop:polynomialPartitionsFixedSizes}, we can extract asymptotic estimates for the number of partitions with $k$ crossings and where the size of each block belongs to~$S$.
The difficulty here lies in two distinct aspects:
\begin{enumerate}[(i)]
\item estimate the minimal singularity~$\rho_S$ and describe the singular behavior around~$\rho_S$ of the generating function~$\gf[P][^S_0][x,1]$ of crossing-free partitions of~$\F[P]^S$, and
\item characterize which $k$-core partitions of~$\F[P]^S$ have the main contribution to the asymptotic.
\end{enumerate}
The first point is discussed in details below in Proposition~\ref{prop:singularBehaviorPS0}.
In contrast, we are able to handle the second point only for particular cases, which we illustrate in Examples~\ref{exm:uniformPartitions} and~\ref{exm:multiplePartitions}.
The following constants will be needed in Propositions~\ref{prop:singularBehaviorPS0} and~\ref{prop:asymptPartitionsFixedSizes}.

\begin{definition}
\label{def:constants}
Given a non-empty subset~$S$ of~$\N^*$ different from the singleton~$\{1\}$, we define~$\tau_S$ to be the unique positive real number such that
\[
\sum_{s \in S} (s-1) {\tau_S}^s = 1.
\]
We furthermore define the constants $\rho_S$, $\alpha_S$ and~$\beta_S$ to be
\[
\rho_S \eqdef \frac{\tau_S}{\sum_{s \in S} s {\tau_S}^s},
\qquad
\alpha_S \eqdef 1 + \sum_{s \in S} {\tau_S}^s,
\qquad\text{and}\qquad
\beta_S \eqdef \sqrt{\frac{2 \left( \sum_{s \in S} s {\tau_S}^s \right)^3}{\sum_{s \in S} s(s-1) {\tau_S}^s}}.
\]
\end{definition}

\begin{remark}
Observe that~$\tau_S$ is indeed well-defined, unique and belongs to~$]0,1]$.
Indeed the function~$\tau \mapsto \sum_{s \in S} (s-1) \tau^s$ is strictly increasing, evaluates to~$0$ when~$\tau = 0$, and is either a power series with radius of convergence~$1$ (if~$S$ is infinite), or a polynomial which evaluates at least to~$1$ when~$\tau = 1$ (if~$S$ is finite).
Observe also that
\[
\rho_S = \frac{\tau_S}{1 + \sum_{s \in S} {\tau_S}^s} \qquad\text{and}\qquad \alpha_S = \frac{\tau_S}{\rho_S},
\]
and that these two constants are both positive.
\end{remark}

These constants naturally appear in the proof of the following statement, which describes the singular behavior of~$\gf[P][^S_0][x,1]$ and the asymptotic of its coefficients.

\begin{proposition}
\label{prop:singularBehaviorPS0}
For any non-empty subset~$S$ of~$\N^*$ different from the singleton~$\{1\}$, the generating function~$\gf[P][^S_0][x,1]$ satisfies
\[
\gf[P][^S_0][x,1] \stackbin[x \sim \rho_S]{}{=} \alpha_S - \beta_S \, \sqrt{1 - \frac{x}{\rho_S}} \,  + O\bigg( 1 - \frac{x}{\rho_S} \bigg),
\]
in a domain dented at $x=\rho_S$, for the constants $\rho_S$, $\alpha_S$ and~$\beta_S$ described in Definition~\ref{def:constants}.
Therefore, its coefficients satisfy
\[
\coeff{x^n}{\gf[P][^S_0][x,1]} \stackbin[\substack{n \to \infty \\ \gcd(S) | n}]{}{=} \frac{\gcd(S) \, \beta_S}{2\sqrt{\pi}} \, n^{-\frac{3}{2}} \, {\rho_S}^{-n} \, (1+o(1))
\]
for $n$ multiple of~$\gcd(S)$, while $\coeff{x^n}{\gf[P][^S_0][x,1]} = 0$ if $n$ is not a multiple of~$\gcd(S)$.
\end{proposition}

\enlargethispage{.3cm}
\begin{proof}
We apply the theorem of A.~Meir and~J.~Moon~\cite{MeirMoon} on the singular behavior of generating functions defined by a smooth implicit-function schema. These notions are recalled in Definition~\ref{def:SmoothImplicitFunctionSchema} and Theorem~\ref{theo:SmoothImplicitFunctionTheorem} from Appendix~\ref{app:methodology}.
As already observed, the generating function~$\gf[P][^S_0][x,1]$ satisfies the functional equation
\[
\gf[P][^S_0][x,1] = 1 + \sum_{s \in S} x^s \, \gf[P][^S_0][x,1]^s.
\]
If we set
\[
\gf[W][][x] \eqdef \gf[P][^S_0][x,1] - 1
\qquad\text{and}\qquad
\gf[G][][x,w] \eqdef \sum_{s \in S} x^s (w+1)^s,
\]
then we obtain a smooth implicit-function schema
\[
\gf[W][][x] = \gf[G][][{x,\gf[W][][x]}].
\]
Conditions~\eqref{item:SmoothImplicitFunctionSchema1} and~\eqref{item:SmoothImplicitFunctionSchema2} of Theorem~\ref{theo:SmoothImplicitFunctionTheorem} are clearly satisfied.
To check Condition~\eqref{item:SmoothImplicitFunctionSchema3}, fix
\[
u \eqdef \rho_S = \frac{\tau_S}{\sum_{s \in S} s {\tau_S}^s}
\qquad\text{and}\qquad
v \eqdef \alpha_S - 1 = \frac{\tau_S}{\rho_S} - 1 = \sum_{s \in S} {\tau_S}^s,
\]
and observe that
\begin{align*}
\gf[G][][u,v] & = \sum_{s \in S} u^s(v+1)^s = \sum_{s \in S} {\rho_S}^s \bigg(\frac{\tau_S}{\rho_S}\bigg)^s = \sum_{s \in S} {\tau_S}^s = v \qquad\text{and} \\
\gf[G][_w][u,v] & = \sum_{s \in S} s u^s(v+1)^{s-1} = \sum_{s \in S} s {\rho_S}^s \bigg(\frac{\tau_S}{\rho_S}\bigg)^{s-1} = \frac{\rho_S}{\tau_S} \sum_{s \in S} s{\tau_S}^s = 1.
\end{align*}
The statement is therefore a direct application of Theorem~\ref{theo:SmoothImplicitFunctionTheorem}.
\end{proof}

From the singular behavior of~$\gf[P][^S_0][x,1]$, and using the composition scheme of Proposition~\ref{prop:polynomialPartitionsFixedSizes}, we can now extract asymptotic estimates for the number of partitions of~$\F[P]^S$ with~$k$ crossings.

\begin{proposition}
\label{prop:asymptPartitionsFixedSizes}
Let~$k \ge 1$, let~$S$ be a non-empty subset of~$\N^*$ different from the singleton~$\{1\}$, let~$\tau_S$, $\rho_S$, $\alpha_S$ and~$\beta_S$ be the constants described in Definition~\ref{def:constants}, and let~$\maxPotentialMatchings{k,S}$ denote the maximum value of the potential function
\[
\potentialMatchings(K) \eqdef \sum_{i > 1} (2i-3) \, n_i(K)
\]
over all $k$-core partitions of~$\F[P]^S$.
There is a constant~$\Lambda_S$ such that the number of partitions with $k$ crossings, $n$ vertices, and where the size of each block belongs to~$S$~is
\[
\coeff{x^n}{\gf[P][^S_k][x,1]} \stackbin[\substack{n \to \infty \\ \gcd(S) | n}]{}{=} \Lambda_S \, n^{\frac{\maxPotentialMatchings{k,S}}{2}} \, {\rho_S}^{-n} \, (1+o(1)),
\]
for $n$ multiple of~$\gcd(S)$, while $\coeff{x^n}{\gf[P][^S_k][x,1]} = 0$ if $n$ is not a multiple of~$\gcd(S)$.
More precisely, the constant~$\Lambda_S$ can be expressed as
\[
\Lambda_S \eqdef \frac{\gcd(S) \, \maxPotentialMatchings{k,S}}{2 \, \Gamma\big(\frac{\maxPotentialMatchings{k,S}}{2}+1\big)} \sum_K \frac{{\tau_S}^{n_1(K)}}{n(K)} \prod_{i > 1} \bigg( \frac{{\rho_S}^i \, \beta_S \, (2i-5)!!}{2^{i-1} \, (i-1)!} \bigg)^{n_i(K)},
\]
where we sum over the $k$-core partitions~$K$ of~$\F[P]^S$ which maximize the potential function~$\potentialMatchings(K)$.
\end{proposition}

\begin{proof}
We exploit the composition scheme obtained in Proposition~\ref{prop:polynomialPartitionsFixedSizes} and the description of the singular behavior of~$\gf[P][^S_0][x,1]$ obtained in Proposition~\ref{prop:singularBehaviorPS0}.
In the same lines as the proof of Proposition~\ref{prop:asymptMatchings}, we obtain
\[
\gf[P][^S_k][x,1] \stackbin[x \sim \rho_S]{}{=} \,\frac{1}{2} \!\!\!\! \sum_{\substack{K \; k\text{-core} \\ \text{partition of } \F[P]^S}} \!\!\! \frac{\potentialMatchings(K)\, {\tau_S}^{n_1(K)}}{n(K)} \prod_{i > 1} \bigg( \frac{{\rho_S}^i \, \beta_S \, (2i-5)!!}{2^{i-1} \, (i-1)!} \bigg)^{n_i(K)} X^{-\potentialMatchings(K)-2} \, + O\left(X^{-\potentialMatchings(K)-1}\right),
\]
where~$X \eqdef \sqrt{1 - \frac{x}{\rho_S}}$.
This expansion is valid in a domain dented at $X=\rho_S$.
The asymptotic behavior of this sum is therefore guided by the $k$-core partitions~$K$ of~$\F[P]^S$ which maximize the potential~$\potentialMatchings(K)$.
Finally, the asymptotic of~$\coeff{x^n}{\gf[P][^S_k][x,1]}$ is obtained combining the contributions of all the singularities~$\set{\rho_S \cdot \xi}{\xi \in \C, \, \xi^{\gcd(S)} = 1}$ of the function~$\gf[P][^S_k][x,1]$.
\end{proof}

Given an arbitrary subset~$S$ of~$\N^*$, it is in general difficult to describe the $k$-core partitions of~$\F[P]^S$ which maximize the corresponding potential~$\potentialMatchings$.
We close this section with two relevant examples that we partially used as prototypes of our results, and for which we can explicitly describe the maximal partitions.

\begin{example}
\label{exm:uniformPartitions}
Let~$q \ge 2$.
Consider \defn{$q$-uniform partitions}, for which~$S=\{q\}$.
We have
\[
\tau_{\{q\}} = \left(\frac{1}{q-1}\right)^{\frac{1}{q}},
\quad
\rho_{\{q\}} = \frac{q-1}{q} \left(\frac{1}{q-1}\right)^{\frac{1}{q}},
\quad
\alpha_{\{q\}} = \frac{q}{q-1},
\quad\text{and}\quad
\beta_{\{q\}} = \sqrt{\frac{2q^2}{(q-1)^3}}.
\]
Therefore, the asymptotic behavior of the number of $q$-uniform non-crossing partitions with~$qm$ vertices is given by
\[
\coeff{x^{qm}}{\gf[P][^{\{q\}\!}_0][x,1]} \stackbin[m \to \infty]{}{=} \sqrt{\frac{q}{2\pi(q-1)^3}} \, m^{-\frac{3}{2}} \left(\frac{q^q}{(q-1)^{q-1}}\right)^m (1 + o(1)).
\]
Assume now that we are interested in $q$-uniform partitions with~$k$ crossings, where~$k = k'(q-1)^2$ is a multiple of~$(q-1)^2$.
The maximal $q$-uniform $k$-core partitions are formed by $k'$ pairs of crossing blocks as illustrated in \fref{fig:maximalUniformCorePartitions} for~$q = 3$ and~$k = 8$.
When~$q = 2$, there are only $4$ maximal $k$-core matchings corresponding to the $4$ possible positions for the root, as discussed in Proposition~\ref{prop:asymptMatchings}.
In contrast, when~$q \ge 3$, we have~$(2q)^{k'+1}$ maximal $q$-uniform $k$-core partitions corresponding on the one hand to the $2q$ possible positions for the root, and on the other hand to the relative positions of the two blocks in each of the $k'$ pairs (see \fref{fig:maximalUniformCorePartitions} for some examples).
\begin{figure}
  \centerline{\includegraphics[width=.9\textwidth]{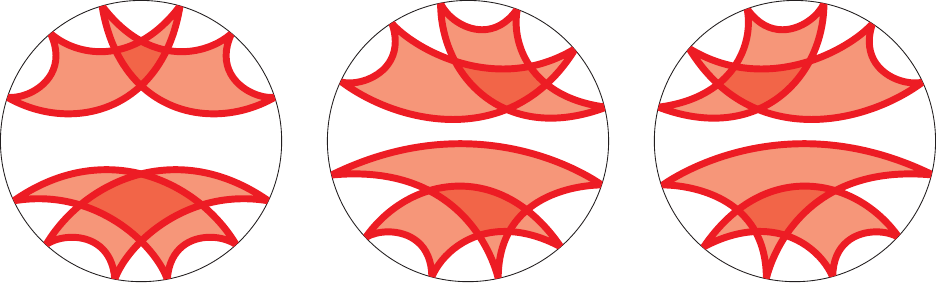}}
  \caption{Some maximal $3$-uniform $8$-core partitions (unrooted).}
  \label{fig:maximalUniformCorePartitions}
\end{figure}
All these maximal $q$-uniform $k$-core partitions have the same parameters: $n_1(K) = k'(2q-1)$, $n_{k'}(K) = 1$, and~$n_i(K) = 0$ for all~$i \notin \{1,k'\}$, and thus~$n(K) = 2qk'$.
It follows that the maximal potential is~$\maxPotentialMatchings{k,\{q\}} = 2k'-3$.
Therefore, for any~$q \ge 3$, the asymptotic behavior of the number of $q$-uniform partitions with $k = k'(q-1)^2$ crossings and~$qm$ vertices is given by
\[
\coeff{x^{qm}}{\gf[P][^{\{q\}\!}_k][x,1]} \stackbin[m \to \infty]{}{=} \frac{\sqrt{2} \, (2k'-3)!! \; q^{k'+\frac{1}{2}}}{k'! \; \Gamma\big(k'-\frac{1}{2}\big) \, (q-1)^{k'+\frac{3}{2}}} \, m^{k'-\frac{3}{2}} \, \left(\frac{q^q}{(q-1)^{q-1}}\right)^m (1 + o(1)).
\]
We obtain the estimate of Proposition~\ref{prop:asymptMatchings} if we plug-in~$q=2$ and~$k'=k$ in this equation and divide it by~$4^k$ (since there are only $4$ maximal $k$-core matchings).
\end{example}

\begin{remark}
Contrarily to what happens for other combinatorial classes, the exponent~$\frac{\maxPotentialMatchings{k,S}}{2}$ of the polynomial growth in~$\coeff{x^n}{\gf[P][^S_k][x,1]}$ is not a constant of the class.
For $q$-uniform partitions and for~$k = k'(q-1)^2$, we have obtained~$\maxPotentialMatchings{k,S} = 2k'-3$.
In fact, for $3$-uniform partitions, we even have $\maxPotentialMatchings{4k',\{3\}} = \maxPotentialMatchings{4k'+2,\{3\}} = 2k'-3$, illustrating that the function~$\maxPotentialMatchings{k,S}$, and thus the exponent of the polynomial growth can have unexpected behaviors.
\end{remark}

\begin{example}
\label{exm:multiplePartitions}
Let~$q \ge 1$.
Consider \defn{$q$-multiple partitions}, for which~$S = q\N^*$.
Since
\[
\sum_{n \ge 1 } (qn-1) \, x^{qn} = \sum_{n \ge 1} qn \, x^{qn} - \sum_{n \ge 1} x^{qn} = \frac{q \, x^q}{(1-x^q)^2} - \frac{x^q}{1-x^q} = 1 + \frac{(q+1) x^q-1}{(1-x^q)^2},
\]
we obtain
\[
\tau_{q\N^*} = \left(\frac{1}{q+1}\right)^{\frac{1}{q}},
\quad
\rho_{q\N^*} = \frac{q}{q+1} \left(\frac{1}{q+1}\right)^{\frac{1}{q}},
\quad
\alpha_{q\N^*} = \frac{q+1}{q},
\quad\text{and}\quad
\beta_{q\N^*} = \sqrt{\frac{2 \, (q+1)}{q^2}}.
\]
Therefore, the asymptotic behavior of the number of $q$-uniform non-crossing partitions with~$qm$ vertices is given by
\[
\coeff{x^{qm}}{\gf[P][^{q\N^*\!\!}_0][x,1]} \stackbin[m \to \infty]{}{=} \sqrt{\frac{q+1}{2\pi q^3}} \, m^{-\frac{3}{2}} \left(\frac{(q+1)^{q+1}}{q^q}\right)^m (1 + o(1)).
\]
Assume now that~$q \ge 3$ and that we are interested in $q$-multiple partitions with~$k$ crossings, where~$k = k'(q-1)^2$ is a multiple of~$(q-1)^2$.
The maximal $q$-multiple $k$-core partitions are precisely the maximal $q$-uniform $k$-core partitions, illustrated in \fref{fig:maximalUniformCorePartitions} for~$q = 3$ and~$k = 8$.
Since~$q \ge 3$, there are~$(2q)^{k'+1}$ such $k$-core partitions, and they all have the same parameters: $n_1(K) = k'(2q-1)$, $n_{k'}(K) = 1$, and~$n_i(K) = 0$ for all~$i \notin \{1,k'\}$, and thus~$n(K) = 2qk'$.
It follows that the maximal potential is~$\maxPotentialMatchings{k,\{q\}} = 2k'-3$.
Therefore, for any~$q \ge 3$, the asymptotic behavior of the number of $q$-multiple partitions with $k = k'(q-1)^2$ crossings and~$qm$ vertices is given by
\[
\coeff{x^{qm}}{\gf[P][^{q\N^*\!\!}_k][x,1]} \stackbin[m \to \infty]{}{=} \frac{\sqrt{2} \, (2k'-3)!! \; q^{3k'-\frac{5}{2}}}{k'! \; \Gamma\big(k'-\frac{1}{2}\big) \, (q+1)^{3k'-\frac{1}{2}}} \, m^{k'-\frac{3}{2}} \, \left(\frac{(q+1)^{q+1}}{q^q}\right)^m (1 + o(1)).
\]
To obtain the estimate for even partitions with $k$ crossings and~$2m$ vertices, we plug-in~$q=2$ and~$k'=k$ in this equation and divide it by~$4^k$ (since there are only $4$ maximal $k$-core~matchings).
\end{example}


\section{Chord and hyperchord diagrams}
\label{sec:diagrams}

In this section, we consider the family~$\F[D]$ of all chord diagrams on the unit circle.
Remember that a chord diagram is given by a set of vertices on the unit circle, and a set of chords between them.
In particular, we allow isolated vertices, as well as several chord incident to the same vertex, but not multiple chords with the same two endpoints.
We let~$\Fcoeff[D][n,m,k]$ denote the set of chord diagrams in~$\F[D]$ with $n$ vertices, $m$ chords, and $k$ crossings.
We define the generating functions
\[
\gf[D][][x,y,z] \eqdef \sum_{n,m,k \in \N} |\Fcoeff[D][n,m,k]| \, x^n y^m z^k \qquad \text{and} \qquad \gf \eqdef \coeff{z^k}{\gf[D][][x,y,z]},
\]
of chord diagrams, and chord diagrams with $k$ crossings, respectively.

\begin{remark}
We insist on the fact that we allow here for isolated vertices in chord diagrams.
However, it is essentially equivalent to enumerate chord diagrams or chord configurations (meaning chord diagrams with no isolated vertices).
Indeed, their generating functions are related by
\[
\gf[C][][x,y,z] = \frac{1}{1+x} \, \gf[D][][\frac{x}{1+x},y,z] \qquad \text{and} \qquad \gf[C][_k][x,y] = \frac{1}{1+x} \, \gf[D][_k][\frac{x}{1+x},y].
\]
\end{remark}


\subsection{Warming up: crossing-free chord diagrams}
\label{subsec:crossingFreeChordDiagrams}

The generating function $\gfo$ of crossing-free chord diagrams was studied in~\cite{FlajoletNoy-nonCrossing}.
We repeat here their analysis for the convenience of the reader and since we will use similar decomposition schemes later for our extension to hyperchord diagrams.

\begin{proposition}[{\cite[Equation (22)]{FlajoletNoy-nonCrossing}}]
\label{prop:crossingFreeDiagrams}
The generating function~$\gfo$ of crossing-free chord diagrams satisfies the functional equation
\begin{equation}
\label{eq:crossingFreeDiagrams}
y \, \gfo^2 + \big( x^2(1+y) - x(1+2y) - 2y \big) \, \gfo + x(1+2y)+y  = 0.
\end{equation}
\end{proposition}

\begin{proof}
Consider first a connected crossing-free chord diagram~$C$.
By connected we mean here that $C$ is connected as a graph.
Call \defn{principal} the chords of~$C$ incident to its first vertex (the first after its root).
These principal chords split~$C$ into smaller crossing-free chord diagrams:
\begin{enumerate}[(i)]
\item the first (before the first principal chord) and last (after the last principal chord) subdiagrams are both connected chord diagrams,
\item each subdiagram inbetween two principal chords consists either in a connected diagram (but not a single vertex), or in two connected diagrams.
\end{enumerate}
This decomposition scheme is illustrated on \fref{fig:decompositionCrossingFreeChordDiagram}.
\begin{figure}[b]
  \centerline{\includegraphics[width=.8\textwidth]{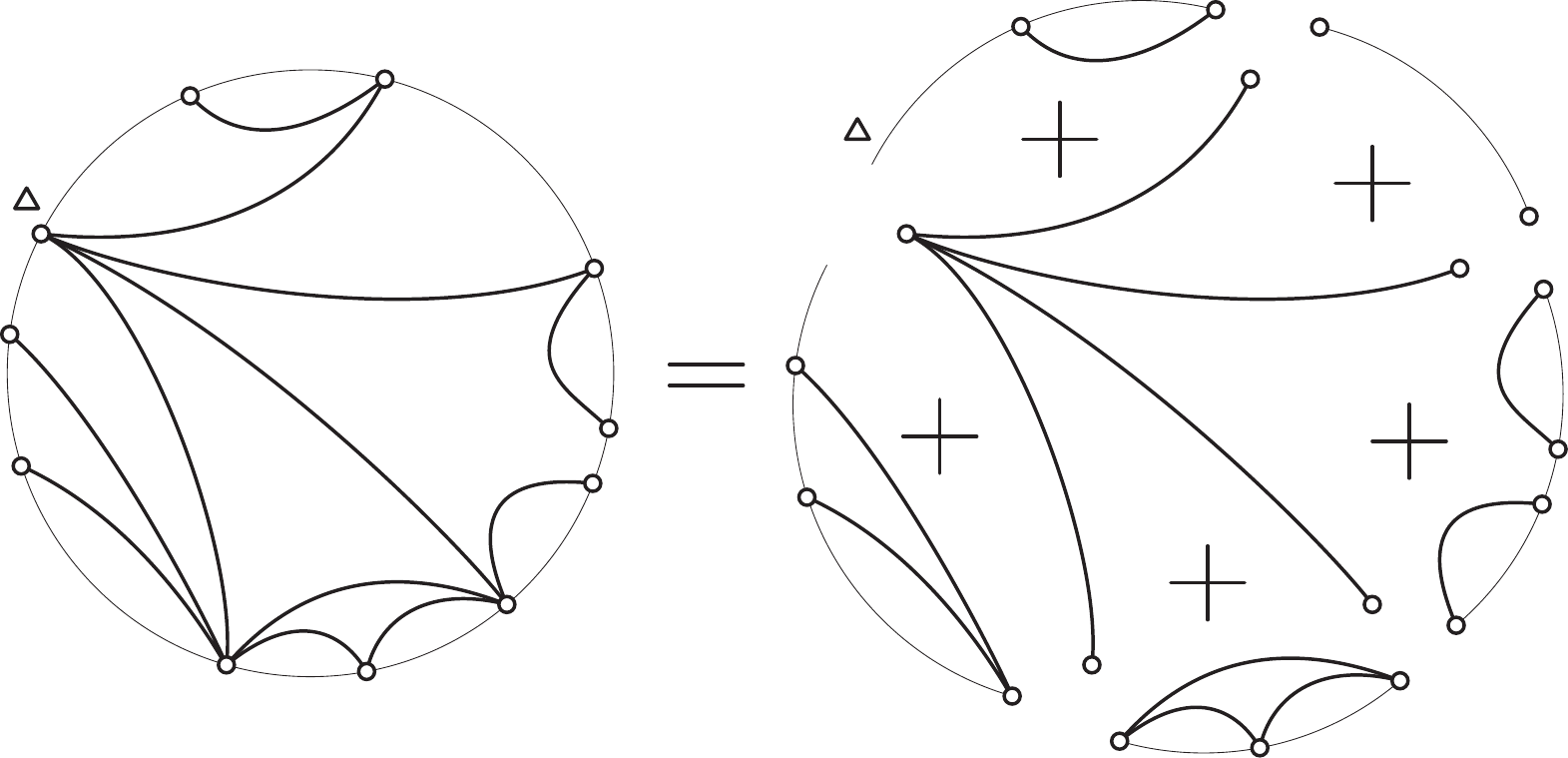}}
  \caption{Decomposition scheme of connected crossing-free chord diagrams.}
  \label{fig:decompositionCrossingFreeChordDiagram}
\end{figure}
This leads to the following functional equation on the generating function~$\gfo[CD][x,y]$ of connected crossing-free chord diagrams:
\begin{equation}
\label{eq:conChordDiag}
\gfo[CD][x,y] = x \left(1 + \frac{y \, \gfo[CD][x,y]^2}{x - y \left(\gfo[CD][x,y] - x + \gfo[CD][x,y]^2 \right)}\right),
\end{equation}
which can be rewritten as
\begin{equation}
\label{eq:connectedChordDiagrams}
y \, \gfo[CD][x,y]^3 + y \, \gfo[CD][x,y]^2 - x(1+2y) \, \gfo[CD][x,y] + x^2(1+y) = 0.
\end{equation}
Finally, since a crossing-free chord diagram can be decomposed into connected crossing-free chord diagrams, we have
\begin{equation}
\label{eq:general-connectedChordDiagrams}
\gfo = 1 + \gfo[CD][x\gfo, y].
\end{equation}
Using this equation to eliminate~$\gfo[CD][x,y]$ in Equation~\eqref{eq:connectedChordDiagrams} leads to the desired formula after straightforward simplifications.
\end{proof}

In the following statement, we exploit the implicit expression of Equation~\eqref{eq:crossingFreeDiagrams} to obtain the dependence of $\diff[i] \gfo$ with respect to $\gfo$.

\begin{proposition}
\label{prop:rationalityAllDiagrams}
All derivatives~$\diff[i] \gfo$ are rational functions in~$\gfo$ and $x$.
\end{proposition}

\begin{proof}
Derivating (with respect to~$x$) the functional Equation~\eqref{eq:crossingFreeDiagrams} of Proposition~\ref{prop:crossingFreeDiagrams}, we obtain that
\[
\diff \gfo = \frac{\big( (1+2y) - 2(1+y)x \big) \, \gfo - 1 - 2y}{2y \, \gfo + (1+y)x^2 - (1+2y)x - 2y}.
\]
Eliminating~$y$ from Equation~\eqref{eq:crossingFreeDiagrams}, we obtain
\[
\diff \gfo = -\frac{(2x-1) \, \gfo^3 - (3x^2 - x - 1) \, \gfo^2 + (x-1) \, \gfo + 1}{(x^2-x) \, \gfo + (x^3 - 3x^2 +4x) \, \gfo - x(x^2+x+1)},
\]
which is rational in~$\gfo$ and $x$.
The result thus follows by successive derivations.
\end{proof}

As we are also interested in asymptotic estimates, we proceed to study the singular behavior of $\gfo$.
As it is proved in~\cite{FlajoletNoy-nonCrossing}, the generating function $\gfo$ has a unique square-root singularity when $y$ varies around $y=1$:
\begin{equation}
\label{eq:sing-D}
\gfo \stackbin[\substack{y \sim 1 \\ x \sim \rho(y)}]{}{=} d_0(y) - d_1(y) \sqrt{1-\frac{x}{\rho(y)}}  + O \left(1 - \frac{x}{\rho(y)}\right),
\end{equation}
uniformly with respect to $y$ for $y$ in a small neighborhood of $1$, and with $d_0(y)$, $d_1(y)$ and $\rho(y)$ analytic at $y=1$.
In fact, when $y=1$ we obtain the singular expansion
\[
\gfo[D][x,1] \stackbin[x \sim \rho(1)]{}{=} -1+3\frac{\sqrt{2}}{2} - \frac{1}{2} \sqrt{-140+99\sqrt{2}} \sqrt{1 - \frac{x}{\rho(1)}}  + O\left( 1 - \frac{x}{\rho(1)} \right),
\]
with~$\rho(1) \eqdef \rho =\frac{3}{2}-\sqrt{2} \simeq 0.08578$.
This is valid in a domain dented at $x = \rho$.
In particular, Equation~\eqref{eq:sing-D} shows that the singular behavior of~$\diff[i] \gfo$ in a neighborhood of $x = \rho(y)$ is of the form
\[
\diff[i] \gfo \stackbin[\substack{y \sim 1 \\ x \sim \rho(y)}]{}{=} \frac{d_1(y) \, (2i-3)!!}{\rho(y)^i \, 2^i} \bigg(1 - \frac{x}{\rho(y)} \bigg)^{\frac{1}{2}-i}  + O\left( \left(1 - \frac{x}{\rho(y)}\right)^{1-i} \right),
\]
where we use again the convention that $(-1)!!=1$ in order to simplify formulas when~$i=1$.
This singular expansion will be exploited later in order to get both asymptotic estimates and the limit law for the number of vertices when fixing the number of crossings.
Finally, we also need the following values, which appear in~\cite[Table 5]{FlajoletNoy-nonCrossing},
\begin{equation}
\label{eq:edges-corediagrams}
-\frac{\rho'(1)}{\rho(1)}= \frac{1}{2}+\frac{\sqrt{2}}{2} \qquad\text{and}\qquad -\frac{\rho''(1)}{\rho(1)}-\frac{\rho'(1)}{\rho(1)}+\left(\frac{\rho'(1)}{\rho(1)}\right)^2=\frac{1}{4}+\frac{\sqrt{2}}{8}.
\end{equation}


\subsection{Core diagrams}
\label{subsec:coreDiagrams}

We now consider chord diagrams with $k$ crossings.
As in the previous section, we study them focussing on their cores.

\begin{definition}
A \defn{core diagram} is a chord diagram where each chord is involved in a crossing.
It is a \defn{$k$-core diagram} if it has exactly $k$ crossings.
The \defn{core}~$\core{D}$ of a chord diagram~$D$ is the subdiagram of~$D$ formed by all its chords involved in at least one crossing.
See \fref{fig:diagram}.
\end{definition}

\begin{figure}[h]
  \centerline{\includegraphics[scale=.75]{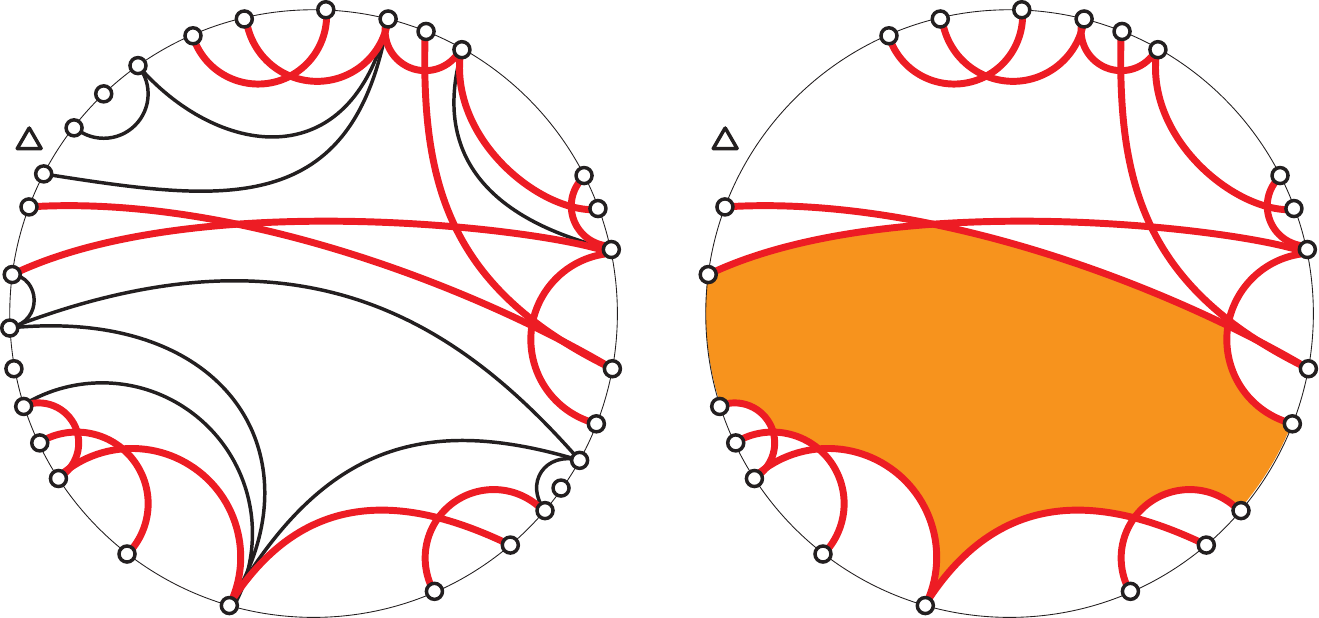}}
  \caption{A chord diagram~$D$ with $10$ crossings (left) and its $10$-core~$\core{D}$ (right). The core diagram~$\core{D}$ has $n(\core{D}) = 21$ vertices, $m(\core{D}) = 14$ chords, and $k(\core{D}) = 10$ crossings. The shaded region of~$\core{D}$ had $2$ boundary arcs and $1$ peak.}
  \label{fig:diagram}
\end{figure}

Let~$K$ be a core diagram.
We let~$n(K)$ denote its number of vertices, $m(K)$ denote its number of chords, and $k(K)$ denote its number of crossings.
We call \defn{regions} of~$K$ the connected components of the complement of~$K$ in the unit disk.
A region has~$i$ \defn{boundary arcs} and $j$ \defn{peaks} if its intersection with the unit circle has~$i$ connected arcs and $j$ isolated points.
We let~$n_{i,j}(K)$ denote the number of regions of~$K$ with $i$ boundary arcs and $j$ peaks, and we set~$\b{n}(K) \eqdef (n_{i,j}(K))_{i,j \in [k]}$.
Note that~$n(K) = \sum_{i,j} i n_{i,j}(K)$.
See again \fref{fig:diagram} for an illustration.

Since a crossing only involves two chords, a $k$-core diagram can have at most $2k$ chords.
This immediately implies the following crucial lemma.

\begin{lemma}
There are only finitely many $k$-core diagrams.
\end{lemma}

\begin{definition}
We encode the finite list of all possible $k$-core diagrams~$K$ and their parameters~$n(K)$, $m(K)$, and~$\b{n}(K) \eqdef (n_{i,j}(K))_{i,j \in [k]}$ in the \defn{$k$-core diagram polynomial}
\[
\gf[KD][_k][\b{x},y] \eqdef \gf[KD][_k][x_{i,j}, y] \eqdef \sum_{\substack{K\;k\text{-core} \\ \text{diagram}}} \frac{\b{x}^{\b{n}(K)} y^{m(K)}}{n(K)} \eqdef \sum_{\substack{K\;k\text{-core} \\ \text{diagram}}} \frac{1}{n(K)}\prod_{i,j \ge 0} {x_{i,j}}^{n_{i,j}(K)} \, y^{m(K)}.
\]
\end{definition}

\begin{example}
\label{exm:2coreDiagrams}
\fref{fig:2CoreDiagrams} represents all $2$-core diagrams, forgetting the position of their roots.
From this exhaustive enumeration, we can compute the $2$-core diagram polynomial:
\[
\gf[KD][_2][\b{x},y] = \frac{1}{2} \, {x_{1,0}}^6 \, x_{2,0} \, y^4 + {x_{1,0}}^6 \, x_{1,1} \, y^4 + \frac{1}{2} \, {x_{0,2}} \, {x_{1,0}}^6 \, y^4 + \frac{1}{2} \, {x_{1,0}}^6 \, y^3 + {x_{0,1}} \, {x_{1,0}}^5 \, y^3.
\]
\begin{figure}[h]
  \centerline{\includegraphics[width=\textwidth]{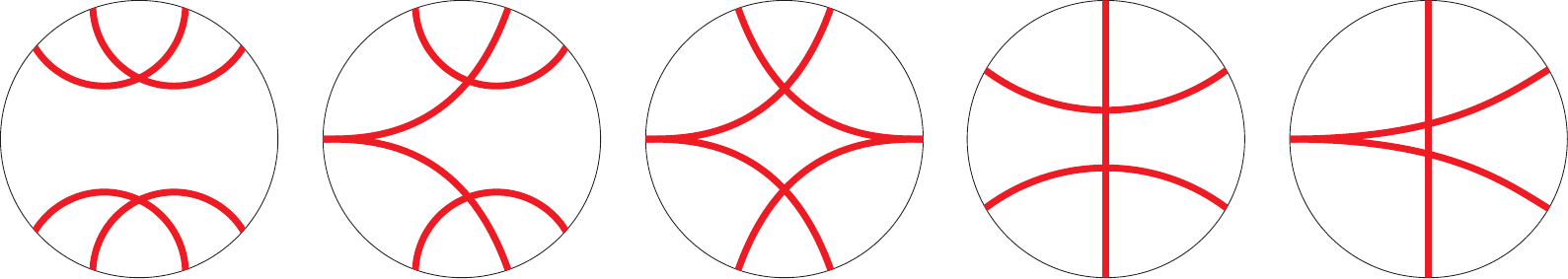}}
  \caption{The $2$-core diagrams (unrooted).}
  \label{fig:2CoreDiagrams}
\end{figure}
\end{example}


\subsection{Computing core diagram polynomials}
\label{subsec:computingCoreDiagramPolynomials}

The algorithm presented in Section~\ref{subsec:connectedMatchings} to generate all small connected core matchings can be adapted to core diagrams.
Call a chord diagram \defn{crossing connected} if its crossing graph is connected, and call \defn{crossing connected components} of a chord diagram its maximal crossing connected subdiagrams.
As in Section~\ref{subsec:connectedMatchings}, we generate all crossing connected diagrams, starting from a single arc and adding arcs one by one in the order given first by their level, and second by their left endpoint.
The essential difference here is that we allow the new constructed arc to start at an existing endpoint~$p$, as soon as it stays nested in all the arcs with left endpoint~$p$.
Details are left to the reader.
Using this algorithm, we have computed the number of crossing connected diagrams with $k$ crossings, $n$ vertices, and $m$ chords for the first values of $k$, $n$,~and~$m$.
Tables~\ref{table:connectedDiagrams1} and~\ref{table:connectedDiagrams2} give $2$-dimensional projections of these $3$-parameter sequences of values.
\begin{table}[h]
\begin{center}
\begin{tabular}{c|cccccc}
\backslashbox{$n$}{$k$} 	
	  & 1 & 2 & 3  & 4   & 5    & 6     \\
\hline
4 	  & 1 &   &    &     &      &       \\
5 	  &   & 5 & 5  &     & 1    &       \\
6 	  &   & 3 & 31 & 54  & 51   & 39    \\
7 	  &   &   & 35 & 231 & 532  & 784   \\
8 	  &   &   & 12 & 346 & 1942 & 5253  \\
9 	  &   &   &    & 225 & 3366 & 17631 \\
10 	  &   &   &    & 55  & 3062 & 33300 \\
11 	  &   &   &    &     & 1430 & 37312 \\
12 	  &   &   &    &     & 273  & 24804 \\
13 	  &   &   &    &     &      & 9100  \\
14 	  &   &   &    &     &      & 1428  \\
\hline
Total & 1 & 8 & 83 & 911 & 10657 & 129651
\end{tabular}
\end{center}
\medskip
\caption{The numbers of crossing connected diagrams with $k$ crossings and $n$ vertices.}
\label{table:connectedDiagrams1}
\vspace{-.8cm}
\end{table}

\begin{table}[h]
\begin{center}
\begin{tabular}{c|ccccccc}
\backslashbox{$m$}{$k$} 	
	  & 1 & 2 & 3  & 4   & 5    & 6      & 7       \\
\hline
2 	  & 1 &   &    &     &      &        &         \\
3 	  &   & 8 & 1  &     &      &        &         \\
4 	  &   &   & 82 & 43  & 11   & 1      &         \\
5 	  &   &   &    & 868 & 920  & 590    & 243     \\
6 	  &   &   &    &     & 9726 & 15524  & 15904   \\
7 	  &   &   &    &     &      & 113536 & 243040  \\
8 	  &   &   &    &     &      &        & 1366570 \\
\hline
Total & 1 & 8 & 83 & 911 & 10657 & 129651 & 1625757
\end{tabular}
\end{center}
\medskip
\caption{The numbers of crossing connected diagrams with $k$ crossings and $m$ chords.}
\label{table:connectedDiagrams2}
\vspace{-.5cm}
\end{table}

Once we have the tables of crossing connected diagrams, we can compute the $k$-core diagram polynomials using a similar method as in Section~\ref{subsec:computingCoreMatchingPolynomials}.
We consider the family~$\F[R]$ of rooted embedded trees where each internal vertex has at least four leaves, including its first and last children.
(Compared to the case of matchings, we just drop the condition that the internal vertices have even degree, since diagrams can have an odd number of vertices.)
From a given forest of trees in~$\F[R]$, we can construct a core diagram~$K$ by
\begin{enumerate}[(i)]
\item replacing each vertex of the forest by a crossing connected diagram, and
\item merging an arbitrary subset of pairs of consecutive vertices of the resulting core diagram which belong to two distinct crossing connected components.
\end{enumerate}
Reciprocally, given a core diagram~$K$, we reconstruct the corresponding forest by
\begin{enumerate}[(i)]
\item splitting the vertices which belong to different crossing connected components of~$K$, and
\item replacing each crossing connected component of~$K$ by a vertex, and joining these vertices into trees according to the nested structure of~$K$.
\end{enumerate}
Using similar notations as in Section~\ref{subsec:computingCoreMatchingPolynomials}, this decomposition leads to the following formulas
\begin{gather*}
\gf[R][][\b{x},\b{t}] = \sum_{j \ge 4} t_j \bigg( \sum_{i \ge 1} x_i \, \gf[R][][\b{x},\b{t}]^{i-1} \bigg)^{j-1} \qquad \text{and} \\
\gf[KD][][\b{x},y,z] = \int_{s=0}^{s=1} \frac{1}{s} \sum_{i \ge 1} \sum_{j \le i} \binom{i}{j} \, x_{i-j,j} \, s^{i-j} \, \gf[R][][x_p \leftarrow \sum_{q \le p} \binom{p}{q} \, x_{p-q} \, s^{p-q}, t_p \leftarrow {\gf[CD][^p][y,z]}]^i ds.
\end{gather*}
As for matchings, this provides an effective method to compute $k$-core diagrams.


\subsection{Generating function of chord diagrams with $k$ crossings}
\label{subsec:generatingFunctionDiagrams}

In this section, we express the generating function~$\gf[D][_k][x,y]$ of chord diagrams with $k$ crossings as a rational function of the generating function~$\gfo$ of crossing-free diagrams, using the $k$-core diagram polynomial~$\gf[KD][_k][\b{x},y]$ defined in the previous section.

First, we say that a chord diagram~$D$ is~\defn{weakly rooted} if we have marked an arc between two consecutive vertices of its core~$\core{K}$.
Again, we have the following rerooting lemma.

\begin{lemma}
For any core diagram~$K$, the number~$D_K(n,m)$ of rooted chord diagrams with $n$ vertices, $m$ chords, and core~$K$ and the number~$\bar D_K(n,m)$ of weakly rooted chord diagrams with $n$ vertices, $m$ chords, and core~$K$ are related by
\[
n(K) D_K(n,m) = n \bar D_K(n,m).
\]
\end{lemma}

As for matchings, we can now construct any chord diagram with $k$ crossings by inserting crossing-free subdiagrams in the regions left by its $k$-core.
We can therefore derive the following expression for the generating function~$\gf$ of diagrams with $k$ crossings, in terms of the generating function~$\gfo$ of crossing-free diagrams, of the $k$-core diagram polynomial~$\gf[KD][_k][x,y]$, and of the polynomials
\[
\gf[D][_0^n][y] \eqdef \coeff{x^n}{\gfo} \qquad \text{and} \qquad \gf[D][_0^{\le p}][x,y] \eqdef \sum_{n \le p} \gf[D][_0^n][y] x^n = \sum_{\substack{n \le p \\ m \ge 0}} |\Fcoeff[D][n,m,0]| \, x^n y^m.
\]

\begin{proposition}
\label{prop:generatingFunctionDiagrams}
For any~$k \ge 1$, the generating function~$\gf$ of chord diagrams with $k$ crossings is given by
\[
\gf = x\diff\gf[KD][_k][x_{0,j} \leftarrow \frac{\gf[D][_0^{j}][y]}{x^j}, \, x_{i,j} \leftarrow {\frac{x^i}{(i-1)!} \diff[i-1] \frac{\gfo[D][x,y] - \gf[D][_0^{\le i+j}][x,y]}{x^{i+j+1}}}, \, y].
\]
In particular, $\gf[D][_k][x,y]$ is a rational function of~$\gfo$ and~$x$.
\end{proposition}

\begin{proof}
Consider a rooted crossing-free chord diagram~$D$, whose vertices are labeled from~$1$ to~$n$ clockwise starting from the root.
Let~$\b{j} \eqdef (j_1,\dots,j_i)$ be a list of~$i$ positive integers whose sum is~$j$.
We say that~$D$ is \defn{$\b{j}$-marked} if we have marked $i$ vertices of~$D$, including the first vertex labeled~$1$, in such a way that there is at least $j_k+1$ vertices between the $k$\ordinal{} and $(k+1)$\ordinal{} marked vertices, for any~$k \in [i]$.
More precisely, if we mark the vertices labeled by~$1 = \alpha_1 < \dots < \alpha_i$ and set by convention~$\alpha_{i+1} = n+1$, then we require that $\alpha_{k+1} - \alpha_k > j_k + 1$ for any~$k \in [i]$.
Note that if~$D$ has less than~$2i+j$ vertices, then it cannot be $\b{j}$-marked.
Otherwise, if~$D$ has at least~$2i+j$ vertices, we have $\binom{n-i-j-1}{i-1}$ ways to place these $i$ marks.
Therefore, the generating function of the rooted $\b{j}$-marked crossing-free chord diagrams is given by
\[
\frac{x^{2i+j}}{(i-1)!} \diff[i-1] \frac{\gfo - \gf[D][_0^{\le i+j}][x,y]}{x^{i+j+1}}.
\]

\enlargethispage{.2cm}
Consider now a weakly rooted chord diagram~$D$ with $k$ crossings.
We decompose this diagram into several subdiagrams as follows.
On the one hand, the core~$\core{D}$ contains all crossings of~$D$.
This core is rooted by the root of~$D$.
On the other hand, each region~$R$ of~$\core{D}$ contains a crossing-free subdiagram~$D_R$.
We root this subdiagram~$D_R$ as follows:
\begin{enumerate}[(i)]
\item if the root of~$D$ points out of~$R$, then $D_R$ is just rooted by the root of~$D$;
\item otherwise, $D_R$ is rooted on the first boundary arc of~$\core{D}$ before the root of~$D$ in clockwise direction.
\end{enumerate}
Moreover, we mark the first vertex of each boundary arc of~$R$.
Note that we do not mark the peaks.
Thus, if the region~$R$ has $i$ boundary arcs, and if the $k$\ordinal{} and $(k+1)$\ordinal{} boundary arcs of~$R$ are separated by~$j_k$ peaks, then we obtain in this region~$R$ of~$\core{D}$ a rooted $(j_1,\dots,j_i)$-marked crossing-free subdiagram~$D_R$.
See \fref{fig:decompositionDiagram}.
Observe that their is a difference of behavior between
\begin{enumerate}[(i)]
\item the regions~$R$ with no boundary arcs and only~$j$ peaks, which are filled in by a crossing-free chord diagram~$D_R$ on precisely $j$ vertices, and
\item the regions~$R$ with at least one boundary arc, whose corresponding chord diagram~$D_R$ can have arbitrarily many additional vertices.
\end{enumerate}

\begin{figure}
  \centerline{\includegraphics[scale=.5]{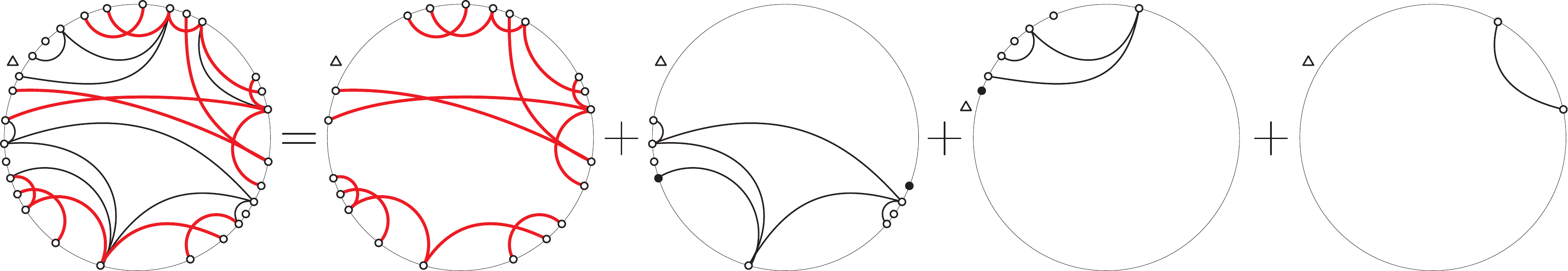}}
  \caption{The decomposition of the diagram of \fref{fig:diagram}\,(left) by its core into rooted marked subdiagrams. The root is represented by~$\vartriangle$ and the marked vertices are represented by~$\bullet$. Only non-empty subdiagrams are  represented.}
  \label{fig:decompositionDiagram}
\end{figure}

Reciprocally we can reconstruct the chord diagram~$D$ from its rooted core~$\core{D}$ and its rooted and marked crossing-free subdiagrams~$D_R$.
We thus obtain that the generating function~$\gf$ from the $k$-core diagram polynomial~$\gf[KD][_k][\b{x},y]$ by replacing a region with~$i \ne 0$ boundary arcs and~$j$ peaks by
\[
\frac{x^i}{(i-1)!} \diff[i-1] \frac{\gfo[D][x,y] - \gf[D][_0^{\le i+j}][x,y]}{x^{i+j+1}},
\]
and a region with no boundary arcs but $j$ peaks by $\gf[D][_0^j][y]/x^j$.
This is precisely the formula stated in the proposition.

The rationality of this function thus follows from Proposition~\ref{prop:rationalityAllDiagrams}, since~$\gf[D][_0^j]$ and $\gf[D][_0^{\le j}][x,y]$ are both polynomials in~$x$ and~$y$, and $y$ can be eliminated as in the proof of Proposition~\ref{prop:rationalityAllDiagrams}.
\end{proof}

\enlargethispage{-.5cm}
\begin{example}
Using the $2$-core diagram polynomial~$\gf[KD][_2][\b{x},y]$ computed in Example~\ref{exm:2coreDiagrams} (see also \fref{fig:2CoreDiagrams}), we can compute the generating function of chord diagrams with two crossings.
Although we do not include here the long and meaningless expression of this function, we provide the first few terms in its development:
\begin{align*}
\gf[D][_2][x,y] = & \quad \; x^5 \, (5 \, y^3 + 25 \, y^4 + 50 \, y^5 + 50 \, y^6 + 25 \, y^7 + 5 \, y^8) \\
& + x^6 \, (33 \, y^3 + 231 \, y^4 + 696 \, y^5 + 1173 \, y^6 + 1200 \, y^7 + 753 \, y^8 + 276 \, y^9) \\
& + x^7 \, (126 \, y^3 + 1176 \, y^4 + 4900 \, y^5 + 11984 \, y^6 + 19012 \, y^7 + 20384 \, y^8 + 14896 \, y^9) \\
& + x^8 \, (364 \, y^3 + 4368 \, y^4 + 23856 \, y^5 + 78384 \, y^6 + 172476 \, y^7 + 267552 \, y^8 + 299712 \, y^9) \\
& + x^9 \, (882 \, y^3 + 13230 \, y^4 + 91260 \, y^5 + 383940 \, y^6 + 1101060 \, y^7 + 2277414 \, y^8 + 3503790 \, y^9) \\
& + x^{10} \, (1890 \, y^3 + 34650 \, y^4 + 293700 \, y^5 + 1530375 \, y^6 + 5493675 \, y^7 + 14419900 \, y^8 + 28648730 \, y^9) \\
& + x^{11} \, (3696 \, y^3 + 81312 \, y^4 + 829092 \, y^5 + 5220666 \, y^6 + 22797610 \, y^7 + 73432238 \, y^8 + 181038264 \, y^9) \, \dots
\end{align*}
For example, there are $23856$ chord diagrams with $8$ vertices, $5$ edges, and $2$ crossings, among which $48$, $624$, $948$, $3996$, and $18240$ use the first, second, third, fourth and fifth core in \fref{fig:2CoreDiagrams} respectively.
Can you describe them?
\end{example}


\subsection{Asymptotic analysis}
\label{subsec:asymptoticDiagrams}

Similarly to our asymptotic analysis in Section~\ref{subsec:asymptoticMatchings}, we can obtain asymptotic results for the number of chord diagrams with $k$ crossings.

\begin{proposition}
\label{prop:asymptDiagrams}
For any $k \ge 1$, the number of chord diagrams with $k$ crossings and $n$ vertices~is
\[
\coeff{x^n}{\gf[D][_k][x,1]} \stackbin[n \to \infty]{}{=} \frac{d_0(1)^{3k} \, d_1(1) \, (2k-3)!!}{(2\rho)^{k-1} \, k! \; \Gamma(k-\frac{1}{2})} \, n^{k-\frac{3}{2}} \, \rho^{-n} \, (1+o(1)),
\]
where
\[
\rho^{-1} = 6 + 4\sqrt{2}, \qquad d_0(1) =  - 1 + 3\frac{\sqrt{2}}{2}, \qquad\text{and}\qquad d_1(1) = \frac{1}{2} \sqrt{- 140 + 99\sqrt{2}}.
\]
\end{proposition}

\begin{proof}
We apply singularity analysis on the composition scheme given by Proposition~\ref{prop:generatingFunctionDiagrams}.
In our analysis, it is more convenient to express the $k$-core diagram polynomial~$\gf[KD][_k][\b{x},1]$ as
\[\gf[KD][_k][\b{x},1] = \sum_{\substack{K\;k\text{-core} \\ \text{diagram}}} \frac{1}{n(K)} \prod_{j \ge 0} {x_{0,j}}^{n_{0,j}(K)} \prod_{\substack{i \ge 1 \\ j \ge 0}} {x_{i,j}}^{n_{i,j}(K)}.
\]
The resulting expression for~$\gf[D][_k][x,1]$ is
\[
x \diff \sum_{\substack{K\;k\text{-core} \\ \text{diagram}}} \frac{1}{n(K)}\prod_{j \ge 0} {\bigg(\frac{\gf[D][_0^{j}][1]}{x^j}\bigg)}^{n_{0,j}(K)}\prod_{\substack{i \ge 1 \\ j \ge 0}} \bigg( {\frac{x^i}{(i-1)!} \diff[i-1] \frac{\gfo[D][x,1] - \gf[D][_0^{\le i+j}][x,1]}{x^{i+j+1}}} \bigg)^{n_{i,j}(K)}.
\]
Analyzing this function around~$x = \rho$ boils down to analyzing the generating function
\[
\rho \diff \sum_{\substack{K\;k\text{-core} \\ \text{diagram}}} \frac{1}{n(K)}\prod_{j \ge 0} {\bigg(\frac{\gf[D][_0^{j}][1]}{\rho^j}\bigg)}^{n_{0,j}(K)} \prod_{\substack{i \ge 1 \\ j \ge 0}} \bigg( \frac{1}{\rho^{j+1} \, (i-1)!} \diff[i-1] \gfo[D][x,1] \bigg)^{n_{i,j}(K)}.
\]
Observe that we forget the terms of the form $\gf[D][_0^{\le i+j}][x,1]$ as they are polynomials in~$x$, and thus analytic functions around~$x = \rho$.
In order to simplify the expressions, we set
\[
\xi(K) \eqdef \frac{1}{n(K)} \prod_{j \ge 0} {\bigg(\frac{\gf[D][_0^{j}][1]}{\rho^j}\bigg)}^{n_{0,j}(K)} \prod_{\substack{i \ge 1 \\ j \ge 0}} \bigg( \frac{1}{\rho^{j+1} \, (i-1)!} \bigg)^{n_{i,j}(K)}.
\]
Let $X \eqdef \sqrt{1 - \frac{x}{\rho}}$.
Developing~$\gfo[D][x,1]$ using its Puiseux's expansion~\eqref{eq:sing-D} around $x=\rho$ we obtain
\begin{align*}
\gf[D][_k][x,1]
&\stackbin[x \sim \rho]{}{=} \rho \diff \sum_{\substack{K\;k\text{-core} \\ \text{diagram}}} \xi(K) \prod_{\substack{i \ge 1 \\ j \ge 0}} \bigg( \diff[i-1]  \gfo[D][x,1] \bigg)^{n_{i,j}(K)} \\
&\stackbin[x \sim \rho]{}{=} \rho \diff \sum_{\substack{K\;k\text{-core} \\ \text{diagram}}} \xi(K) \prod_{\substack{i \ge 1 \\ j \ge 0}} \bigg( \diff[i-1] \big( d_0(1) + d_1(1) \, X \, +O\left(X^2\right) \big) \bigg)^{n_{i,j}(K)} \\
& \stackbin[x \sim \rho]{}{=} \rho \diff \sum_{\substack{K\;k\text{-core} \\ \text{diagram}}} \xi(K) \prod_{j \ge 0} d_0(1)^{n_{1,j}(K)} \prod_{\substack{i > 1 \\ j \ge 0}} \bigg( \frac{d_1(1) \, (2i-5)!!}{ 2^{i-1} \, \rho^{i-1}} \, X^{3-2i} \, + O\left(X^{4-2i}\right) \bigg)^{n_{i,j}(K)} \\
& \stackbin[x \sim \rho]{}{=} \rho \diff \sum_{\substack{K\;k\text{-core} \\ \text{diagram}}} \zeta(K) \, X^{-\potentialDiagrams(K)} \, +O\left(X^{-\potentialDiagrams(K)+1}\right) \\
& \stackbin[x \sim \rho]{}{=} \rho \sum_{\substack{K\;k\text{-core} \\ \text{diagram}}} - \zeta(K) \, \potentialDiagrams(K) \, X^{-\potentialDiagrams(K)-2} \, + O\left(X^{-\potentialDiagrams(K)-1}\right),
\end{align*}
where
\begin{gather*}
\zeta(K) \eqdef \frac{1}{n(K)} \prod_{j \ge 0} {\bigg(\frac{\gf[D][_0^{j}][1]}{\rho^j}\bigg)}^{n_{0,j}(K)} d_0(1)^{n_{1,j}(K)} \prod_{\substack{i > 1 \\ j \ge 0}} \bigg( \frac{d_1(1) \, (2i-5)!!}{2^{i-1} \, \rho^{i+j} \, (i-1)!}\bigg)^{n_{i,j}(K)}, \\
\text{and} \qquad \potentialDiagrams(K) \eqdef \sum_{\substack{i > 1 \\ j \ge 0}} (2i-3) \, n_{i,j}(K).
\end{gather*}
Following the same lines as in Section~\ref{subsec:asymptoticMatchings}, the main contribution to the asymptotic arise from the \mbox{$k$-core} diagrams which maximizes~$\potentialDiagrams(K)$.
These $k$-core diagrams satisfy~$n_{k,0}(K) = 1$, $n_{1,0}(K) = 3k$ and $n_{i,j}(K) = 0$ for all~$(i,j) \ne (k,0), (1,0)$.
Consequently $\potentialDiagrams(K)=2k-3$.
Therefore,
\[
\gf[D][_k][x,1] \stackbin[x \sim \rho]{}{=} \frac{d_0(1)^{3k} \, d_1(1) \, (2k-3)!!}{(2\rho)^{k-1} \, k!} \, X^{1-2k} \, +O\left(X^{2-2k}\right),
\]
and we conclude applying the Transfer Theorem for singularity analysis (see Theorem~\ref{theo:transfer} in Appendix~\ref{app:methodology}).
\end{proof}

Finally, with the same techniques, we can also compute the limiting distribution of the number of edges in a $k$-chord diagram with $n$ vertices, chosen uniformly at random.

\begin{theorem}
The number of edges in a chord diagram with $k$ crossings and $n$ vertices, chosen uniformly at random, follows a normal distribution with expectation~$\mu_n$ and variance~$\sigma_n$, where
\[
\mu_n = \bigg(\frac{1}{2}+\frac{\sqrt{2}}{2}\bigg) n \, (1+o(1)) \qquad\text{and}\qquad \sigma_n = \bigg(\frac{1}{4}+\frac{\sqrt{2}}{8}\bigg) n \, (1+o(1)).
\]
\end{theorem}

\begin{proof}
Direct application of the Quasi-Powers Theorem (see Theorem~\ref{theo:quasi-powers} in Appendix~\ref{app:methodology}), by means of the values computed in Equation~\eqref{eq:edges-corediagrams}.
The main contribution on the analysis arises from maximal $k$-core diagrams.
Observe that the constants defining the expectation and the variance are exactly the same as in the planar configurations.
\end{proof}


\subsection{Random generation}
\label{subsec:randomGenerationDiagrams}

In this section, we provide random generators for the combinatorial family of chord diagrams with a given number of crossings, using the methodology of Bolzmann samplers.
We proceed in three steps, obtaining random generators for:
\begin{enumerate}[(i)]
\item connected crossing-free chord diagrams,
\item all crossing-free chord diagrams,
\item chord diagrams with precisely $k$ crossings.
\end{enumerate}
Once we have a Boltzmann sampler for crossing-free chord diagrams, the design of a random generator for chord diagrams with precisely $k$ crossings follows exactly the same lines as in Section~\ref{subsec:randomGenerationMatchings}.
In this section, we therefore only discuss Steps~(i) and~(ii) above.

We first describe a Bolzmann sampler for connected crossing-free chord diagrams.
It is convenient to write Equation~\eqref{eq:conChordDiag} (with $y=1$) in the form
\[
\gfo[CD][x,1] = x \left( 1 + \frac{\gfo[CD][x,1]^2}{x} \sum_{r=0}^{\infty} \bigg( \frac{\gfo[CD][x,1] - x + \gfo[CD][x,1]^2}{x} \bigg)^{r} \right).
\]
The smallest singularity of $\gfo[CD][x,1]$ is located at $\rho_0 \simeq 0.09623$.
For $r\geq 0$, write
\[
\gf[CD][^r_0][x]=\frac{\gfo[CD][x,1]^2}{x} \left( \frac{\gfo[CD][x,1]-x+\gfo[CD][x,1]^2}{x} \right)^{r}
\]
and fix $\theta \in (0,\rho_0)$.
Observe that the combinatorial class associated to $\gf[CD][^r_0][x]$ can be defined by means of cartesian products and unions of connected crossing-free chord diagrams, hence the Boltzmann sampler $\Gamma\gf[CD][^r_0][\theta]$ is immediately defined from $\Gamma\gfo[CD][\theta]$.
Let
\[
p_r(\theta) \eqdef \frac{\gf[CD][^r_0][\theta]}{\gfo[CD][\theta,1]}
\qquad\text{and}\qquad
p_{-1}(\theta) \eqdef \frac{\theta}{\gfo[CD][\theta,1]}.
\]
Then $P(\theta) \eqdef \{p_{r}(\theta)\}_{r\geq -1}$ defines a discrete probability distribution.
Now we can define the Boltzmann sampler $\Gamma\gfo[CD][\theta]$ by
\[
\Gamma\gfo[CD][\theta] \eqdef P(\theta) \longrightarrow \Gamma\gf[CD][^r_0][\theta].
\]
As it happened in the perfect matching situation, the branching process defined with this Boltzmann sampler is subcritical, hence the algorithm finishes in expected finite time.

We now describe a random sampler for general crossing-free chord diagrams.
For this, we analyze Equation~\eqref{eq:general-connectedChordDiagrams}, which describes the counting formula for general chord diagrams by means of a composition scheme with the generating function associated to connected chord diagrams.
The Boltzmann sampler in this situation is reminiscent to the $L$-substitution that appears in~\cite{Fusy}.
Fix $\theta' \in (0, \frac{3}{2}-\sqrt{2})$ (recall that the smallest singularity of $\gfo[D][\theta]$ is located at $\rho=\frac{3}{2}-\sqrt{2}$), and define
\[
q_s(\theta') \eqdef \theta'^s\gfo[D][\theta']^{s-1} \coeff{x^s}{\gf[CD][^s_0][x]}
\qquad\text{and}\qquad
p_{-1}(\theta) \eqdef \frac{1}{\gfo[D][\theta']}.
\]
Then $Q(\theta') \eqdef \{q_{s}(\theta')\}_{s\geq -1}$ defines a discrete probability distribution and we can apply the same argument as in the case of connected objects.
Once more, the choice of a parameter smaller than the smallest singularity ensures that the algorithm finishes with an expected finite time.
Observe that in the second Boltzmann sampler, a choice of a connected chord diagram is needed.
This is performed using a rejection process over the Boltzmann sampler for connected chord diagrams.

To conclude this section, \fref{fig:probabilitiesCoreDiagrams} shows the probability of appearance of each of the five $2$-core diagrams.
Although the picture only presents the probabilities of appearance of each core diagram for small values of~$n$ and~$m$, we hope that the reader can still observe that the maximal $2$-core (in blue) is the only one whose probability increases when the number of vertices or the number of chords increases.
As it happened in the perfect matching situation, the main contribution arises from these maximal configurations when the number of vertices is large enough.

\begin{figure}[h]
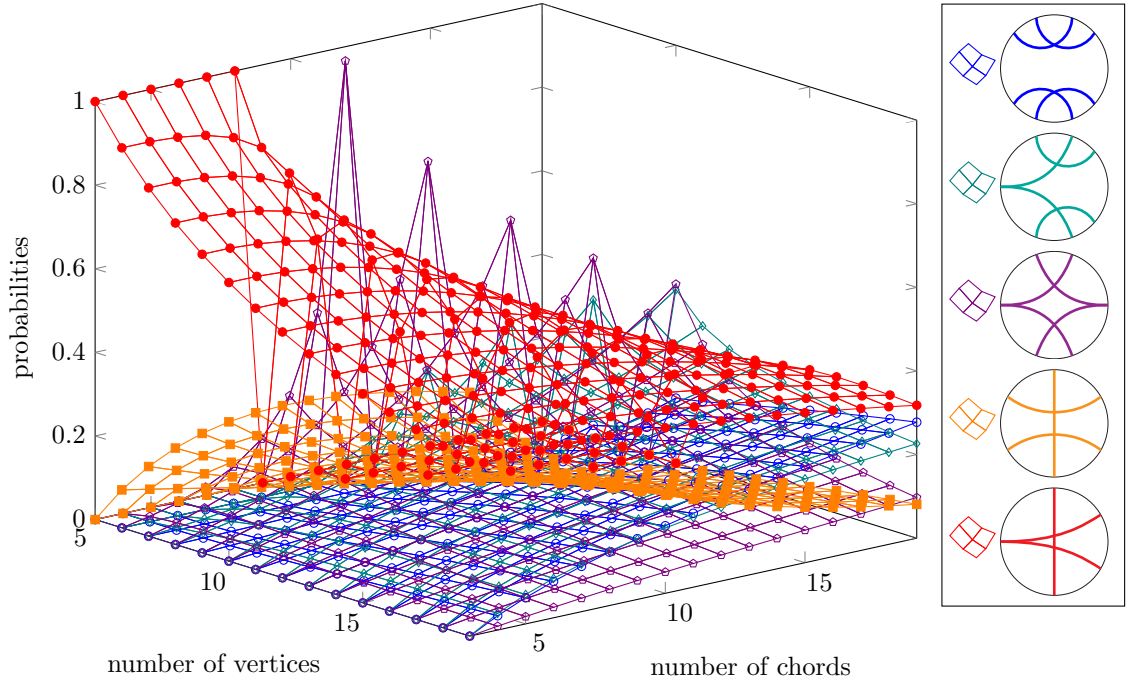

\begin{center}
\hspace*{-1cm}
\include{tableProbabilitiesCoreDiagrams}
\end{center}
\vspace{-.7cm}
  \caption{Probabilities of appearance of the different $2$-core diagrams.}
  \label{fig:probabilitiesCoreDiagrams}
\end{figure}


\subsection{Extension to hyperchord diagrams}
\label{subsec:hyperchordDiagrams}

As from matchings to partitions, we can extend the results of this section from chord diagrams to hyperchord diagrams.
A \defn{hyperchord} is the convex hull of finitely many points of the unit circle.
Given a point set~$V$ on the circle, a \defn{hyperchord diagram} on~$V$ is a set of hyperchords with vertices in~$V$.
Note that we allow isolated vertices in hyperchord diagrams.
As for partitions, a \defn{crossing} between two hyperchords~$U,V$ is a pair of crossing chords~$u_1u_2$ and~$v_1v_2$, with~$u_1,u_2 \in U$ and~$v_1,v_2 \in V$.
We consider the family~$\F[H]$ of hyperchord diagrams, and we let~$\Fcoeff[H][n,m,k]$ be the set of hyperchord diagrams with $n$ vertices, $m$~hyperchords, and $k$~crossings, counted with multiplicities.
We set
\[
\gf[H][][x,y,z] \eqdef \sum_{n,m,k} |\Fcoeff[H][n,m,k]| \, x^n y^m z^k \qquad\text{and}\qquad \gf[H][_k][x,y] \eqdef \coeff{z^k}{\gf[H][][x,y,z]}.
\]

As for chord diagrams, our first step is to study the generating function~$\gfo[H]$ of crossing-free hyperchord diagrams.
We extend here the analysis of P.~Flajolet and M.~Noy for chord diagrams~\cite{FlajoletNoy-nonCrossing} that we presented in Section~\ref{subsec:crossingFreeChordDiagrams}.
Note that two non-crossing hyperchords~$U,V$ can share at most two vertices.
Moreover, if~$U$ and~$V$ share two vertices, then they lie on opposite sides of the chord joining them, and we say that~$U,V$ are \defn{kissing} hyperchords.
An example of crossing-free hyperchord diagram is given in \fref{fig:decompositionCrossingFreeHyperhordDiagram}\,(left).

\begin{proposition}
\label{prop:crossingFreeHyperchordDiagrams}
The generating function~$\gfo[H]$ of crossing-free hyperchord diagrams satisfies the functional equation
\begin{equation}
\label{eq:expr-H0}
p_3(x,y) \, \gfo[H]^3 + p_2(x,y) \, \gfo[H]^2 + p_1(x,y) \, \gfo[H] + p_0(x,y) = 0,
\end{equation}
where
\begin{align*}
p_0(x,y) \eqdef & - 2 \, x^2 - x + 2 \, x \, y^3 + y^2 + x^2 \, y^4 - 7 \, x^2 \, y - 7 \, x^2 \, y^2 - x^2 \, y^3 - 3 \, x \, y, \\
p_1(x,y) \eqdef & - 2 \, x^3 - 2 \, x^3 \, y^4 - 8 \, x^3 \, y + 2 \, x - 3 \, y^2 - 12 \, x^3 \, y^2 - 8 \, x^3 \, y^3, \\
& + 6 \, x \, y - x^2 \, y^4 + x^2 + 4 \, x^2 \, y + 4 \, x^2 \, y^2 - 4 \, x \, y^3, \\
p_2(x,y) \eqdef & x^2 \, y^3 + x^2 + 3 \, x^2 \, y^2 - x - 3 \, x \, y + 2 \, x \, y^3 + 3 \, x^2 \, y + 3 \, y^2, \\
p_3(x,y) \eqdef & - y^2.
\end{align*}
\end{proposition}

\begin{proof}
The proof is similar to that of Proposition~\ref{prop:crossingFreeDiagrams}.
We first consider connected crossing-free hyperchord diagrams, with generating function~$\gfo[CH][x,y]$.
We decompose them according to their \defn{principal} hyperchords (those incident to the first vertex, and with at least two vertices).
If a connected hyperchord diagram is neither an isolated vertex, nor a single hyperchord with a single vertex, then it has at least one principal hyperchord.
Its principal hyperchords can then be grouped into clusters such that
\begin{enumerate}[(i)]
\item the principal hyperchords in a given cluster form a sequence of kissing hyperchords, and
\item principal hyperchords of distinct clusters share only the first vertex of the diagram.
\end{enumerate}
Note that each cluster is either a single chord, or can be considered as a sequence of~$r$ kissing principal hyperchords of size at least~$3$, whose~$r+1$ principal boundary chords may or not be principal hyperchords of the diagram.
It remains to fill in the gaps left by the principal hyperchords in the hyperchord diagram:
\begin{enumerate}[(i)]
\item the first (before the first cluster), and the last (after the last cluster) gaps contain connected crossing-free hyperchord diagrams,
\item each gap between two consecutive clusters, as well as each gap between two consecutive vertices of a principal hyperchord, contains either a connected crossing-free hyperchord diagram with at least two vertices, or two disconnected crossing-free hyperchord diagrams.
\end{enumerate}
We refer to the schematic decomposition presented in \fref{fig:decompositionCrossingFreeHyperhordDiagram}.
\begin{figure}
  \centerline{\includegraphics[width=.8\textwidth]{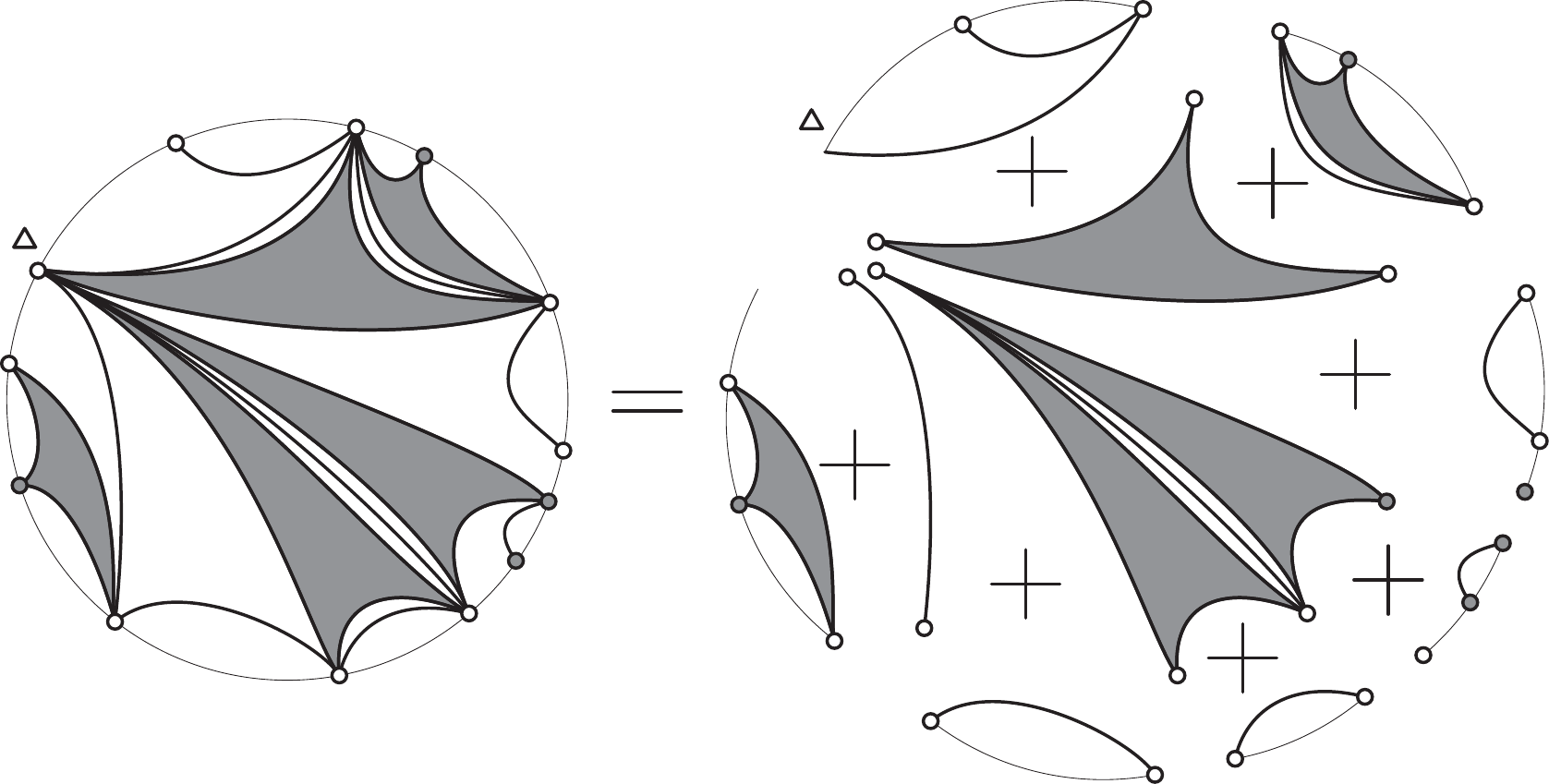}}
  \caption{Decomposition scheme of connected crossing-free hyperchord diagrams. Hyperchords are represented by shaded regions, and hyperchords of size~$1$ are represented by shaded vertices.}
  \label{fig:decompositionCrossingFreeHyperhordDiagram}
\end{figure}
This decomposition directly translates to the functional equation
\[
\gfo[CH][x,y] = x \, (1+y) + \frac{g \, \gfo[CH][x,y]^2}{x \, (1+y) \, (1-f \, g)},
\]
where
\[
f \eqdef \frac{\gfo[CH][x,y]^2 + \gfo[CH][x,y] - x \, (1+y)}{x^2 \, (1+y)^2}
\quad \text{and} \quad
g \eqdef x \, y \, (1+y) + \frac{x^2 \, y \, (1+y)^4 \, f}{1 - x \, (1+y) \, f - x \, y \, (1+y)^2 \, f}.
\]
Finally, since a crossing-free hyperchord diagram can be decomposed into connected crossing-free hyperchord diagrams, we have
\begin{equation}
\label{eq:composition}
\gfo[H][x,y] = 1 + \gfo[CH][{x\gfo[H][x,y]}, y].
\end{equation}
Eliminating~$\gfo[CH][x,y]$ from these equations leads to the desired formula after simplifications.
\end{proof}

\begin{remark}
If we forget the variable~$y$ which encodes the number of hyperchords, and if we forbid isolated vertices in hyperchord diagrams, the resulting generating function
\[
\gfo[\widetilde{H}][x] \eqdef \frac{1}{1+x} \, \gfo[H][\frac{x}{1+x},1]
\]
satisfies the functional equation
\[
(1 + x)^5 \, \gfo[\widetilde{H}][x]^3 - (1 + x)^2 (9 x^2 + 4 x + 3) \, \gfo[\widetilde{H}][x]^2 + (23 x^3 - 7 x^2 + 5 x + 3) \, \gfo[\widetilde{H}][x] + (17 x^2 - 1) = 0.
\]
It was already obtained by M.~Klazar in~\cite{Klazar-nonCrossingHypergraphs} with a slightly different decomposition scheme.
\end{remark}

The next proposition studies the asymptotic behaviour of $\gfo[H][x,y]$ around $y = 1$.
Observe that we cannot apply Theorem~\ref{theo:SmoothImplicitFunctionTheorem} since the coefficients in Equation~\eqref{eq:expr-H0} are not all positive.
We can use instead the fact that $\gfo[H][x,y]$ is an algebraic function.
Another technique, by means of more elaborated arguments, will be presented in the context of Subsection~\ref{subsec:hyperchordDiagramsFixedSizes}, which covers this proposition.

\begin{proposition}
\label{prop:asymptCrossingFreeHyperchordDiagrams}
The smallest singularity of the generating function $\gfo[H][x,1]$ of crossing-free hyperchord diagrams is located at the smallest real root~$\rho \simeq 0.015391$ of the polynomial
\[
R(x) \eqdef  256 \, x^4 - 768 \, x^3 + 736 \, x^2 - 336 \, x + 5,
\]
Moreover, when $y$ varies uniformly in a small neighborhood of $1$, the singular expansion of~$\gfo[H][x,y]$ is
\[
\gfo[H][x,y]\stackbin[y \sim 1]{}{=}h_0(y)-h_1(y)\sqrt{1-\frac{x}{\rho(y)}}+O\left(1+\frac{x}{\rho(y)}\right),
\]
valid in a domain dented at $\rho(y)$ (for each choice of $y$), where $h_0(y)$, $h_1(y)$ and $\rho(y)$ are analytic functions around $y=1$, with
\[
\rho(1)=\rho \simeq 0.015391, \qquad h_0(1) \simeq 1.034518 \qquad\text{and}\qquad h_1(1)\simeq 0.00365515.
\]
\end{proposition}

\begin{proof}
We use the methodology of~\cite[Section VII.7]{FlajoletSedgewick}.
We write Equation~\eqref{eq:expr-H0} in the form $P(x,y, \gfo[H][x,y] )=0$, where $P(x,y,z)$ is a polynomial with integer coefficients.

We first study the problem when $y = 1$.
It is clear that $\gfo[H][x,1]$ is analytic at $x = 0$.
Consequently, according to~\cite[Lemma VII.4]{FlajoletSedgewick}, $\gfo[H][x,1]$ can be analytically continued along any simple path emanating from the origin that does not cross any point of the set in which both $P(x,1,z)$ and $\diff[][z] P(x,1,z)$ vanish.
This set is discrete and, by means of Elimination Theory for algebraic functions (in this case, eliminating variable $z$), can be written as the set of roots of
\[
R(x) \eqdef 256 \, x^4 - 768 \, x^3 + 736 \, x^2 - 336 \, x + 5.
\]
Hence, the smallest singularity of $\gfo[H][x,1]$ is the smallest root~$\rho \simeq 0.015391$ of the polynomial~$R(x)$.
In particular $\gfo[H][\rho,1]$ satisfies the equation $P(\rho, 1, \gfo[H][\rho,1])=0$, and is approximately equal to ${\gfo[H][\rho,1] = h_0(1) \simeq 1.034518}$.

We now proceed to study the nature of $\gfo[H][x,1]$ around $x = \rho$.
As $\gfo[H][x,1]$ is algebraic, we can develop it around its smallest singularity using its \defn{Puiseux expansion}, and exploiting the so-called \defn{Newton Polygon Method}.
See ~\cite[Page 498]{FlajoletSedgewick}.
With this purpose, write $U \eqdef 1 - \frac{x}{\rho}$, and $\gfo[H][x,1] = \gfo[H][\rho,1] + c \, U^{\alpha} \, (1+o(1))$.
By means of indeterminate coefficients we find the correct value of $\alpha$: developing the relation $P(\rho \, (1-U^{\frac{1}{\alpha}}),1,\gfo[H][\rho,1] + c \, U^{\alpha}) = 0$ we obtain that $\alpha = \frac{1}{2}$.
Once we know this, by indeterminate coefficients on the expression $P(\rho, 1, \gfo[H][\rho,1])=0$ we obtain that $h_1(1)\simeq 0.00365515$.

Finally, we continue analyzing $\gfo[H][x,y]$ when~$y$ moves in a small neighbourhood around~$y = 1$.
Using the same arguments, for a fixed value of~$y$ close to~$1$, the smallest singularity $\rho(y)$ of $\gfo[H][x,y]$ satisfies that $R(\rho(y),y)=0$, with
\begin{align*}
R(x,y) = & \quad\; (1 + 4 \, y + 10 \, y^2 + 4 \, y^3 + y^4)\\
&  +  (-6-40 \, y-142 \, y^2-304 \, y^3-390 \, y^4-296 \, y^5-130 \, y^6-32 \, y^7-4 \, y^8) \, x\\
& + (13 + 100 \, y + 360 \, y^2 + 748 \, y^3 + 922 \, y^4 + 636 \, y^5 + 192 \, y^6-12 \, y^7-15 \, y^8) \, x^2\\
& + (-12-96 \, y-336 \, y^2-672 \, y^3-840 \, y^4-672 \, y^5-336 \, y^6-96 \, y^7-12 \, y^8) \, x^3\\
& + (4 + 32 \, y + 112 \, y^2 + 224 \, y^3 + 280 \, y^4 + 224 \, y^5 + 112 \, y^6 + 32 \, y^7 + 4 \, y^8) \, x^4.
\end{align*}
As the coefficients (which depend on $y$) of $R(x,y)$ do not vanish at $y=1$, we conclude that $\rho(y)$ is an analytic function in a neighbourhood of $y=1$, and that the singularity type is invariably of square root type.
\end{proof}

The analysis carried out in Proposition~\ref{prop:asymptCrossingFreeHyperchordDiagrams} can be exploited in order to obtain the limit distribution for the number of hyperchords in a crossing-free hyperchord diagram of prescribed size uniformly choosen at random.
For each $y$ in a neighbourhood of $1$ the singular expansion of $\gfo[H][x,y]$ is of square-root type, hence by the Quasi-Powers Theorem (see Theorem~\ref{theo:quasi-powers} in Appendix~\ref{app:methodology}), the limiting law is normally distributed.
We can compute the expectation and the variance from polynomial $R(x,y)$ in Proposition~\ref{prop:asymptCrossingFreeHyperchordDiagrams}: as $R(\rho(y),y)=0$ by iterated derivations with respect to $y$ we obtain closed formulas of both $\rho'(y)$ and $\rho''(y)$ in terms of $y$, $\rho(y)$ (and $\rho'(y)$ in the case of $\rho''(y)$).
Writing $y=1$ we get approximate values: these computations give $\rho' (1) \simeq -0.031243$ and  $\rho''(1) \simeq 0.080456$, hence the expectation and the variance for this limiting distribution are $\mu \, n \, (1+o(1))$ and $\sigma^2 \, n \, (1+o(1))$, where
\[
\mu= -\frac{\rho'(1)}{\rho(1)}\simeq 2.029890
\qquad\text{and}\qquad \sigma^2 = -\frac{\rho''(1)}{\rho(1)}-\frac{\rho'(1)}{\rho(1)}+\left(\frac{\rho'(1)}{\rho(1)}\right)^2\simeq 0.923054.
\]

\medskip
Now that we have obtained the asymptotic behavior of crossing-free hyperchord diagrams, we can proceed to the study of hyperchord diagrams with~$k$ crossings.
As usual now, we focus on their cores.

\begin{definition}
A \defn{core hyperchord diagram} is a hyperchord diagram where each chord is involved in a crossing.
It is a \defn{$k$-core hyperchord diagram} if it has exactly $k$ crossings.
The \defn{core}~$\core{H}$ of a hyperchord diagram~$H$ is the subdiagram of~$H$ formed by all its hyperchords involved in at least one crossing.
See \fref{fig:hyperchordDiagram}.
\end{definition}

\begin{figure}[h]
  \centerline{\includegraphics[scale=.75]{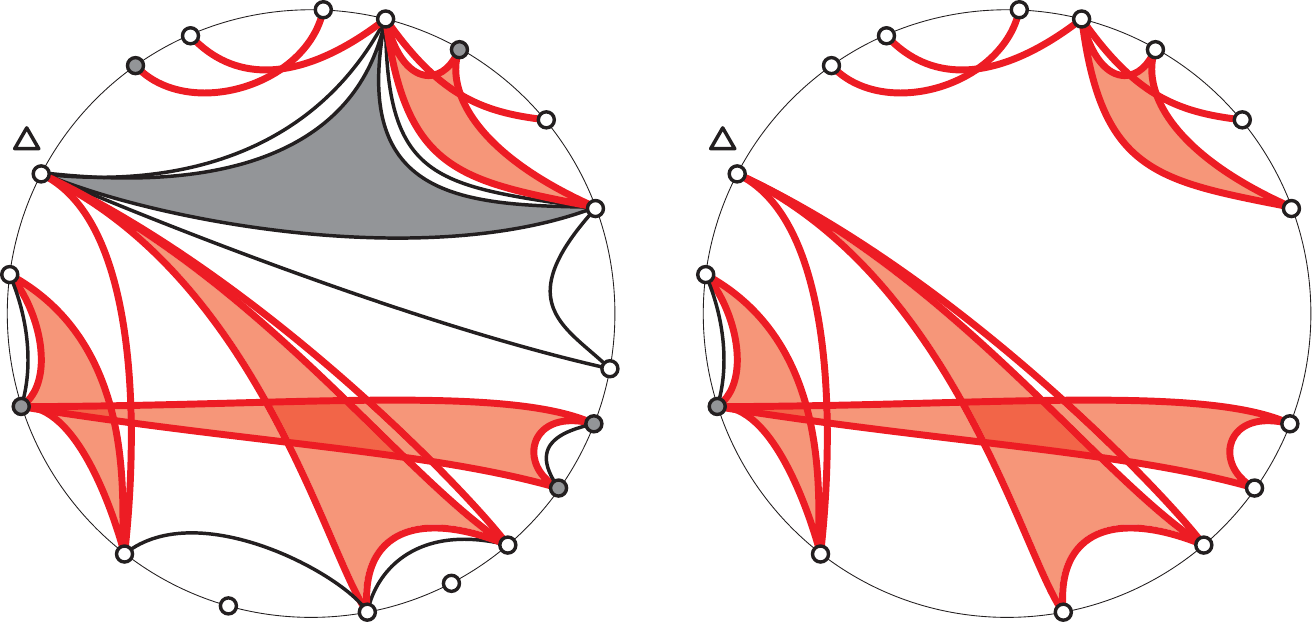}}
  \caption{A hyperchord diagram~$H$ with $10$ crossings (left) and its $10$-core~$\core{D}$ (right). The core hyperchord diagram~$\core{H}$ has $n(\core{D}) = 21$ vertices, $m(\core{D}) = 14$ chords, and $k(\core{D}) = 10$ crossings. The shaded region of~$\core{D}$ had $2$ boundary arcs and $1$ peak.}
  \label{fig:hyperchordDiagram}
\end{figure}

Let~$K$ be a core hyperchord diagram.
We let~$n(K)$ denote its number of vertices, $m(K)$ denote its number of hyperchords, and $k(K)$ denote its number of crossings.
We call \defn{regions} of~$K$ the connected components of the complement of~$K$ in the unit disk.
A region has~$i$ \defn{boundary arcs} and $j$ \defn{peaks} if its intersection with the unit circle has~$i$ connected arcs and $j$ isolated points.
We let~$n_{i,j}(K)$ denote the number of regions of~$K$ with $i$ boundary arcs and $j$ peaks, and we set~$\b{n}(K) \eqdef (n_{i,j}(K))_{i,j \in [k]}$.
Note that~$n(K) = \sum_{i,j} i n_{i,j}(K)$.
See again \fref{fig:hyperchordDiagram} for an illustration.

Since a crossing only involves two hyperchords, a $k$-core hyperchord diagram can have at most $2k$ hyperchords.
Moreover, since we count crossings with multiplicities, each hyperchord in a $k$-core hyperchord diagram has at most~$k+3$ vertices.
This immediately implies the following crucial lemma.

\begin{lemma}
There are only finitely many $k$-core hyperchord diagrams.
\end{lemma}

\begin{definition}
We encode the finite list of all possible $k$-core hyperchord diagrams~$K$ and their parameters~$n(K)$, $m(K)$, and~$\b{n}(K) \eqdef (n_{i,j}(K))_{i,j \in [k]}$ in the \defn{$k$-core hyperchord diagram polynomial}
\[
\gf[KH][_k][\b{x},y] \eqdef \gf[KH][_k][x_{i,j}, y] \eqdef \sum_{\substack{K\;k\text{-core} \\ \text{hyperchord} \\ \text{diagram}}} \frac{\b{x}^{\b{n}(K)} y^{m(K)}}{n(K)} \eqdef \sum_{\substack{K\;k\text{-core} \\ \text{hyperchord} \\ \text{diagram}}} \frac{1}{n(K)}\prod_{i,j \ge 0} {x_{i,j}}^{n_{i,j}(K)} \, y^{m(K)}.
\]
\end{definition}

\begin{remark}
The algorithm described in Section~\ref{subsec:computingCoreDiagramPolynomials} to generate the $k$-core diagrams for small values of~$k$ can again be adapted to enumerate $k$-core hyperchord diagrams.
Details are left to the reader.
\end{remark}

Using a similar method as in Section~\ref{subsec:generatingFunctionMatchings}, we obtain the following expression of the generating function~$\gf[H][_k][x,y]$ of hyperchord diagrams with $k$ crossings, in terms of the $k$-core hyperchord diagram polynomial~$\gf[KH][_k][\b{x},y]$, and of the polynomials
\[
\gf[H][_0^n][y] \eqdef \coeff{x^n}{\gfo[H]} \qquad \text{and} \qquad \gf[H][_0^{\le p}][x,y] \eqdef \sum_{n \le p} \gf[H][_0^n][y] x^n = \sum_{\substack{n \le p \\ m \ge 0}} |\Fcoeff[H][n,m,0]| \, x^n y^m.
\]

\begin{proposition}
\label{prop:generatingFunctionHyperchordDiagrams}
For any~$k \ge 1$, the generating function~$\gf[H][_k][x,y]$ of the hyperchord diagrams with $k$ crossings is given by
\[
\gf[H][_k][x,y] = x\diff\gf[KH][^S_k][x_{0,j} \leftarrow \frac{\gf[H][_0^{j}][y]}{x^j}, \, x_{i,j} \leftarrow {\frac{x^i}{(i-1)!} \diff[i-1] \frac{\gf[H][_0][x,y] - \gf[H][_0^{\le i+j}][x,y]}{x^{i+j+1}}}, \, y].
\]
In particular, $\gf[H][_k][x,y]$ is a rational function of the generating function~$\gf[H][_0][x,y]$ and of the variables~$x$ and~$y$.
\end{proposition}

Note that, contrarily to the cases of matchings, partitions and diagrams, we cannot anymore eliminate~$y$ in the expression of Proposition~\ref{prop:crossingFreeHyperchordDiagrams}.

Finally, using the expression of the generating function~$\gf[H][_k][x,y]$ given by Proposition~\ref{prop:generatingFunctionHyperchordDiagrams}, we can derive the asymptotic behavior of the number of hyperchord diagrams with $k$ crossings. The analysis is identical to that of the proof of Proposition~\ref{prop:asymptDiagrams}.

\begin{proposition}
\label{prop:asymptHyperchordDiagrams}
For any $k \ge 1$, the number of hyperchord diagrams with $k$ crossings and $n$ vertices~is
\[
\coeff{x^n}{\gf[D][_k][x,1]} \stackbin[n \to \infty]{}{=} \frac{h_0(1)^{3k} \, h_1(1) \, (2k-3)!!}{(2\rho)^{k-1} \, k! \; \Gamma(k-\frac{1}{2})} \, n^{k-\frac{3}{2}} \, \rho^{-n} \, (1+o(1)),
\]
where $\rho \simeq 0.015391$ is the smallest real root of~${R(x) \eqdef  256 \, x^4 - 768 \, x^3 + 736 \, x^2 - 336 \, x + 5}$, and where~$h_0(1) \simeq 1.034518$ and~$h_1(1)\simeq 0.00365515$ (see also Proposition~\ref{prop:asymptCrossingFreeHyperchordDiagrams}).
\end{proposition}


\subsection{Extension to hyperchord diagrams with restricted block sizes}
\label{subsec:hyperchordDiagramsFixedSizes}

As for partition, we conclude this section with hyperchord diagrams where we restrict the sizes of the hyperchords.
For a non-empty subset~$S$ of~$\N^* \eqdef \N \ssm \{0\}$, we denote by~$\F[H]^S$ the family of hyperchord diagrams, where the cardinality of each hyperchord belongs to~$S$.
For example, chord diagrams are hyperchord diagrams where each hyperchord has size~$2$, \ie~$\F[D] = \F[H]^{\{2\}}$.
We consider here the generating function~$\gf[H][^S_k][x,y]$ of hyperchord diagrams of~$\F[H]^S$ with $k$ crossings.

Once again, our first step is to compute the generating function~$\gf[H][^S_0][x,y]$ of crossing-free hyperchord diagrams of~$\F[H]^S$.
Adapting the decomposition scheme described in the proof of Proposition~\ref{prop:crossingFreeHyperchordDiagrams}, we obtain the following statement.

\begin{proposition}
\label{prop:crossingFreeHyperchordDiagramsFixedSizes}
The generating function~$\gf[CH][^S_0][x,y]$ of connected crossing-free hyperchord diagrams of~$\F[H]^S$ satisfies the functional equation:
\[
\gf[CH][^S_0][x,y] = x \, (1+\delta_1 \, y) + \frac{g \, \gf[CH][^S_0][x,y]^2}{x \, (1+\delta_1 \, y) \, (1-f \, g)},
\]
where
\begin{align*}
f & \eqdef \frac{\gf[CH][^S_0][x,y]^2 + \gf[CH][^S_0][x,y] - x \, (1+\delta_1 \, y)}{x^2 \, (1+\delta_1 \, y)^2} \\
g & \eqdef \delta_2 \, x \, y \, (1+\delta_1 \, y) + \frac{x \, (1+\delta_1 \, y) \, (1+\delta_2 \, y) \, h}{1 - h}, \\
\text{and} \qquad
h & \eqdef y \, (1+\delta_2 \, y) \sum_{s \in S \ssm \{1,2\}} \big( x \, (1+\delta_1 \, y) \, f \big)^{s-2},
\end{align*}
and where~$\delta_1 = 0$ if~$1 \notin S$ and $\delta_1 = 1$ otherwise, and similarly~$\delta_2 = 0$ if~$2 \notin S$ and $\delta_2 = 1$ otherwise
In turn, the generating function~$\gf[H][^S_0][x,y]$ of crossing-free hyperchord diagrams of~$\F[H]^S$ can be expressed from~$\gf[CH][^S_0][x,y]$ as
\begin{equation}
\label{eq:compositionFixedSizes}
\gf[H][^S_0][x,y] = 1 + \gf[CH][^S_0][{x\gf[H][^S_0][x,y]}, y].
\end{equation}
\end{proposition}

Although we cannot find a nice closed formula for~$\gf[H][^S_0][x,y]$ in general, the situation is simpler for the following examples.

\begin{example}
\label{exm:qUniformHyperchordDiagrams}
Let~$q \ge 3$. Consider \defn{$q$-uniform hyperchord diagrams}, for which~$S = \{q\}$. The generating function of connected crossing-free $q$-uniform hyperchord diagrams satisfies the functional equation
\[
(\gf[CH][^{\{q\}\!}_0][x,y] - x) \, x^{q-1} - y \, \gf[CH][^{\{q\}\!}_0][x,y] \, \big(\gf[CH][^{\{q\}\!}_0][x,y]^2 + \gf[CH][^{\{q\}\!}_0][x,y] - x \big)^{q-1} = 0
\]
and therefore we get
\[
\big( \gf[H][^{\{q\}\!}_0][x,y] - 1 - x \, \gf[H][^{\{q\}\!}_0][x,y] \big) \, x^{q-1} - y \, \big( \gf[H][^{\{q\}\!}_0][x,y] - 1 \big) \, \big( \gf[H][^{\{q\}\!}_0][x,y] - 1 - x \big)^{q-1} = 0
\]
Rephrasing the arguments of Proposition~\ref{prop:asymptCrossingFreeHyperchordDiagrams} we can conclude that $\gf[CH][^{\{q\}\!}_0][x,1]$ has a unique smallest real singularity and its singular behaviour is of square-root type in a domain dented at its singular point.

Finally, we observe that the situation is even simpler when~$q \in \{1,2\}$.
Indeed, we clearly have~$\gf[H][^{\{1\}}_0][x,y] = \frac{1}{1-x(1+y)}$ when~$S=\{1\}$, and we obtain Equation~\eqref{eq:crossingFreeDiagrams} when applying Proposition~\ref{prop:crossingFreeHyperchordDiagramsFixedSizes} for~$S=\{2\}$.
\end{example}

\begin{example}
\label{exm:qMultipleHyperchordDiagrams}
Let~$q \ge 3$. Consider \defn{$q$-multiple hyperchord diagrams}, for which~$S = q\N^*$. The generating function of connected crossing-free $q$-multiple hyperchord diagrams satisfies the functional equation
\[
(\gf[CH][^{q\N^*\!\!}_0][x,y] - x) \, x^q - \big(\gf[CH][^{q\N^*\!\!}_0][x,y]^2 + \gf[CH][^{q\N^*\!\!}_0][x,y] - x \big)^{q-2} P_q(\gf[CH][^{q\N^*\!\!}_0][x,y], x, y) = 0,
\]
where $P_q(C, x, y)$ is the polynomial of degree~$5$ given by
\[
P_q(C, x, y) = C^5 + (-x+2) \, C^4 + (xy-4x+1) \, C^3 + x(2x+y-3) \, C^2 + 2x^2(-y+2) \, C + x^3(-2+y).
\]
The cases $q=1$ and $q=2$ are similar and left to the reader.
\end{example}

We now analyze the singular behavior of crossing-free hyperchord diagrams in~$\F[H]^S$. In this case, the argument is somehow indirect.

\begin{proposition}
\label{prop:asymptCrossingFreeHyperchordDiagramsFixedSizes}
Let~$S$ be a non-empty subset of~$\N^*$ different from the singleton~$\{1\}$. The univariate generating function~$\gf[H][^S_0][x,1]$ of crossing-free hyperchord diagrams of~$\F[H]^S$ has a unique smallest singularity $\rho_S$, and a square-root type singular expansion
\[
\gf[H][^S_0][x,1] \stackbin[x \sim \rho_S]{}{=} \alpha_S - \beta_S \sqrt{1-\frac{x}{\rho_S}} + O\bigg( 1-\frac{x}{\rho_S} \bigg)
\]
in a domain dented at $x = \rho_S$.
\end{proposition}

\begin{proof}
We start proving that $\gf[CH][^S_0][x,1]$ diverges at a finite value of~$x = \varrho_S$.
Consider an element~$q\ge 2$ of~$S$.
Let $\varrho_{\{q\}}$ be the (unique) smallest singularity of~$\gf[CH][^{\{q\}\!}_0][x,1]$.
As discussed in Example \ref{exm:qUniformHyperchordDiagrams}, the generating function $\gf[CH][^{\{q\}\!}_0][x,1]$ has a square-root type singularity at $x=\varrho_{\{q\}}$, hence $\diff\gf[CH][^{\{q\}\!}_0][\varrho_{\{q\}},1]$ diverges.
Next, observe that
\[
\coeff{x^n}{\gf[CH][^S_0][x,1]} \geq \coeff{x^n}{\gf[CH][^{\{q\}\!}_0][x,1]}
\]
for all values of $n$, because all $q$-uniform hyperchord diagram are hyperchord diagrams of~$\F[H]^S$.
Consequently,
\[
\coeff{x^n}{\diff \gf[CH][^S_0][x,1]} \geq \coeff{x^n}{\diff \gf[CH][^{\{q\}\!}_0][x,1]}
\]
for all values of~$n$.
Finally, as both functions are analytic at~$x = 0$, there exists a real value~$\varrho_S \leq \varrho_{\{q\}}$ such that~$\diff \gf[CH][^S_0][\varrho_S,1]$ diverges when~$x$ tends to~$\varrho_S$ as~$x$ grows from~$0$.
In particular,~$\varrho_S$ must be the dominant singularity of~$\gf[CH][^S_0][x,1]$ (observe that we do not assure that the singularity is of square root-type, only that the derivative diverges at that point).

We continue analyzing the singular development of~$\gf[H][^S_0][x,1]$ around its smallest singularity.
Observe that Equation~\eqref{eq:compositionFixedSizes} can be written in the form $\chi(x\gf[H][^S_0][x,1])=x$, where
\[
\chi(u)=\frac{u}{1 + \gf[CH][^S_0][u,1]}
\]
is the functional inverse of $x\gf[H][^S_0][x,1]$.
Let $\rho_S$ be the dominant singularity of~$x\gf[H][^S_0][x,1]$ (and hence the dominant singularity of $\gf[H][^S_0][x,1]$).
Write $\tau_S=\rho_S\gf[H][^S_0][\rho_S,1]$.
In particular, $\rho_S$ satisfies that $\rho_S=\chi(\tau_S)$.
Developing the relation $\chi'(\tau_S)=0$, we obtain that
\begin{equation}
\label{eq:inverse}
1+\gf[CH][^S_0][\tau_S,1]=\tau_S \diff\gf[CH][^S_0][\tau_S,1].
\end{equation}
As $\gf[CH][^S_0][x,1]$ is analytic at $x=0$ and diverges at $x=\varrho_S$, Equation~\eqref{eq:inverse} has a solution $\tau_S < \varrho_S$.
We have then that $\chi$ has a branch point at $u=\tau_S$, and by the Inverse Function Theorem $\gf[H][^S_0][x,1]$ ceases to be analytic at $x=\rho_S$.
Finally, we conclude that $x\gfo[H][x,1]$ has a square root type singular behaviour around $x=\rho_S$.
\end{proof}

In this particular setting, $\rho_S$ is a computable constant that can be calculated (with a desired precision) in the following way.
Observe that both $\gf[CH][^S_0][x,1]$ and $\diff \gf[CH][^S_0][x,1]$ are analytic functions at~$x = \tau_S$.
Hence we can obtain approximations for $\tau_S$ by truncating conveniently the Taylor expansions of $\gf[CH][^S_0][x,1]$ and $\diff \gf[CH][^S_0][x,1]$ in Equation~\eqref{eq:inverse}.
In fact, one needs to consider a lot of Taylor coefficients in order to obtain a good estimate of~$\tau_S$, because experimentally the position of the solution of Equation~\eqref{eq:inverse} is very close to the the singularity of~$\gf[CH][^S_0][x,1]$.

Once an approximation of $\tau_S$ is computed, we obtain a good approximation of both $\rho_S$ and $\alpha_S$ using the relation $\rho_S=\chi(\tau_S)$ and $\tau_S = \rho_S \gfo[H][\rho_S, 1] = \rho_S \alpha_S$.
Finally, an approximation for $\beta_S$ can be obtained using indeterminate coefficients on the the relation~$\chi(x\gf[H][^S_0][x,1])=x$.

We have applied this method to approximate the constants~$\tau_S$, $\rho_S$, ${\rho_S}^{-1}$, $\alpha_S$, and $\beta_S$ for both $q$-uniform hyperchord diagrams (\ie $S = \{q\}$, see Example~\ref{exm:qUniformHyperchordDiagrams}) and $q$-multiple hyperchord diagrams (\ie $S = q\N^*$, see Example~\ref{exm:qMultipleHyperchordDiagrams}), for $3 \le q \le 7$. The results are shown in Table~\ref{table:constant-q-uniform}.
\begin{table}[h]
\begin{center}
\renewcommand{\arraystretch}{1.2}
\begin{tabular}{cc|ccc|cc}

& $S$     & $\tau_S$     & $\rho_S$     & ${\rho_S}^{-1}$ & $\alpha_S$  & $\beta_S$    \\

\hline

\multirow{5}{*}{\rotatebox{90}{\hspace*{-.1cm}$q$-uniform}} &
  $\{3\}$ & $0.16648974$ & $0.14078101$ & $7.10323062$ & $1.18261501$ & $0.04374341$ \\	
& $\{4\}$ & $0.29124158$ & $0.22185941$ & $4.50735894$ & $1.31273036$ & $0.08298341$ \\
& $\{5\}$ & $0.38048526$ & $0.27126972$ & $3.68636788$ & $1.40260866$ & $0.10797005$ \\
& $\{6\}$ & $0.44765569$ & $0.30473450$ & $3.28154504$ & $1.46900231$ & $0.12399216$ \\
& $\{7\}$ & $0.50026001$ & $0.32902575$ & $3.03927574$ & $1.52042812$ & $0.13445024$ \\

\hline

\multirow{5}{*}{\rotatebox{90}{\hspace*{-.1cm}$q$-multiple}} &
  $3\N^*$ & $0.16334708$ & $0.13864031$ & $7.21290960$ & $1.17820771$ & $0.03365135$ \\
& $4\N^*$ & $0.28781764$ & $0.22003286$ & $4.54477579$ & $1.30806666$ & $0.05948498$ \\
& $5\N^*$ & $0.37742727$ & $0.26987181$ & $3.70546302$ & $1.39854280$ & $0.07361482$ \\
& $6\N^*$ & $0.44503426$ & $0.30365836$ & $3.29317462$ & $1.46557553$ & $0.08138694$ \\
& $7\N^*$ & $0.49802658$ & $0.32817932$ & $3.04711459$ & $1.51754407$ & $0.08564296$ \\

\end{tabular}
\end{center}
\medskip
\caption{Approximate values of the constants~$\tau_S$, $\rho_S$, ${\rho_S}^{-1}$, $\alpha_S$, and $\beta_S$ for the families of $q$-uniform and $q$-multiple hyperchord diagrams, for~$3 \le q \le 7$.}
\label{table:constant-q-uniform}
\vspace{-.5cm}
\end{table}

Observe that the growth constant for $q$-uniform hyperchord diagrams is just slightly smaller than the growth constants of the corresponding $q$-multiple hyperchord diagrams.

\medskip
We have now obtained the complete asymptotic behavior of crossing-free hyperchord diagrams in~$\F[H]^S$ and we can therefore proceed to study hyperchord diagrams of~$\F[H]^S$ with precisely $k$ crossings.
Applying once more the same method as in Section~\ref{subsec:generatingFunctionMatchings}, we obtain an expression of the generating function~$\gf[H][^S_k][x,y]$ of hyperchord diagrams of~$\F[H]^S$ with $k$ crossings in terms of the corresponding $k$-core hyperchord diagram polynomial
\[
\gf[KH][^S_k][\b{x},y] \eqdef \sum_{\substack{K\;k\text{-core} \\ \text{hyperchord} \\ \text{diagram of } \F[H]^S}} \frac{\b{x}^{\b{n}(K)} \, y^{m(K)}}{n(K)},
\]
and of the polynomials
\[
\gf[H][_0^{S\,|\,n}][y] \eqdef \coeff{x^n}{\gf[H][^S_0]} \qquad \text{and} \qquad \gf[H][_0^{S\,|\,\le p}][x,y] \eqdef \sum_{n \le p} \gf[H][_0^{S\,|\,n}][y] x^n.
\]

\begin{proposition}
\label{prop:generatingFunctionHyperchordDiagramsFixedSizes}
For any~$k \ge 1$, the generating function~$\gf[H][^S_k][x,y]$ of the hyperchord diagrams with $k$ crossings and where the size of each hyperchord belongs to~$S$ is given by
\[
\gf[H][^S_k][x,y] = x\diff\gf[KH][^S_k][x_{0,j} \leftarrow \frac{\gf[H][_0^{S\,|\,j}][y]}{x^j}, \, x_{i,j} \leftarrow {\frac{x^i}{(i-1)!} \diff[i-1] \frac{\gf[H][^S_0][x,y] - \gf[H][_0^{S\,|\,\le i+j}][x,y]}{x^{i+j+1}}}, \, y].
\]
In particular, $\gf[H][^S_k][x,y]$ is a rational function of the generating function~$\gf[H][^S_0][x,y]$ and of the variables~$x$ and~$y$.
\end{proposition}

Using this composition scheme and the singular behavior of~$\gf[H][^S_0][x,1]$ described in Proposition~\ref{prop:asymptCrossingFreeHyperchordDiagramsFixedSizes}, computations similar to that of the proof of Proposition~\ref{prop:asymptDiagrams} lead to the following asymptotic result.

\begin{proposition}
\label{prop:asymptHyperchordDiagramsFixedSizes}
Let$k \ge 1$ and let~$S$ be a non-empty subset of~$\N^*$ different from the singleton~$\{1\}$.
Let~$\rho_S$ be the smallest singularity and $\alpha_S, \beta_S$ be the coefficients of the asymptotic expansion of the generating function~$\gf[H][^S_0][x,1]$ around~$\rho_S$, as defined in Proposition~\ref{prop:asymptCrossingFreeHyperchordDiagramsFixedSizes}.
Let~$\maxPotentialDiagrams{k,S}$ denote the maximum value of the potential function
\[
\potentialDiagrams(K) \eqdef \sum_{\substack{i > 1 \\ j \ge 0}} (2i-3) \, n_{i,j}(K)
\]
over all $k$-core hyperchord diagrams of~$\F[H]^S$.
There is a constant~$\Lambda_S$ such that the number of hyperchord diagrams with $k$ crossings, $n$ vertices, and where the size of each block belongs to~$S$~is
\[
\coeff{x^n}{\gf[H][^S_k][x,1]} \stackbin[\substack{n \to \infty \\ \gcd(S) | n}]{}{=} \Lambda_S \, n^{\frac{\maxPotentialDiagrams{k,S}}{2}} \, {\rho_S}^{-n} \, (1+o(1)),
\]
for $n$ multiple of~$\gcd(S)$, while $\coeff{x^n}{\gf[H][^S_k][x,1]} = 0$ if $n$ is not a multiple of~$\gcd(S)$.
More precisely, the constant~$\Lambda_S$ can be expressed as
\[
\Lambda_S \eqdef \frac{\gcd(S) \, \rho_S \, \maxPotentialDiagrams{k,S}}{\Gamma\big(\frac{\maxPotentialMatchings{k,S}}{2}+1\big)} \sum_K \frac{1}{n(K)} \prod_{j \ge 0} {\bigg(\frac{\gf[D][_0^{S\,|\,j}][1]}{{\rho_S}^j}\bigg)}^{n_{0,j}(K)} \alpha_S^{n_{1,j}(K)} \prod_{\substack{i > 1 \\ j \ge 0}} \bigg( \frac{\beta_S \, (2i-5)!!}{2^{i-1} \, {\rho_S}^{i+j} \, (i-1)!}\bigg)^{n_{i,j}(K)},
\]
where we sum over the $k$-core hyperchord diagrams~$K$ of~$\F[H]^S$ which maximize the potential function~$\potentialDiagrams(K)$.
\end{proposition}

\begin{remark}
As for the partitions~$\F[P]^S$, we observe that the exponent~$\frac{\maxPotentialDiagrams{k,S}}{2}$ in the polynomial growth of~$\coeff{x^n}{\gf[H][^S_k][x,1]}$ is not a constant of the class, but really depends on the value of both~$S$ and~$k$.
The reader is invited to work out examples of $q$-uniform and $q$-multiple hyperchord diagrams to get convinced that this exponent can have an unexpected behavior.
\end{remark}


\section{Limits of the method on further families of chord diagrams}
\label{sec:limitsMethod}

To conclude, we discuss further problems concerning chord configurations whose analysis is slightly different.
In Section~\ref{subsec:trees}, we consider chord configurations without cycles, and we observe that the corresponding $k$-core generating function is not a polynomial.
In Section~\ref{subsec:higerGenus}, we discuss the main difficulties arising when generalizing the problem to configurations on surfaces.


\subsection{Trees with crossings}
\label{subsec:trees}

Let $\F[T]$ be the family of chord diagrams on the unit circle, without cycles and such that the corresponding graph is connected.
We say that $T \in \F[T]$ is a tree with $k$ crossings if the corresponding chord diagram has $k$ crossings.

Crossing-free trees are easy to describe~\cite{FlajoletNoy-nonCrossing}.
Call \defn{butterfly} an ordered pair of trees with a common root (the pair of trees looks like the two wings of a butterfly).
A crossing-free tree $T$ whose first vertex has degree~$d$ can then be viewed as a sequence of $d$ butterflies attached to the first vertex.
This decomposition scheme, illustrated in~\fref{fig:decompositionCrossingFreeTree}, yields the following statement.
\begin{figure}
  \centerline{\includegraphics[width=.8\textwidth]{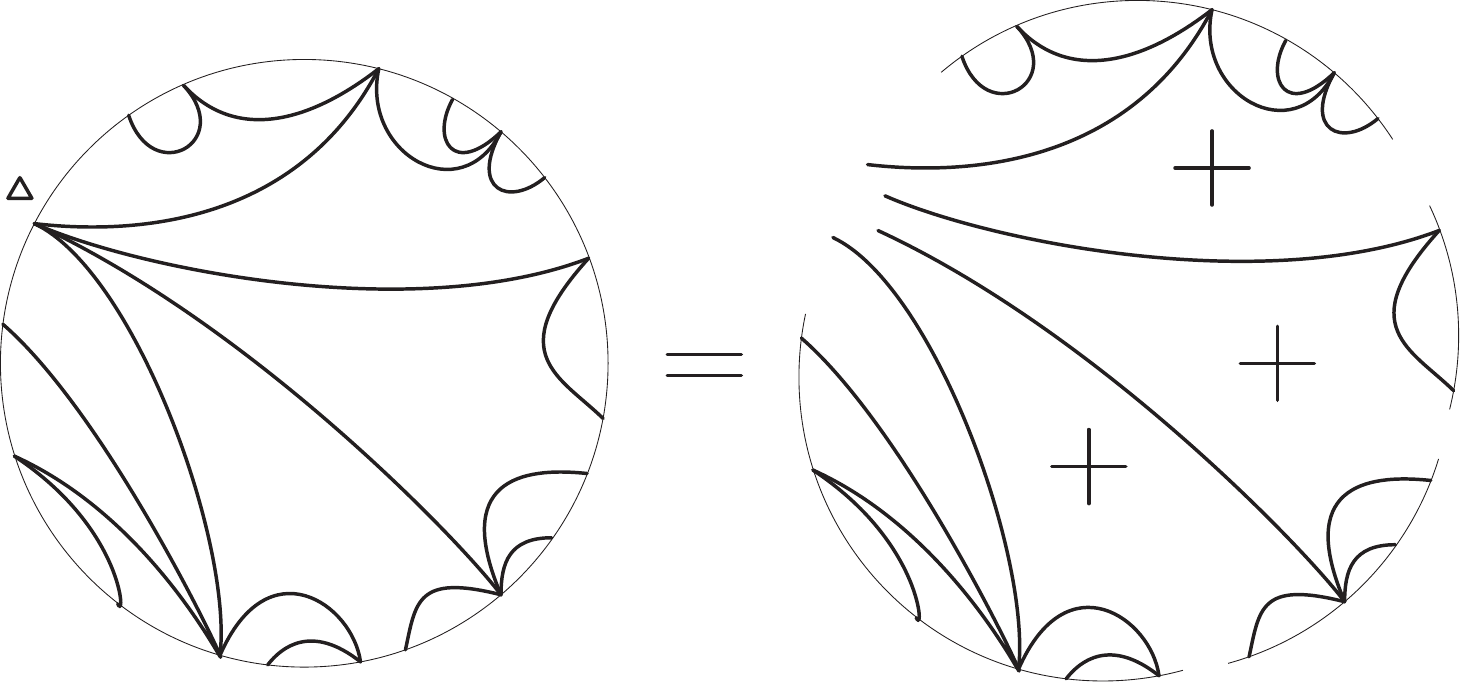}}
  \caption{Decomposition scheme of crossing-free trees.}
  \label{fig:decompositionCrossingFreeTree}
\end{figure}

\begin{proposition}[{\cite[Equation (3)]{FlajoletNoy-nonCrossing}}]
\label{prop:trees}
The generating functions $\gfo[T][x]$ and $\gf[B][][x]$ of crossing-free trees and butterflies are related by
\[
\gfo[T][x] = \frac{x}{1-\gf[B][][x]} \qquad\text{and}\qquad \gf[B][][x]=\frac{\gfo[T][x]^2}{x}.
\]
In particular, $\gfo[T][x]$ satisfies the functional equation
\[
\gfo[T][x]^3 - x \, \gfo[T][x] + x^2 = 0.
\]
\end{proposition}

By implicit differentiation in this functional equation for~$\gfo[T][x]$, we get that
\[
\diff \gfo[T][x]= \frac{\gfo[T][x]-2x}{3\gfo[T][x]^2-x}.
\]
By means of singularity analysis, one can prove that $\gfo[T][x]$ has a unique real positive smallest singularity at $\rho = \frac{4}{27}$, and it has a square-root type singularity around this point.
We write $X \eqdef \sqrt{1 - \frac{x}{\rho}}$.
In a dented domain around $\rho$, the functions~$\gfo[T][x]$ and~$\diff \gfo[T][x]$ have singular expansions
\[
\gfo[T][x] \stackbin[x \sim \rho]{}{=} \frac{2}{9} - \frac{2\sqrt{3}}{9} \, X \, + O\left(X^2\right)
\qquad \text{and} \qquad
\diff \gfo[T][x] \stackbin[x \sim \rho]{}{=} \frac{3\sqrt{3}}{4} \, X^{-1} \,  + O\left(1\right).
\]

Consider now the family of trees with $k \ge 1$ crossings.
The main issue for the analysis of these trees is to properly define the core of a tree.
Throughout this paper, cores codify the structure of the crossings inside the configurations, in such a way that the configurations are obtained by inserting crossing-free subconfigurations in the regions left by the cores.
For trees, it is therefore necessary to define the cores to be connected, in order to be sure that the final configuration obtained after the insertion of crossing-free trees is also connected.
This leads us to the following definition of core trees.
\begin{definition}
A \defn{core tree} is a tree where each chord incident to a leaf is involved in a crossing.
It is a \defn{$k$-core tree} if it has exactly $k$ crossings.
The \defn{core}~$\core{T}$ of a tree~$T$ is the subtree of~$T$ formed by all its chords involved in at least one crossing, together with the geodesic paths in~$T$ between these chords.
See \fref{fig:tree}.
\end{definition}

\begin{figure}[h]
  \centerline{\includegraphics[scale=.75]{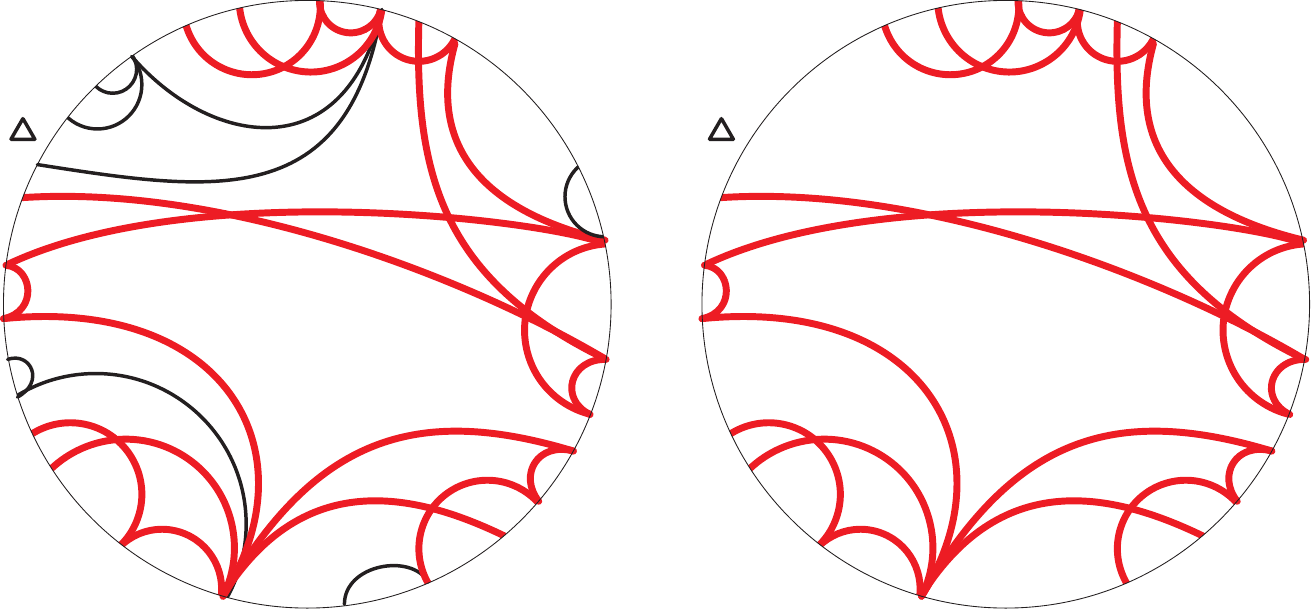}}
  \caption{A tree~$T$ with $8$ crossings (left) and its $8$-core~$\core{T}$ (right).}
  \label{fig:tree}
\end{figure}

The main problem in this situation is that, for a given number of crossings, the family of $k$-core trees is not finite (see \eg \fref{fig:1CoreTrees}).
In other words, the \defn{$k$-core tree generating function}~$\gf[KT][_k][x]$ is not a polynomial anymore ($x$ still encodes the number of vertices).
We denote by $\gf[T][_k][x,y]$ the generating function of trees with $k$ crossings, where $x$ encodes the number of vertices and $y$ encodes the size of the $k$-core (its number of vertices).
Since a tree with $k$ crossings is obtained from its $k$-core tree inserting in each corner a crossing-free tree, we obtain
\[
\gf[T][_k][x,y] = x \diff \gf[KT][_k][\frac{y \, \gf[T][_0][x]^2}{x}].
\]

\begin{figure}
  \centerline{\includegraphics[width=.9\textwidth]{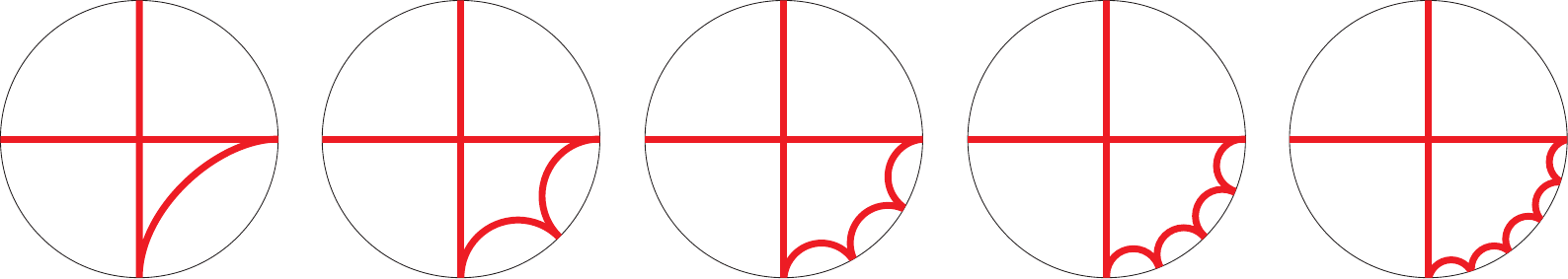}}
  \caption{The first elements of the family of $1$-core trees (unrooted).}
  \label{fig:1CoreTrees}
\end{figure}

We first analyze the motivating situation when~$k = 1$.
In this simple situation the family of $1$-core trees is suggested in \fref{fig:1CoreTrees}.
Therefore, the $1$-core tree generating function is
\[
\gf[KT][_1][x]=\sum_{n\ge 4} x^{n}=\frac{x^4}{1-x}.
\]
Therefore,
\[
\gf[T][_1][x,y] = x \diff \gf[KT][_1][\frac{y\gfo[T][x]^2}{x}] = x \diff \frac{\left(\frac{y\gfo[T][x]^2}{x}\right)^4}{1-\frac{y\gfo[T][x]^2}{x}}=x \diff \frac{y^4 \, \gfo[T][x]^8}{x^3 \, (x-y\gfo[T][x])^2}.
\]
The first terms of $\gf[T][_1][x,y]$ are
\begin{align*}
\gf[T][_1][x,y]=
& \quad\, 4 \, x^4 \, y^4 \\
& + x^5 \, (5 \, y^5 + 40 \, y^4) \\
& + x^6 \, (6 \, y^6 + 60 \, y^5 + 312 \, y^4) \\
& + x^7 \, (7 \, y^7 + 84 \, y^6 + 525 \, y^5 + 2240 \, y^4) \\
& + x^8 \, (8 \, y^8 + 112 \, y^7 + 816 \, y^6 + 4080 \, y^5 + 15504 \, y^4) \\
& + x^9 \, (9 \, y^9 + 144 \, y^8 + 1197 \, y^7 + 6840 \, y^6 + 29925 \, y^5 + 105336 \, y^4) \, \dots
\end{align*}

This expansion suggests that the distribution of the size of the core is constant when the number of vertices tends to infinity.
In order to make this observation precise, we study the limit distribution for the size of the core.
The strategy to get this result is to adapt~\cite[Proposition IX.1]{FlajoletSedgewick} where the composition scheme $\gf[A][][yf(x)]$ is studied under certain technical conditions.

\begin{proposition}
\begin{enumerate}[(i)]
\item The number of trees with $1$ crossing and $n$ vertices is asymptotically equal~to
\begin{equation*}
\label{eq:asympt-trees1}
\coeff{x^n}{\gf[T][_1][x,1]} \stackbin[n \to \infty]{}{=} \frac{2048\sqrt{3}}{385641 \; \Gamma\big(\frac{1}{2}\big)} \, n^{-\frac{1}{2}} \left(\frac{27}{4}\right)^n (1+o(1)).
\end{equation*}
\item The size of the core in a tree with $1$ crossing and $n$ vertices, chosen uniformly at random, follows a discrete limit law,  whose limiting probability generating function is
\[
p(y) = \frac{529 \, y^4 \, (9-y)}{8 \, (27-4y)^{2}},
\]
with expectation $\mu = \frac{777}{184} \simeq 4.222826$ and variance $\sigma = \frac{40857}{4232}\simeq 9.654300$.
\end{enumerate}
\end{proposition}

\begin{proof}
The generating function $\gf[T][_1][x,y]$ can be written in the form
\[
\gf[T][_1][x,y]=y \, A\left(y \, f(x)\right) g(x),
\]
where
\begin{gather*}
\gf[A][][z]= z^3\frac{4-3z}{(1-z)^2}, \qquad f(x)= \frac{\gfo[T][x]^2}{x}\\
\text{and} \qquad g(x)= \diff \left(\frac{\gfo[T][x]^2}{x}\right)= 2\gfo[T][x] \frac{\gfo[T][x]-2x}{3\gfo[T][x]^2-x} -\frac{1}{x}\gfo[T][x]^2.
\end{gather*}
First, we examine the univariate counting problem ($y=1$).
Observe that $f(\rho) = \frac{4}{27}$.
As $\gf[A][][z]$ is analytic at $z = \frac{4}{27}$, the function $\gf[T][_1][x,1]=A\left(f(x)\right) g(x)$ is singular at $x=\rho$, and analytic in a dented domain around this point.
Its singular expansion at this point is then obtained by combining the Taylor expansion of $\gf[A][][z]$ at $z = \frac{4}{27}$ with the singular expansions of both $f(x)$ and $g(x)$ at $x=\rho$.
These expansions are given by
\begin{align*}
\gf[A][][z] \stackbin[z \sim \frac{4}{27}]{}{=} & \frac{2048}{128547} + \frac{37952}{109503}\left(z-\frac{4}{27}\right)+O\left(\left(z-\frac{4}{27}\right)^2\right), \\
f(x) \stackbin[x \sim \rho]{}{=} & \frac{4}{27} - \frac{8\sqrt{3}}{27} \, X \, + O(X^2), \\
\text{and} \qquad g(x) \stackbin[x \sim \rho]{}{=} & \frac{\sqrt{3}}{3}X^{-1}-\frac{4}{3}+\frac{2\sqrt{3}}{3}X \,+O(X^2).
\end{align*}
Developing this composition we obtain that
\[
\gf[A][][{f(x)}] \, g(x) \stackbin[z \sim \rho]{}{=} \frac{2048 \, \sqrt{3}}{385641} \, X^{-1} + O(1),
\]
and consequently, the first part of the statement is directly obtained by means of the Transfer Theorem for singularity analysis (see Theorem~\ref{theo:transfer} in Appendix~\ref{app:methodology}).

Let us now analyze the bivariate situation.
Let $0 < y < 1$.
Observe that $\gf[A][][y \, f(x)]$ is singular at $x=\rho$.
The singular expansion of $\gf[A][][z]$ around $z = \frac{4y}{27}$ is
\[
\gf[A][][z] \stackbin[z \sim \frac{4y}{27}]{}{=} \gf[A][][\frac{4y}{27}] + \gf[A]['][\frac{4y}{27}] \left(z - \frac{4y}{27}\right)+O\left(\left(z - \frac{4y}{27}\right)^2\right).
\]
By means of similar arguments to the asymptotic estimate calculus we can obtain the singular expansion of $y\gf[A][][y \, f(x)] \, g(x)$ around $x=\rho$
\begin{align*}
y \, \gf[A][][{y \, f(x)}] \, g(x) \stackbin[x \sim \rho]{}{=} & y \, \gf[A][][\frac{4y}{27} - \frac{8\sqrt{3}}{27}y \, X \,  + O\left(X^2\right)] \left( \frac{\sqrt{3}}{3} \, X^{-1} \,+O(1) \right) \\
\stackbin[x \sim \rho]{}{=} & y \left( \gf[A][][\frac{4y}{27}]- y \, \gf[A]['][\frac{4y}{27}] \frac{8\sqrt{3}}{27} \, X \,+O\left(X^2\right)\right)  \left( \frac{\sqrt{3}}{3} \, X^{-1}\,+O(1) \right) .
\end{align*}
Finally,
\[
p(y) = \lim_{n \to \infty}\frac{\coeff{x^n}{y \, \gf[A][][{yg(x)}] \, f(x)}}{\coeff{x^n}{\gf[A][][{g(x)}] \, f(x)}} = y\frac{\gf[A][][\frac{4y}{27}]}{\gf[A][][\frac{4}{27}]} = \frac{529 \, y^4 \, (9-y)}{8 \, (27-4y)^2}
\]
This limit implies convergence to the discrete limit law with this probability generating function, and the proposition is established.
\end{proof}

Observe that the first terms of $p(y)$ are
\[
0.816358 \, y^4 + 0.151177 \, y^5 + 0.026876 \, y^6 + 0.004645 \, y^7 + 0.000786 \, y^8 + 0.000131 \, y^9 + O\left(y^{10}\right)
\]
which shows that the size of the core is, with high probability, around 4.

\medskip
The analysis of $k$ crossing trees for an arbitrary value of $k$, seems more involved.
In particular, the analysis of $1$ crossing trees shows that the families of cores for general $k$ is not finite.
Furthermore, the structure of chord crossings is not enough to characterize the cores, as we need to obtain a connected structure without cycles.
The analysis of $1$ crossing trees suggests that the generating function associated to cores of $k$ crossing trees must be a rational function, and that the expected value of the size of the core is constant.
We do not include here the detailed analysis, which will be the source of future research.


\subsection{Configurations on surfaces of higher genus}
\label{subsec:higerGenus}

To finish, we discuss a natural extension of the problems studied in this paper to surfaces of higher genus with boundaries.
As discussed in \cite{RueSauThilikos}, defining the analogues of planar configurations in higher genus surfaces is already involved.
In the sphere, each combinatorial chord configuration has essentially a unique graphical representation.
This is not true in higher genus surfaces: a combinatorial configuration can be designed in the surface (up to homotopy of the chords) in several ways.
For this reason the definition of crossings must be stated in a topological way, and not only by combinatorial means as in the case of the sphere.

Given a surface~$\bS$ with boundaries, a \defn{configuration} on~$\bS$ is a set of curves embedded in~$\bS$ whose endpoints lie on the boundaries of~$\bS$.
Configurations are considered up to homotopy~of~$\bS$:
\begin{enumerate}[(i)]
\item we do not allow two homotopic curves in the same configuration,
\item the multiplicity of the crossing between two curves is the minimum number of crossings between two curves homotopic to them,
\item two configurations are considered as equivalent if there is a boundary preserving homotopy of~$\bS$ sending one configuration to the other.
\end{enumerate}
An example of a core configuration on a surface with $2$ boundaries and genus~$3$ is represented in \fref{fig:configurationSurface} (all chords are involved in at least one crossing).

\begin{figure}[h]
  \centerline{\includegraphics[scale=.7]{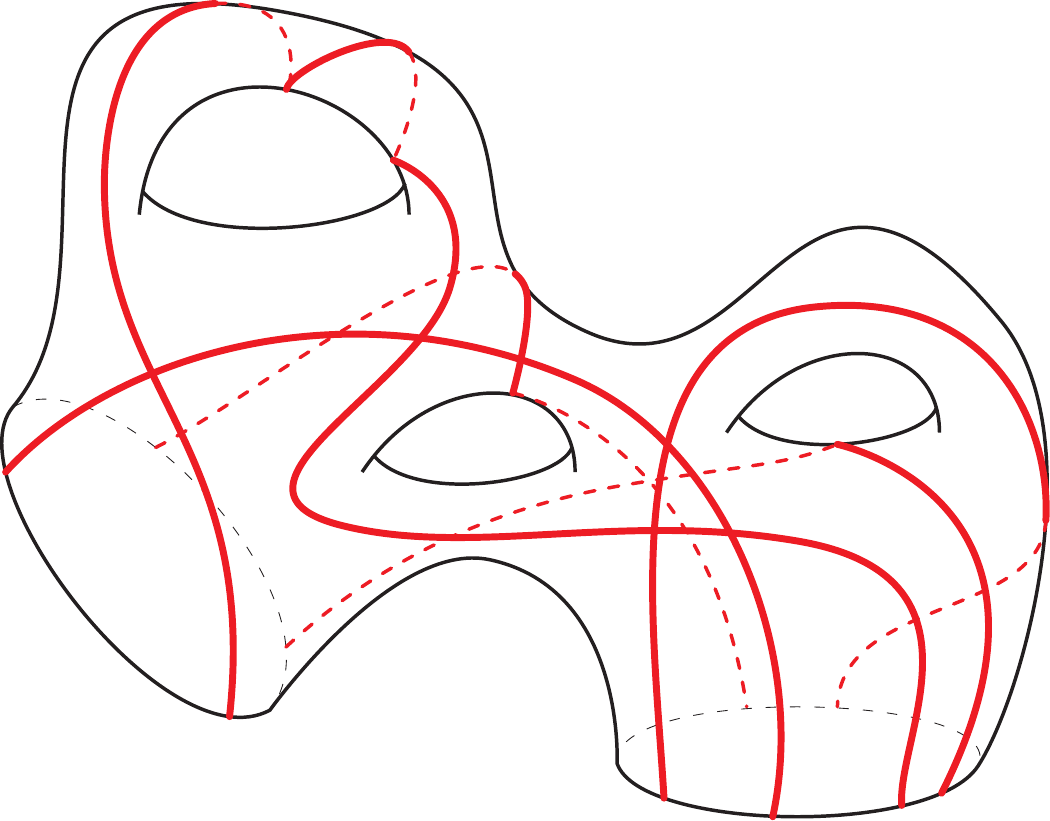}}
  \caption{A core configuration with $6$ crossings on a surface with $2$ boundaries and genus~$3$.}
  \label{fig:configurationSurface}
\end{figure}

Observe that even with this definition, there is still infinitely many configurations with given numbers of vertices and curves.
Indeed, as soon as the surface is non-trivial, we can operate Dehn twists on any handle or boundary component of the surface to create infinitely many non homotopic configurations from a given configuration (see~\cite{Ivanov} for precise definitions on geometric topology).
To fix this problem for the enumeration, a convenient method is to choose a generating set of the homotopy group of~$G$ (\ie a collection of simple closed curves on the surface meeting at a base point whose complement in the surface is a collection of disks), and to enumerate configurations according to their number of crossings with this set of curves.

In this situation, we can still decompose the configurations with $k$ crossings according to their $k$-cores.
As illustrated in \fref{fig:decompositionSurface}, this decomposition results in many crossing-free configurations on smaller surfaces, given by the connected components of the complement of the $k$-core in~$\bS$.
\begin{figure}
  \centerline{\includegraphics[scale=.5]{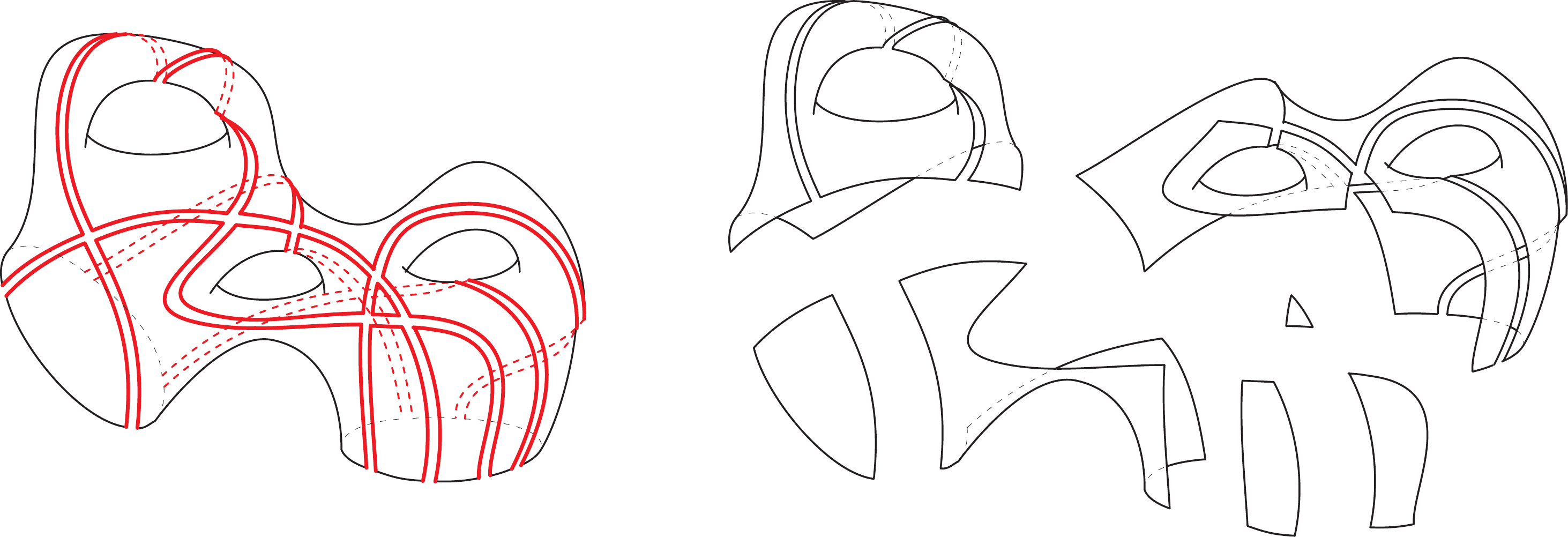}}
  \caption{Cutting along the chords of the core configuration of \fref{fig:configurationSurface} (left) decomposes the surface on smaller pieces (right). Each piece should be filled in with a crossing-free configuration.}
  \label{fig:decompositionSurface}
\end{figure}
Also the decomposition scheme is valid, we still encounter here an essential problem to derive enumerative formulas from this decomposition. Namely, we do not know how to count crossing-free configurations on surfaces with high genius or more than one boundary. If the complement of the crossing-free configuration is a collection of disks, and if the surface has only one boundary component, then we can retract this component to a point and the configuration becomes a  map on~$\bS$ with a single vertex. By duality, this is equivalent to a unicellular map on~$\bS$. These maps have been enumerated in~\cite{Chapuy} (see also~\cite{BernardiChapuy}). More precisely, \cite[Theorem 2]{Chapuy} asserts that the numbers $\varepsilon_g(n)$ of unicellular maps with $n$ vertices on an orientable closed surface of genus $g$ satisfy the recurrence relation
\[
2g \varepsilon_{g}(n) = \binom{n+1-2(g-1)}{3} \varepsilon_{g-1}(n) + \binom{n+1-2(g-2)}{5} \varepsilon_{g-2}(n) + \dots + \binom{n+1}{2g+1} \varepsilon_{0}(n).
\]
However, we cannot go beyond unicellular map yet.
Due to these limitations, we believe that the study of perfect matchings with $k$ crossings on surfaces of higher genus would be the first step towards the comprehension of these families of chord configurations.


\section*{Acknowlegdments}

This paper started while the first author visited ICMAT in Madrid in May~2012 and continued during a visit of the second author at Laboratoire d'Informatique, \'Ecole Polytechnique in Paris in February~2013. We thank these institutions for their hospitality. The second author also thanks the support and hospitality of Charles University in Prague, where the final stage of this work was achieved.


\bibliographystyle{alpha}
\bibliography{PilaudRueQuasiPlanarFamilies}

\newcommand{\etalchar}[1]{$^{#1}$}
\begin{thebibliography}{CDD{\etalchar{+}}07}

\bibitem[BC11]{BernardiChapuy}
Olivier Bernardi and Guillaume Chapuy.
\newblock Counting unicellular maps on non-orientable surfaces.
\newblock {\em Adv. in Appl. Math.}, 47(2):259--275, 2011.

\bibitem[B{\'o}n99]{Bona}
Mikl{\'o}s B{\'o}na.
\newblock Partitions with {$k$} crossings.
\newblock {\em Ramanujan J.}, 3(2):215--220, 1999.

\bibitem[BR12]{BernardiRue}
Olivier Bernardi and Juanjo Ru{\'e}.
\newblock Enumerating simplicial decompositions of surfaces with boundaries.
\newblock {\em European J. Combin.}, 33(3):302--325, 2012.

\bibitem[CDD{\etalchar{+}}07]{ChenDengDuStanleyYan}
William Y.~C. Chen, Eva Y.~P. Deng, Rosena R.~X. Du, Richard~P. Stanley, and
  Catherine~H. Yan.
\newblock Crossings and nestings of matchings and partitions.
\newblock {\em Trans. Amer. Math. Soc.}, 359(4):1555--1575, 2007.

\bibitem[Cha11]{Chapuy}
Guillaume Chapuy.
\newblock A new combinatorial identity for unicellular maps, via a direct
  bijective approach.
\newblock {\em Adv. in Appl. Math.}, 47(4):874--893, 2011.

\bibitem[CMS09]{ChapuyMarcusSchaeffer}
Guillaume Chapuy, Michel Marcus, and Gilles Schaeffer.
\newblock A bijection for rooted maps on orientable surfaces.
\newblock {\em SIAM J. Discrete Math.}, 23(3):1587--1611, 2009.

\bibitem[CP92]{CapoyleasPach}
Vasilis Capoyleas and J{\'a}nos Pach.
\newblock A {T}ur\'an-type theorem on chords of a convex polygon.
\newblock {\em J. Combin. Theory Ser. B}, 56(1):9--15, 1992.

\bibitem[DFLS04]{DuchonFlajoletLouchardSchaeffer}
Philippe Duchon, Philippe Flajolet, Guy Louchard, and Gilles Schaeffer.
\newblock Boltzmann samplers for the random generation of combinatorial
  structures.
\newblock {\em Combin. Probab. Comput.}, 13(4-5):577--625, 2004.

\bibitem[FN99]{FlajoletNoy-nonCrossing}
Philippe Flajolet and Marc Noy.
\newblock Analytic combinatorics of non-crossing configurations.
\newblock {\em Discrete Math.}, 204(1-3):203--229, 1999.

\bibitem[FN00]{FlajoletNoy-chordDiagrams}
Philippe Flajolet and Marc Noy.
\newblock Analytic combinatorics of chord diagrams.
\newblock In {\em Formal power series and algebraic combinatorics ({M}oscow,
  2000)}, pages 191--201. Springer, Berlin, 2000.

\bibitem[FO90]{FlajoletOdlyzko}
Philippe Flajolet and Andrew Odlyzko.
\newblock Singularity analysis of generating functions.
\newblock {\em SIAM J. Discrete Math.}, 3(2):216--240, 1990.

\bibitem[FS09]{FlajoletSedgewick}
Philippe Flajolet and Robert Sedgewick.
\newblock {\em Analytic combinatorics}.
\newblock Cambridge University Press, Cambridge, 2009.

\bibitem[Fus09]{Fusy}
{\'E}ric Fusy.
\newblock Uniform random sampling of planar graphs in linear time.
\newblock {\em Random Structures Algorithms}, 35(4):464--522, 2009.

\bibitem[Hwa98]{Hwang}
Hsien-Kuei Hwang.
\newblock On convergence rates in the central limit theorems for combinatorial
  structures.
\newblock {\em European J. Combin.}, 19(3):329--343, 1998.

\bibitem[Iva02]{Ivanov}
Nikolai~V. Ivanov.
\newblock Mapping class groups.
\newblock In {\em Handbook of geometric topology}, pages 523--633.
  North-Holland, Amsterdam, 2002.

\bibitem[Jon03]{Jonsson}
Jakob Jonsson.
\newblock Generalized triangulations of the $n$-gon.
\newblock Unpublished manuscript. An abstract was included in
  \emph{``Topological and Geometric Combinatorics, April~2003''},
  Mathematisches Forschungsinstitut Oberwolfach, Report No.~16, 2003.

\bibitem[Kla00]{Klazar-nonCrossingHypergraphs}
Martin Klazar.
\newblock Counting pattern-free set partitions. {II}. {N}oncrossing and other
  hypergraphs.
\newblock {\em Electron. J. Combin.}, 7:Research Paper 34, 25, 2000.

\bibitem[Kla03]{Klazar-connectedMatchings}
Martin Klazar.
\newblock Non-{$P$}-recursiveness of numbers of matchings or linear chord
  diagrams with many crossings.
\newblock {\em Adv. in Appl. Math.}, 30(1-2):126--136, 2003.
\newblock Formal power series and algebraic combinatorics (Scottsdale, AZ,
  2001).

\bibitem[MM89]{MeirMoon}
Amram Meir and John~W.; Moon.
\newblock On an asymptotic method in enumeration.
\newblock {\em J. Combin. Theory Ser. A}, 51(1):77--89, 1989.

\bibitem[Nak00]{Nakamigawa}
Tomoki Nakamigawa.
\newblock A generalization of diagonal flips in a convex polygon.
\newblock {\em Theoret. Comput. Sci.}, 235(2):271--282, 2000.
\newblock Combinatorics and optimization (Okinawa, 1996).

\bibitem[OEI10]{OEIS}
The {O}n-{L}ine {E}ncyclopedia of {I}nteger {S}equences.
\newblock Published electronically at \url{http://oeis.org}, 2010.

\bibitem[PS09]{PilaudSantos}
Vincent Pilaud and Francisco Santos.
\newblock Multitriangulations as complexes of star polygons.
\newblock {\em Discrete Comput. Geom.}, 41(2):284--317, 2009.

\bibitem[Rio75]{Riordan}
John Riordan.
\newblock The distribution of crossings of chords joining pairs of {$2n$}
  points on a circle.
\newblock {\em Math. Comp.}, 29:215--222, 1975.
\newblock Collection of articles dedicated to Derrick Henry Lehmer on the
  occasion of his seventieth birthday.

\bibitem[RST13]{RueSauThilikos}
Juanjo Ru{\'e}, Ignasi Sau, and Dimitrios~M. Thilikos.
\newblock Asymptotic enumeration of non-crossing partitions on surfaces.
\newblock {\em Discrete Math.}, 313(5):635--649, 2013.

\bibitem[SS12]{SerranoStump}
Luis Serrano and Christian Stump.
\newblock Maximal fillings of moon polyominoes, simplicial complexes, and
  {S}chubert polynomials.
\newblock {\em Electron. J. Combin.}, 19(1):Paper 16, 18, 2012.

\bibitem[Tou52]{Touchard}
Jacques Touchard.
\newblock Sur un probl\`eme de configurations et sur les fractions continues.
\newblock {\em Canadian J. Math.}, 4:2--25, 1952.

\bibitem[Wri77]{WrightI}
Edward~M. Wright.
\newblock The number of connected sparsely edged graphs.
\newblock {\em J. Graph Theory}, 1(4):317--330, 1977.

\bibitem[Wri78]{WrightII}
Edward~M. Wright.
\newblock The number of connected sparsely edged graphs. {II}. {S}mooth graphs
  and blocks.
\newblock {\em J. Graph Theory}, 2(4):299--305, 1978.

\bibitem[Wri80]{WrightIII}
Edward~M. Wright.
\newblock The number of connected sparsely edged graphs. {III}. {A}symptotic
  results.
\newblock {\em J. Graph Theory}, 4(4):393--407, 1980.

\bibitem[Wri83]{WrightIV}
Edward~M. Wright.
\newblock The number of connected sparsely edged graphs. {IV}. {L}arge
  nonseparable graphs.
\newblock {\em J. Graph Theory}, 7(2):219--229, 1983.

\end{thebibliography}


\appendix

\section{Asymptotic analysis and limit laws}
\label{app:methodology}

Throughout this paper, we have used language and basic results of \defn{Analytic Combinatorics}.
We refer to the book of P.~Flajolet and R.~Sedgewick~\cite{FlajoletSedgewick} for a detailed presentation of this area.
For the convenience of the reader, we only recall in this appendix the main tools used in this paper.


\subsection{Transfer theorems for singularity analysis}
Consider a combinatorial class~$(\F[A],|\cdot|)$, and let~$a_n \eqdef |\set{A \in \F[A]}{\,|A| = n}|$ denote the number of objects of~$\F[A]$ of size~$n$, and ${\gf[A][][x]  = \sum_{n\ge 0} a_n x^n}$ denote the associated generating function.
We denote by~$\coeff{x^n}{\gf[A][][x] } = a_n$ the coefficient of~$x^n$ in~$\gf[A][][x] $.
By means of analytic methods we can obtain asymptotic estimates for $\coeff{x^n}{\gf[A][][x] }$ in terms of the smallest singularities of~$\gf[A][][x] $.
Since $\gf[A][][x] $ has non-negative coefficients, its smallest singularity (if it exists) is a positive real number by Pringsheim's Theorem, see~\cite[Theorem IV.6]{FlajoletSedgewick}.
Let~$\rho$ be this singularity.
We write $\gf[A][][x]  \stackbin[x \sim \rho]{}{=} \gf[B][][x]$ if $\lim_{x \to \rho} \gf[A][][x] / \gf[B][][x] = 1$.

With this language, we obtain the asymptotic expansion of $\coeff{x^n}{\gf[A][][x] }$ by transfering the behaviour of $\gf[A][][x] $ around its smallest singularity~$\rho$ from a simpler function~$\gf[B][][x]$ for which we know the asymptotic behaviour of the coefficients.
The first result in this direction is the \defn{Transfer Theorem for singularity analysis}~\cite{FlajoletOdlyzko,FlajoletSedgewick}.

\begin{theorem}[Transfer Theorem for a single singularity~\cite{FlajoletOdlyzko}, simplified version]
\label{theo:transfer}
Assume that the generating function $\gf[A][][x] $ is analytic in a dented domain $\Delta = \Delta(\phi,R)$, defined as
\[
\Delta(\phi,R)= \bigset{x}{x \neq 1,\, |x|<R,\, |\mathrm{Arg}(x-\rho)|>\phi},
\]
for some $R>1$ and~$0<\phi<\pi/2$.
If
\[
\gf[A][][x] \stackbin[x \sim \rho]{}{=} \left(1 - \frac{x}{\rho} \right)^{-\alpha}+ O\left(\left(1-\frac{x}{\rho}\right)^{-\alpha+1}\right)
\]

when $x$ tends to~$\rho$ in the domain~$\Delta$, and $\alpha \notin \{0,-1,-2,\dots\}$ then
\[
\coeff{x^n}{\gf[A][][x]} \stackbin[n \to \infty]{}{=} \frac{1}{\Gamma(\alpha)} \, n^{\alpha-1} \, \rho^{-n} \, (1+o(1)),
\]
where~$\Gamma(x) \eqdef \int_0^\infty t^{x-1} e^{-t} dt$ denotes the classical Gamma function.
\end{theorem}

When the function~$\gf[A][][x] $ has several singularities on its circle of convergence, the contributions of all singularities add up together to the asymptotic of the coefficients of~$\gf[A][][x] $.
In this paper, we essentially encounter the situation when the set of smallest singularities of~$\gf[A][][x] $ is of the form~$\set{\xi \cdot \rho}{\xi \in \C, \, \xi^p = 1}$ for some~$\rho \in \R$ and some~$p \ge 1$.
In order we deal with this situation, we need a refined version of Theorem~\ref{theo:transfer} as it is stated in~\cite[Theorem VI.5]{FlajoletSedgewick}.
We include below a simplified version of this theorem.

\begin{theorem}[Transfer Theorem for multiple singularities~\cite{FlajoletOdlyzko}, simplified version]
\label{theo:multiple-singularities}
Assume that the generating function $\gf[A][][x] $ is analytic for $|x|< \rho$ and have a finite number of singularities on the circle $|x|=\rho$ at points $\rho_j = \rho e^{i\theta_j}$ for $j \in [r]$.
Assume that there exists a dented domain $\Delta_0$ such that $\gf[A][][x] $ is analytic in the indented disc
\[
\Delta = \bigcap_{j \in [r]} (\rho_j \cdot \Delta_0),
\]
where $\rho_j \cdot \Delta_0$ is the image of $\Delta_0$ by the mapping $z \mapsto \rho_j z$.
Assume that
\[
\gf[A][][x]  \stackbin[x \sim \rho_j]{}{=} \left(1 - \frac{x}{\rho_j} \right)^{-\alpha_j} + O\left(\left(1-\frac{x}{\rho_j}\right)^{-\alpha_j+1}\right).
\]
when $x$ tends to~$\rho_j$ in the domain~$\Delta$, for some $\alpha_j \notin \{0,-1,-2,\dots\}$.
Then the contributions of all singularities of~$\gf[A][][x]$ add up to the asymptotic its coefficients:
\[
\coeff{x^n}{\gf[A][][x]} \stackbin[n \to \infty]{}{=} \, \sum_{j \in [r]} \frac{1}{\Gamma(\alpha_j)} \, n^{\alpha_j-1} \, \rho_j^{-n} \, (1+o(1)).
\]
In particular, if~$\gf[A][][x]$ has $r$ singularities of module~$\rho$ and if~$\gf[A][][xe^{\frac{2i\pi}{r}}] = \gf[A][][x]$, then~$\theta_j = \frac{2j\pi}{r}$ and $\gf[A][][x]$ has the same type of singularity~$\alpha = \alpha_1 = \dots = \alpha_r$ at all its singularities.
Therefore,
\[
\coeff{x^n}{\gf[A][][x]} \stackbin[\substack{n \to \infty \\ r | n}]{}{=} \frac{r}{\Gamma(\alpha)} \, n^{\alpha-1} \, \rho^{-n} \, (1+o(1)).
\]
for $n$ multiple of~$r$, while~$\coeff{x^n}{\gf[A][][x]} = 0$ for all other values of~$n$.
\end{theorem}


\subsection{Functional specifications for generating functions}

By means of analytic methods, it is frequent to obtain functional equations on the function~$\gf[A][][x] $ from which we want to extract asymptotic estimates of the coefficients.
Different inversion techniques can be useful for that purpose, in particular Lagrangian inversion when we obtain a functional equation of the form
\[
\gf[A][][x]  = x \Psi(\gf[A][][x] ),
\]
for a smooth function~$\Psi$ around the origin with a non-zero derivative at the origin.
In this paper, we are sometimes not in a Lagrangian situation, and we need sharper tools to obtain asymptotic estimates.

\begin{definition}[Smooth implicit-function schema \cite{MeirMoon}]
\label{def:SmoothImplicitFunctionSchema}
Let $\gf[A][][x]  = \sum_{n \in \N} a_n x^n$ be an analytic function at~$0$ such that~$a_0 = 0$ and~$a_n \ge 0$ for all~$n \ge 1$.
This function~$\gf[A][][x]$ belongs to a \defn{smooth implicit-function schema} if
\[
\gf[A][][x]  = \gf[G][][{x, \gf[A][][x]}],
\]
for a bivariate function~$\gf[G][][x,y]$ satisfying the following properties:
\begin{enumerate}[(i)]
\item
\label{item:SmoothImplicitFunctionSchema1}
$\gf[G][][x,y]$ is analytic in a domain~$|x| < U$ and~$|y| < V$,
\item
\label{item:SmoothImplicitFunctionSchema2}
The coefficients of~$\gf[G][][x,y] = \sum_{n,m} g_{n,m} \, x^n y^m$ satisfy $g_{0,0} = 0$, $g_{0,1} \neq 1$, and $g_{n,m} \ge 0$ for all~$n,m \in \N$, with at least one strict inequality for some~$n$ and some~$m \ge 2$,
\item
\label{item:SmoothImplicitFunctionSchema3}
There exists real numbers~$0 < u < U$ and~$0 < v < V$ which satisfy the \defn{characteristic system}
\[
\gf[G][][u,v] = v \qquad\text{and}\qquad \gf[G][_y][u,v] = 1,
\]
where~$\gf[G][_y][u,v]$ denotes the derivative of~$\gf[G][][x,y]$ with respect to~$y$ evaluated at~$(x,y) = (u,v)$.
\end{enumerate}
\end{definition}

Under these assumptions, the following theorem due to A.~Meir and~J.~Moon~\cite{MeirMoon} ensures that the function~$\gf[A][][x] $ presents a square-root singularity, from which we can extract asymptotic estimates of the coefficients by the Transfer Theorem.

\begin{theorem}[Smooth implicit-function theorem~\cite{MeirMoon}]
\label{theo:SmoothImplicitFunctionTheorem}
If the function~$\gf[A][][x] $ belongs to the smooth implicit-function schema of Definition~\ref{def:SmoothImplicitFunctionSchema}, then $\gf[A][][x] $ converges at~$x \sim u$ and has a square-root singularity:
\[
\gf[A][][x]  \stackbin[x \sim u]{}{=} v - \sqrt{\frac{2u \, \gf[G][_x][u,v]}{\gf[G][_{yy}][u,v]}} \, \sqrt{1 - \frac{x}{u}} + O\left( 1 - \frac{x}{u} \right),
\]
where~$\gf[G][_x][x,y]$ and~$\gf[G][_{yy}][x,y]$ denote respectively the derivative of~$\gf[G][][x,y]$ with respect to~$x$, and the second derivative of~$\gf[G][][x,y]$ with respect to~$y$, evaluated at~$(x,y) = (u,v)$.
\end{theorem}


\subsection{Limiting distributions}

The framework of Analytic Combinatorics is also powerful to handle probabilities in a combinatorial class.
Let $\chi:\F[A]  \to \mathbb{N}$ be a parameter on the combinatorial class~$\F[A]$.
For~$n,m \in \N$, let~$a_{n,m} \eqdef |\set{A \in \F[A]}{|A|=n, \, \chi(A) = m}|$ denote the number of objects of~$\F[A]$ of size~$n$ and parameter~$m$.
Define the bivariate generating function
\[
\gf[A][][x,y] = \sum_{n,m \in \N} a_{n,m} \, x^n \, y^m,
\]
where $y$ marks the parameter $\chi$.
In particular $\gf[A][][x,1] = \gf[A][][x] $, and $\sum_{m \in \N} a_{n,m} = a_n$.
For every value of $n$, the parameter $\chi$ defines a random variable $\nX_{n}$ over the elements of $\F[A] $ of size $n$ with discrete probability density function $\PP\left(\nX_n=m\right)=a_{m,n}/{a_n}$.
This discrete probability distribution can be encapsulated into the discrete law
\[
p_n(y) = \frac{\coeff{x^n}{\gf[A][][x,y]}}{\coeff{x^n}{\gf[A][][x,1]}}.
\]

The \defn{Quasi-powers Theorem}~\cite{Hwang} provides a direct way to deduce normal limit laws from singular expansions of generating functions.
In other words, the Quasi-Powers Theorem gives sufficient conditions to assure normal limit laws in the context of Analytic Combinatorics.
It can be applied when the behavior of the dominant singularity remains unchanged when~$y$ varies around~$1$.

\begin{theorem}[Quasi-Powers Theorem~\cite{Hwang}]
\label{theo:quasi-powers}
Let $\gf[F][][x,y]$ be a bivariate analytic function on a neighborhood of $(0,0)$, with non-negative coefficients.
Assume that the function $\gf[F][][x,y]$ admits, in a region
\[
\mathcal{R}=\{|y-1|<\varepsilon\}\times\{|x|\leq r\}
\]
for some $r,\varepsilon > 0$, a representation of the form
\[
\gf[F][][x,y] = \gf[A][][x,y] + \gf[B][][x,y] \, \gf[C][][x,y]^{-\alpha},
\]
where $\gf[A][][x,y]$, $\gf[B][][x,y]$ and $\gf[C][][x,y]$ are analytic in $\mathcal{R}$, and such that
\begin{itemize}
\item $\gf[C][][x,y] = 0$ has a unique simple root $\rho<r$ in $|x|\leq r$,
\item $\gf[B][][\rho,y] \neq 0$,
\item neither $\partial_x \gf[C][][\rho,y]$ nor $\partial_y \gf[C][][\rho,y]$ vanish, so there exists a non-constant function $\rho(y)$ analytic at $y = 1$ such that $\rho(1)=\rho$ and $\gf[C][][\rho(y),y] = 0$,
\item finally
\[
\sigma^2 = -\frac{\rho''(1)}{\rho(1)}-\frac{\rho'(1)}{\rho(1)}+\left(\frac{\rho'(1)}{\rho(1)}\right)^2
\]
is different from $0$.
\end{itemize}
Then the random variable with density probability function
\[
p_n(y)=\frac{\coeff{x^n}{F(x,y)}}{\coeff{x^n}{F(x,1)}}
\]
converge in distribution to a Normal distribution.
The corresponding expectation $\mu_n$ and variance $\sigma_n^2$ converge asymptotically to $-\frac{\rho'(1)}{\rho(1)} \, n$ and~$\sigma^2 n$, respectively.
\end{theorem}

\end{document}